\documentclass[11pt]{amsart} 
\usepackage{amsthm,amsbsy,amsfonts,amssymb,amscd}

\usepackage[mathscr]{euscript} 

\usepackage{enumerate} 
\usepackage{tikz} 
\usetikzlibrary{calc,intersections,through, backgrounds,arrows,decorations.pathmorphing,backgrounds,positioning,fit,petri} 
\usetikzlibrary{calc}

\usepackage{tikz-cd} 
\usepackage[all]{xy}
\xyoption{knot}
\xyoption{poly}

\usepackage{hyperref}

\usepackage[margin=1in]{geometry}

\newcommand{\LL}{{\mathbf{L}}}

\newcommand{\cF}{{\mathcal{F}}}

\newcommand{\bT}{\mathbf{T}}
\newcommand{\Par}{{\mathscr{P}}}
\newcommand{\gd}{\mathrm{gd}}
\newcommand{\dd}{\mathrm{dd}}

\newcommand{\tto}{\twoheadrightarrow}

\newcommand{\cB}{\mathcal B}

\newcommand{\cR}{\mathcal{R}}
\newcommand{\plus}{\scalebox{0.6}{{\rm+}}}

\newcommand{\RR}{\mathscr{R}}
\newcommand{\AAAA}{\mathscr{A}}

\newcommand{\id}{\mathrm{id}}
\newcommand{\Id}{\mathrm{Id}}

\newcommand{\Hom}{\mathrm{Hom}} 
\newcommand{\Ext}{\mathrm{Ext}}

\newcommand{\minus}{\scalebox{0.9}{{\rm -}}}

\newcommand{\End}{\mathrm{End}}

\newcommand{\dfs}{{/\kern-2pt/}}

\newcommand{\JJJ}{\mathscr{J}}
\newcommand{\op}{{\rm op}}
\newcommand{\mk}{\Bbbk}
\newcommand{\cC}{\mathcal{C}}
\newcommand{\cL}{\mathcal{L}}

\newcommand{\cA}{\mathcal{A}}
\newcommand{\Ob}{{\rm{Ob}}}

\newcommand{\charr}{{\rm{char}}}
\newcommand{\pe}{\mathfrak{pe}}

\newcommand{\Str}{\mathrm{STr}}
\newcommand{\St}{\mathrm{St}}
\newcommand{\fn}{\mathfrak{n}}

\newcommand{\fh}{\mathfrak{h}}

\newcommand{\mN}{\mathbb{N}}

\newcommand{\mZ}{\mathbb{Z}}
\newcommand{\mF}{\mathbb{F}}
\newcommand{\mS}{\mathbb{S}}

\numberwithin{equation}{section}

\newcommand{\con}{{\rm con}}
\newcommand{\cQ}{\mathcal{Q}}

\newcommand{\Rep}{{\rm Rep}}

\newtheoremstyle{notes} {} {} {} {} {\bfseries} {.} {.5em} {}

\theoremstyle{plain}

\newtheorem{prop}[subsubsection]{Proposition}

\newtheorem{lemma}[subsubsection]{Lemma}

\newtheorem{cor}[subsubsection]{Corollary}

\newtheorem{thm}[subsubsection]{Theorem}

\newtheorem{thmA}{Theorem}

\theoremstyle{remark}

\newtheorem{qu}[subsubsection]{Question} 

\newtheorem{rem}[subsubsection]{Remark} 

\newtheorem{ddef}[subsubsection]{Definition} 

\swapnumbers

\usepackage{etoolbox}

\usepackage{ytableau}
\usepackage{youngtab}

\makeatletter
\pretocmd{\appendix}{\addtocontents{toc}{\protect\addvspace{10\p@}}}{}{}
\makeatother

\theoremstyle{remark}

\newtheorem{ex}[subsubsection]{Example}

\newtheoremstyle{construction} {} {} {} {} {\bfseries} { } {0pt} {}

\theoremstyle{construction}

\title[The periplectic Brauer algebra]{The periplectic Brauer algebra}

\author{Kevin~Coulembier}

\newcommand{\ind}{{\rm Ind}}
\newcommand{\res}{{\rm Res}}
\newcommand{\resi}{{\rm res}}


\begin{document} 
\date{} 
\begin{abstract}
We study the periplectic Brauer algebra introduced by Moon in the study of invariant theory for periplectic Lie superalgebras. We determine when the algebra is quasi-hereditary, when it admits a quasi-hereditary 1-cover and, for fields of characteristic zero, describe the block decomposition. To achieve this, we also develop theories of Jucys-Murphy elements, Bratteli diagrams, Murphy bases, obtain a Humphreys-BGG reciprocity relation and determine some decomposition multiplicities of cell modules. As an application, we determine the blocks in the category of finite dimensional integrable modules of the periplectic Lie superalgebra.
	\end{abstract}

\maketitle 

%\tableofcontents

% \pagebreak
\section{Introduction} 

In \cite{Moon}, Moon introduced algebras $A_n$ for all positive integers $n$, bearing resemblance to Brauer algebras. Brauer algebras $B_n(\delta)$ appear naturally in invariant theory of Lie (super)algebras preserving an {\em even} bilinear form, {\it viz.} orthogonal and symplectic Lie algebras and, more generally, orthosymplectic Lie superalgebras, see {\it e.g.}~\cite{Ram, ES, FFT, Vera}. Similarly, the algebra~$A_n$ was introduced to study the invariant theory for a Lie superalgebra preserving an {\em odd} bilinear form, known as the periplectic, or strange or peculiar, Lie superalgebra.

The algebra~$A_n$ hence acts as an odd analogue of~$B_n(0)$. This was made explicit by Kujawa and Tharp in~\cite{Kujawa} and by Serganova in~\cite{Vera}. It is pointed out in~\cite{Kujawa, Vera} that~$A_n$, contrary to~$B_n(\delta)$, fails to be cellular in the sense of \cite{CellAlg} in any obvious way. Acquiring better understanding of the structure of~$A_n$ is essential for the study of the representation theory of the periplectic superalgebra. This is our main motivation to study~$A_n$ and we already apply our results to solve the long-standing open problem of describing linkage in the category of finite dimensional modules over the periplectic superalgebra. 

We study several aspects of the ring-theoretic behaviour and representation theory of the periplectic Brauer algebra~$A_n$. Despite the resemblance to other diagram algebras such as Iwahori-Hecke, Brauer, Temperley-Lieb and BMW algebras, there are some peculiarities of the periplectic Brauer algebra which complicate its study. As already mentioned, $A_n$ is not cellular. Furthermore we demonstrate that, as can be expected through its connection with periplectic Lie superalgebras, the centre is very small. Finally, the periplectic Brauer algebra is not part of a known family of algebras which are generically semisimple. Alternatively, we have no logical realisation of~$A_n$ as the specialisation of a semisimple algebra over some commutative ring.

Now we review the main results on~$A_n$ which will be obtained in this paper. Consider an algebraically closed field $\mk$ of characteristic~$p\ge 0$.
 We use the sets
$$\JJJ^0(n):=\{n-2i\,|\, 0\le i< n/2\}\qquad \mbox{and}\qquad \JJJ(n):=\{n-2i\,|\,0\le i\le n/2\}.$$
\begin{thmA}[Block decomposition]
The simple modules of~$A_n$ are labelled by the set $\Lambda_A$ of~$p$-restricted partitions of~$i$, for all $i\in\JJJ^0(n)$.
Assume $p\not\in[2,n]$, the simple modules $L(\lambda)$ and $L(\mu)$ belong to the same block if and only if~$\lambda$ and $\mu$ have the same $2$-core. An equivalent condition is that the number of even minus the number of odd contents of~$\lambda$ equals that of~$\mu$.
\end{thmA}
The labelling of simple modules by~$\Lambda_A$ is originally due to Kujawa and Tharp, see~\cite[Theorem~4.3.1]{Kujawa}, but we provide an alternative proof. 

\begin{thmA}[Quasi-heredity]
If~$n$ is odd and $p\not\in[2,n]$, the algebra~$(A_n,\le)$ is quasi-hereditary, in the sense of \cite{CPS}. Here, $\le$ is the partial order on~$\Lambda_A$ determined by~$\mu<\lambda$ if and only if~$|\lambda|<|\mu|$.
\end{thmA}
The algebra~$A_2$ is hereditary and hence quasi-hereditary for any partial order. However, this is expected to be an exception, as we show that~$A_4$ is already of infinite global dimension (in fact, even the injective dimension of the left regular module is infinite).

Even though the periplectic Brauer algebras are not cellular, we prove they have an interesting standardly based structure, which is a weaker notion introduced in~\cite{JieDu}. In particular, the cell modules almost always form a standard system, in analogy with the result of Kleshchev and Nakano in~\cite{Nakano} for symmetric groups and of Hartmann and Paget in~\cite{Paget} for Brauer algebras.
\begin{thmA}[Cell modules]
The algebra~$A_n$ is standardly based for~$\mathbf{L}_A:=\sqcup_{i\in\JJJ(n)}\{\lambda\,|\,\lambda\vdash i\}$. The cell modules $\{W(\lambda)\,|\,\lambda\in\mathbf{L}_A\}$ form a standard system if and only if~$p\not\in\{2,3\}$ and $n\not\in\{2,4\}$.

Assume that~$p\not\in[2,n]$, the multiplicities in the cell filtration of the projective covers (these multiplicities do not depend on the chosen filtration if $n\not\in\{2,4\}$) satisfy
$$(P(\mu):W(\lambda))\;=\; [W(\lambda^t ):L(\mu^t )],\qquad\mbox{for }\; \mu\in \Lambda_A\quad\mbox{and}\quad \lambda\in \mathbf{L}_A.$$
where~$\lambda^t $ and $\mu^t $ denote the transpose partitions.
\end{thmA}
By \cite[Lemma~5.3.1]{Borelic}, the statement that the cell modules constitute a standard system implies that~$A_n$ admits a quasi-hereditary 1-cover, in the sense of \cite{Rou}. We construct this cover if $p\not\in[2,n]$.

Theorems 1, 2 and 3 and other results in the paper lift the knowledge of the periplectic Brauer algebras, concerning quasi-heredity, standardly based structures, block decomposition, Jucys-Murphy elements and Bartelli diagrams to the same level as that of the ordinary Brauer algebra. For $B_n(\delta)$, the corresponding topics were investigated in~\cite{blocks, Enyang, Paget, CellAlg, LeducRam, Koenig, Naz}.

The final main result is an application to the representation theory of the periplectic Lie superalgebra, which completes a recent partial result by Chen in~\cite{Chen}.
\begin{thmA}\label{ThmBlocks}
The category of integrable modules over $\mathfrak{pe}(m)$ contains precisely $m+1$ blocks.
\end{thmA}
We also find an explicit description of which simple modules are contained in each block. This the last classical Lie superalgebras for which the block decompositions were unknown. The periplectic case is complicated by the vanishing of the centre of the universal enveloping algebra, see~\cite{Go}. This is also the only case for which there are only finitely many blocks. The result in Theorem~\ref{ThmBlocks} has been obtained independently at the same time in \cite{gang}.

The paper is organised as follows. In Section~\ref{SecPrel} we introduce the periplectic Brauer algebra and category following \cite{Kujawa} and fix some conventions for symmetric groups. In Section~\ref{SecCover} we use the periplectic Brauer category to introduce an algebra $C_n$  which we can easily study using the general theory of \cite{Borelic}. In Section~\ref{SecAC} we show that $C_n$ is closely related to the periplectic Brauer algebra, by a double centraliser property, which allows to obtain the results in Theorems~2 and 3 above. In Section~\ref{SecBasis} we study the iterative restriction of cell modules from $A_{n}$ to~$A_{n-1}$, leading to a Bratteli diagram. Then we define the Murphy basis of the cell modules in terms of this Bratteli diagram. In Section~\ref{SecJM} we introduce mutually commuting elements in~$A_n$, which we call the Jucys-Murphy elements. We use these to study the centre of~$A_n$ and show that their action on the Murphy basis adheres to the general theory of families of JM elements developed in~\cite{MathasCrelle}. This allows to prove one direction of the claim in Theorem~1. In Section~\ref{SecDM} we determine some decomposition multiplicities which allow to conclude the proof of Theorem~1. In Section~\ref{Secpm} we determine the blocks in the category of intgrable modules over the periplectic Lie superalgebra, yielding Theorem~4. We also give some interpretations of other results in the paper. Finally, in Section~\ref{SecEx} determine all decomposition multiplicities for $A_n$ with~$n\le 5$ and the path algebra description for $n<5$. We briefly investigate Koszul and Ringel duality for these examples.  
 
\section{Preliminaries}\label{SecPrel}

We set $\mN=\{0,1,2,\ldots\}$. Throughout the paper, we let $\mk$ be an arbitrary {\em algebraically closed} field. By an ``algebra'', we mean an algebra over~$\mk$ which is associative, finite dimensional and has an identity element $1$ for multiplication. Unless explicitly stated otherwise, modules are assumed to be finite dimensional, unital and left modules. The category of such modules over~$A$ is denoted by~$A$-mod. The isoclasses of simple modules are labelled by~$\Lambda=\Lambda_A$. For $\lambda\in\Lambda$ we write $L(\lambda)$ or $L_A(\lambda)$ for the corresponding simple module and $P(\lambda)=P_A(\lambda)$ for its projective cover in~$A$-mod. For a $\mZ$-graded object $V$, the degree $j$ part is~$V_{(j)}$. For any~$A$-module $M$, we denote by~$\cF(M)$ the category of modules admitting a filtration where the sections are direct summands of~$M$.

\subsection{The periplectic Brauer category}
In this section, we introduce the {\em periplectic Brauer category} $\cA$, following~\cite{Kujawa}. This is a $\mk$-linear small skeletal category.

\subsubsection{}The periplectic Brauer category $\cA$ is an analogue of the Brauer category $\cB(\delta)$ of \cite{BrCat} at $\delta=0$. In fact, when~$\charr(\mk)=2$, the two coincide. The category $\cA$ has been introduced as $\Rep_0 P$ in~\cite[\S4.5]{Vera}, and independently as the ``marked Brauer category'' $\cB(0,\minus 1)$ in~\cite[\S3]{Kujawa}.

\subsubsection{Objects and morphisms}\label{ObHom} The set of objects of the periplectic Brauer category $\cA$ is $\mN$. The $\mk$-vector space $\Hom_{\cA}(i,j)$ is spanned by all {\em $(i,j)$-Brauer diagrams} of \cite[Definition~2.1]{BrCat}. These diagrams correspond to all partitions of a set of~$i+j$ dots into pairs. They are graphically represented by imagining the $i$ dots on a horizontal line and the $j$ dots on a second horizontal line, above the first one. The Brauer diagram then consists of~$(i+j)/2$ lines, connecting the dots belonging to the same pair. An example of a $(6,8)$-Brauer diagram is
$$
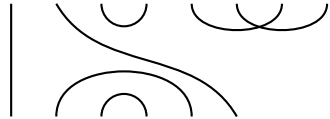
\begin{tikzpicture}[scale=1,thick,>=angle 90]
\begin{scope}[xshift=4cm]
\draw  (2.8,-0.5) -- +(0,1.5);
\draw (4,-0.5) to [out=90, in=180] +(0.3,0.3);
\draw (4.6,-0.5) to [out=90, in=0] +(-0.3,0.3);

\draw (3.4,-0.5) to [out=90, in=180] +(0.9,0.6);
\draw (5.2,-0.5) to [out=90, in=0] +(-0.9,0.6);

\draw (4,1) to [out=-90, in=180] +(0.3,-0.3);
\draw (4.6,1) to [out=-90, in=0] +(-0.3,-0.3);

\draw (5.2,1) to [out=-90, in=-90] +(1.2,0);

\draw (5.8,-0.5) to [out=120, in=-60] +(-2.4,1.5);

\draw (5.8,1) to [out=-90, in=-90] +(1.2,0);

\end{scope}
\end{tikzpicture}
$$ The lines in Brauer diagrams which connect the lower and upper horizontal line will be referred to as {\em propagating lines}. Lines connecting two points on the lower line are called {\em caps} and lines connecting two points on the upper line are {\em cups}. Imaginary horizontal lines in between the two horizontal lines containing the dots are referred to as {\em lines of latitude}.

The composition of two morphisms is the `usual one' in~\cite[Definition~2.3]{BrCat}, up to a possible minus sign. These minus signs originate firstly because the cups and caps are `odd' morphisms, so they only commute up to a minus sign and, secondly, the caps are skew symmetric, see \cite{Kujawa} or Section~\ref{Secpm}. This implies that one will need to keep track of the `relative order' of cups and caps, as well as assign a `direction' to caps. This is made concrete in the next subsection.

\subsubsection{Marked diagrams}\label{SecMaDi}
For a Brauer diagram, we will choose a linear order on the set of its cups and caps. Further, we will declare any cap to be either left-handed or right-handed.

Graphically, this is achieved by {\em marking} the Brauer diagram. Any cup is marked with a diamond~$\Diamond$ and any cap with a left $\lhd$ or right $\rhd$ arrow, where left means against the sense of moving from the left dot to the right along the line. The height of the lines of latitude of these markings is then the graphical realisation of the linear order on the set of cups and caps, where the maximal cup/cap in the order corresponds to the highest marking. This implies that we do not allow two symbols to be on the same line of latitude. The result is a {\em marked Brauer diagram}.

 An example of a marked Brauer diagram, for the Brauer diagram of \ref{ObHom} is
$$
\begin{tikzpicture}[scale=1,thick,>=angle 90]
\begin{scope}[xshift=4cm]
\draw  (2.8,-0.5) -- +(0,1.5);
\draw (4,-0.5) to [out=90, in=180] +(0.3,0.3);
\draw (4.6,-0.5) to [out=90, in=0] +(-0.3,0.3);
\draw (4.3,-0.2) node[]{$\rhd$};

\draw (3.4,-0.5) to [out=90, in=180] +(0.9,0.7);
\draw (5.2,-0.5) to [out=90, in=0] +(-0.9,0.7);
\draw (4.3,0.2) node[]{$\lhd$};

\draw (4,1) to [out=-90, in=180] +(0.3,-0.4);
\draw (4.6,1) to [out=-90, in=0] +(-0.3,-0.4);
 \draw (4.3,0.6) node[] {$\Diamond$};

\draw (5.2,1) to [out=-70, in=180] +(0.6,-0.2);
\draw (6.4,1) to [out=-110, in=0] +(-0.6,-0.2);
 \draw (5.68,0.8) node[] {$\Diamond$};

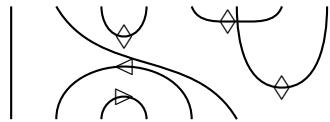
\draw (5.8,1) to [out=-90, in=180] +(0.6,-1.08);
\draw (7,1) to [out=-90, in=0] +(-0.6,-1.08);
 \draw (6.4,-0.08) node[] {$\Diamond$};

\draw (5.8,-0.5) to [out=120, in=-60] +(-2.4,1.5);

\end{scope}
\end{tikzpicture}
$$

Given a Brauer diagram, we define a {\em standard order} on the set of cups and caps. For this, all cups are bigger than all caps; a cup is bigger than another cup if its left-most dot is to the left of the left-most dot of the other one; a cap is bigger than another cap if its left-most dot is to the right of the left-most dot of the other one. The {\em standard marking} then corresponds to this standard order together with the choice that every cap is marked by a right arrow. The standard marking for the Brauer diagram in~\ref{ObHom} is
$$
\begin{tikzpicture}[scale=1,thick,>=angle 90]
\begin{scope}[xshift=4cm]
\draw  (2.8,-0.5) -- +(0,1.5);
\draw (4,-0.5) to [out=90, in=180] +(0.3,0.7);
\draw (4.6,-0.5) to [out=90, in=0] +(-0.3,0.7);
\draw (4.3,0.2) node[]{$\rhd$};

\draw (3.4,-0.5) to [out=90, in=180] +(0.9,0.4);
\draw (5.2,-0.5) to [out=90, in=0] +(-0.9,0.4);
\draw (4.3,-0.1) node[]{$\rhd$};

\draw (4,1) to [out=-70, in=180] +(0.3,-0.2);
\draw (4.6,1) to [out=-110, in=0] +(-0.3,-0.2);
 \draw (4.3,0.8) node[] {$\Diamond$};

\draw (5.2,1) to [out=-90, in=180] +(0.6,-0.4);
\draw (6.4,1) to [out=-90, in=0] +(-0.6,-0.4);
 \draw (5.8,0.6) node[] {$\Diamond$};

\draw (5.8,1) to [out=-90, in=180] +(0.6,-0.6);
\draw (7,1) to [out=-90, in=0] +(-0.6,-0.6);
 \draw (6.4,0.4) node[] {$\Diamond$};

\draw (5.8,-0.5) to [out=120, in=-60] +(-2.4,1.5);

\end{scope}
\end{tikzpicture}
$$

\subsubsection{Composition of morphisms}
Now consider a $(j,k)$-Brauer diagram $d_1$ and an~$(i,j)$-Brauer diagram $d_2$. We will define their composition~$d_1\circ d_2$ as a morphism in~$\cA$.
We create a diagram $d_1\ast d_2$, by drawing $d_1$ on top of~$d_2$ and identifying the $j$ dots on the lower line of~$d_1$ with those on the upper line of~$d_2$. If the corresponding diagram contains closed loops, we set $d_1\circ d_2=0$. If there are 
no loops, $d_1\ast d_2$ can be identified with a partitioning into pairs of~$i+k$ dots, so an~$(i,k)$-Brauer diagram. Then we set $d_1\circ d_2=(-1)^{\gamma(d_1,d_2)}d_1\ast d_2$, where~$\gamma(d_1,d_2)\in\mN$ will be defined using marked diagrams.

We equip $d_1$ and $d_2$ with their standard marking, which gives some decoration of~$d_1\ast d_2$ obtained by keeping the lines of latitude of all markings fixed. We can turn this decoration into a (standard) marking by using two types of operations: (i) permuting adjacent lines of latitudes of two markings and (ii) cancelling a diamond with an arrow which is on the same line of the diagram and lies on an adjacent line of latitude. We can choose a finite number of such operations to obtain a diagram with standard marking. Then~$\gamma(d_1,d_2)\in\mN$ is the sum of the number of operations we used of the first type and the number of operations of the second type where the arrow pointed away from the diamond. It is proved in \cite{Kujawa} that $\gamma$ is independent of the chosen procedure. As an example, we demonstrate
$$
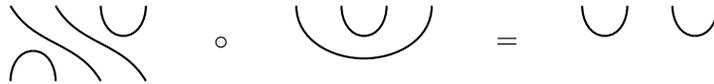
\begin{tikzpicture}[scale=1,thick,>=angle 90]
\begin{scope}[xshift=4cm]
\draw (0,0) to [out=90, in=180] +(0.3,0.4);
\draw (0.6,0) to [out=90, in=0] +(-0.3,0.4);
\draw (1.2,0) to [out=120, in=-60] +(-1.2,1);
\draw (1.8,0) to [out=120, in=-60] +(-1.2,1);
\draw (1.2,1) to [out=-90, in=180] +(0.3,-0.4);
\draw (1.8,1) to [out=-90, in=0] +(-0.3,-0.4);

\draw (2.8,0.5) node[]{$\circ$};

\draw (3.8,1) to [out= -90, in =180] +(0.9, -0.7);
\draw (5.6,1) to [out= -90, in =0] +(-0.9, -0.7);
\draw (4.4,1) to [out= -90, in =180] +(0.3, -0.4);
\draw (5,1) to [out= -90, in =0] +(-0.3, -0.4);

\draw (6.6,0.5) node[]{$=$};

\draw (7.6,1) to [out= -90, in =180] +(0.3, -0.4);
\draw (8.2,1) to [out= -90, in =0] +(-0.3, -0.4);
\draw (8.8,1) to [out= -90, in =180] +(0.3, -0.4);
\draw (9.4,1) to [out= -90, in =0] +(-0.3, -0.4);

\end{scope}
\end{tikzpicture}
$$
Indeed, the procedure can be summarised as
$$
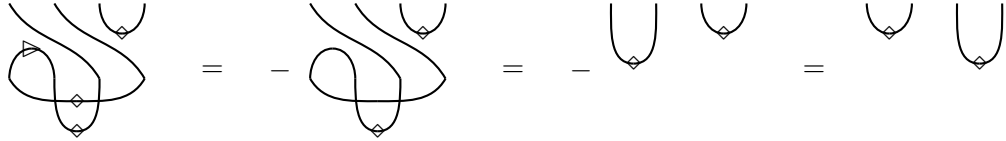
\begin{tikzpicture}[scale=1,thick,>=angle 90]
\begin{scope}[xshift=4cm]

\draw (0.3,0.4) node[]{$\rhd$};
\draw (0,0) to [out=90, in=180] +(0.3,0.4);
\draw (0.6,0) to [out=90, in=0] +(-0.3,0.4);
\draw (1.2,0) to [out=120, in=-60] +(-1.2,1);
\draw (1.8,0) to [out=120, in=-60] +(-1.2,1);
\draw (1.2,1) to [out=-90, in=180] +(0.3,-0.4);
\draw (1.8,1) to [out=-90, in=0] +(-0.3,-0.4);
\draw (1.5,0.6) node[]{$\diamond$};

\draw (0,0) to [out=-60, in = 180] +(0.9, -0.3);
\draw (1.8,0) to [out=-120, in = 0] +(-0.9, -0.3);
\draw (0.6,0) to [out=-90, in =180] +(0.3, -0.7);
\draw (1.2,0) to [out=-90, in =0] +(-0.3, -0.7);
\draw (0.9,-0.3) node[]{$\diamond$};
\draw (0.9,-0.7) node[]{$\diamond$};

\draw (2.7,0.1) node[]{$=$};

\draw (3.6,0.1) node[]{$-$};
%\draw (0.3,0.4) node[]{$\rhd$};
\draw (4,0) to [out=90, in=180] +(0.3,0.4);
\draw (4.6,0) to [out=90, in=0] +(-0.3,0.4);
\draw (5.2,0) to [out=120, in=-60] +(-1.2,1);
\draw (5.8,0) to [out=120, in=-60] +(-1.2,1);
\draw (5.2,1) to [out=-90, in=180] +(0.3,-0.4);
\draw (5.8,1) to [out=-90, in=0] +(-0.3,-0.4);
\draw (5.5,0.6) node[]{$\diamond$};

\draw (4,0) to [out=-60, in = 180] +(0.9, -0.3);
\draw (5.8,0) to [out=-120, in = 0] +(-0.9, -0.3);
\draw (4.6,0) to [out=-90, in =180] +(0.3, -0.7);
\draw (5.2,0) to [out=-90, in =0] +(-0.3, -0.7);
%\draw (0.9,-0.3) node[]{$\diamond$};
\draw (4.9,-0.7) node[]{$\diamond$};

\draw (6.7,0.1) node[]{$=$};

\draw (7.6,0.1) node[]{$-$};

\draw (8,1) to [out=-90, in=180] +(0.3,-0.8);
\draw (8.6,1) to [out=-90, in=0] +(-0.3,-0.8);
\draw (8.3,0.2) node[]{$\diamond$};

\draw (9.2,1) to [out=-90, in=180] +(0.3,-0.4);
\draw (9.8,1) to [out=-90, in=0] +(-0.3,-0.4);
\draw (9.5,0.6) node[]{$\diamond$};

\draw (10.7,0.1) node[]{$=$};

\draw (11.4,1) to [out=-90, in=180] +(0.3,-0.4);
\draw (12,1) to [out=-90, in=0] +(-0.3,-0.4);
\draw (11.7,0.6) node[]{$\diamond$};

\draw (12.6,1) to [out=-90, in=180] +(0.3,-0.8);
\draw (13.2,1) to [out=-90, in=0] +(-0.3,-0.8);
\draw (12.9,0.2) node[]{$\diamond$};

\end{scope}
\end{tikzpicture}
$$
 
%An example of this procedure is given in~\cite[Figure~1]{Kujawa}.

Clearly, the identity morphism of~$i\in\Ob\cA$ is represented by the diagram with~$i$ non-intersecting propagating lines. We denote this isomorphism by~$e_i^\ast$.

\subsubsection{The periplectic Brauer algebra}
The algebras of \cite{Moon} are obtained as the endomorphism algebras in~$\cA$. For $n\in \mZ_{\ge 2}$, we define the {\em periplectic Brauer algebra} as
$$A_n:=\End_{\cA}(n).$$
We can of course also define $A_i=\End_{\cA}(i)\cong\mk$ for $i\in\{0,1\}$, but unless included explicitly we do not consider these cases. This allows for easier formulation of precise statements.

A definition of~$A_n$ in terms of generators and relations is given in~\cite[Definition~2.2]{Moon}, \cite[\S3.5]{Vera} and \cite[\S4.1]{Kujawa}. 
In particular, $A_n$ is generated by the simple reflections $s_i\in\mS_n$, for $i\in\{1,\ldots,n-1\}$, where~$s_i$ interchanges the dots in positions $i$ and $i+1$ from the left, and any diagram with precisely one cup and cap. We also introduce $\varepsilon_i\in A_n$, for $i\in\{1,\ldots,n-1\}$, as
$$
\begin{tikzpicture}[scale=0.7,thick,>=angle 90]
\begin{scope}[xshift=4cm]

\draw (1.5,0.5) node[] {$\varepsilon_i:=$};
\draw  (2.5,0) -- +(0,1);
\draw  (3.1,0) -- +(0,1);

\draw [dotted] (3.4,.5) -- +(0.4,0);

\draw  (4,0) -- +(0,1);

\draw (4.6,0) to [out=90, in=180] +(0.3,0.3);
\draw (5.2,0) to [out=90, in=0] +(-0.3,0.3);

\draw (4.6,1) to [out=-90, in=180] +(0.3,-0.3);
\draw (5.2,1) to [out=-90, in=0] +(-0.3,-0.3);

\draw  (5.8,0) -- +(0,1);
\draw  (6.4,0) -- +(0,1);
\draw [dotted] (6.7,.5) -- +(0.4,0);
\draw  (7.3,0) -- +(0,1);

\end{scope}
\end{tikzpicture}
$$
where we have $i-1$ propagating lines to the left of the cup. We denote by~$(j,i)\in\mS_n\subset A_n$ the elementary permutation, exchanging $i$ and $j$, and by~$ \overline{(j,i)}\in A_n$ the Brauer diagram containing only non-crossing propagating lines, except for one cup and cap, both connecting the $i$th and $j$th dot from the left. In particular, $(i,i+1)=s_i$ and $\overline{(i,i+1)}=\varepsilon_i$.

\subsubsection{Monoidal structure} \label{anti}
It is proved in~\cite[Theorem~3.2.1]{Kujawa} that~$\cA$ is a strict $\mk$-linear monoidal {\em super} category and is generated by four morphisms (diagrams). These are
\begin{enumerate}
\item $I$, the identity morphism of~$1\in\Ob \cA$, represented by a straight line;
\item $X$, the endomorphism $(1,2)$ of~$2\in\Ob \cA$ corresponding to the generator of $\mS_2$;
\item $\cup$, the unique diagram in~$\Hom_{\cA}(0,2)$; and
\item $\cap$, the unique diagram in~$\Hom_{\cA}(2,0)$.
\end{enumerate}
The explicit relations are given in~\cite[Theorem~3.2.1]{Kujawa}.

Using that result, we can define, by checking the consistency with the relations, a {\em contravariant} auto-equivalence $\varphi:\cA\to\cA$, which satisfies $\varphi(i)=i$ for all $i\in\Ob\cA$ and
$$I\mapsto I,\;\;\; X\mapsto -X,\;\;\; \cup\mapsto -\cap\;\;\;\mbox{and }\;\;\;\cap\mapsto \cup,$$ 
and $\varphi(\alpha\otimes\beta)=\varphi(\alpha)\otimes \varphi(\beta)$, for morphisms $\alpha$ and $\beta$.
Note that we also impose~$\varphi(\alpha\circ\beta)=\varphi(\beta)\circ \varphi(\alpha)$, meaning we ignore the $\mF_2$-grading on $\cA$.
 This equivalence restricts to the involutive algebra anti-automorphism of~$A_n$, mentioned in~\cite[Section~3.5]{Vera} and \cite[Lemma~4.1.2(2)]{Kujawa}. However, $\varphi$ differs from the endofunctor in~\cite[Proposition~3.4.1(3)]{Kujawa}, which is contravariant in an $\mF_2$-graded sense.

\subsubsection{}\label{embed} For $0\le k\le n$, we fix throughout the paper the embedding $A_k\hookrightarrow A_n$ by mapping a $(k,k)$-Brauer diagram~$d$ to the $(n,n)$-Brauer diagram $d\otimes e^\ast_{n-k}$.

\subsection{Quasi-hereditary algebras and standard systems}
An algebra~$A$ with some partial order~$\le$ on~$\Lambda$ will be denoted as $(A,\le)$.
\begin{ddef}[\cite{CPS}] The algebra~$(A,\le)$ is {\em quasi-hereditary} if there are modules $\{{\Delta}(\lambda),\lambda\in \Lambda\}$ in~$A${\rm-mod}, such that
\begin{itemize}
\item we have $[\Delta(\lambda):L(\lambda)]=1$ and $[\Delta(\lambda):L(\mu)]=0$ unless $\mu\le \lambda$,
\item there is a surjection~$P(\lambda)\tto \Delta(\lambda)$, such that the kernel has a filtration where the section are isomorphic to modules $\Delta(\nu)$ for~$\nu>\lambda $.
\end{itemize}
\end{ddef}
The modules $\Delta(\lambda)=\Delta_A(\lambda)$ are the {\em standard modules}. Note that $\Delta(\lambda)$ is the maximal quotient of $P(\lambda)$ for which $[\Delta(\lambda):L(\mu)]=0$ unless $\mu\le \lambda$.

\subsubsection{}If an algebra $A$ admits an anti-algebra automorphism~$\phi$, we obtain a contravariant equivalence of categories
\begin{equation}\label{eqTheta}\Upsilon_{\phi}:A\mbox{-mod}\;\stackrel{\sim}{\to}\; A\mbox{-mod},\qquad M\mapsto M^\ast=\Hom_{\mk}(M,\mk),\end{equation}
where the action of~$a\in A$ on~$\alpha\in M^\ast$ is given by
$$(a\alpha)(v)\,=\, \alpha(\phi(a)v),\qquad\mbox{for all}\; v\in M.$$

\begin{prop}\label{NewProp}
Let $(A,\le)$ be a quasi-hereditary algebra with anti-algebra automorphism~$\phi$. There exists a (unique and involutive) bijection $\lambda\mapsto\lambda^\phi$ of $\Lambda$ such that
$$\dim\Hom_A(\Delta(\lambda),\Upsilon_\phi(\Delta(\mu)))\;=\; \delta_{\lambda,\mu^\phi}.$$
Furthermore, we have
$$(P(\lambda):\Delta(\mu))\;=\;[\Delta(\mu^\phi):L(\lambda^{\phi})],\quad\mbox{for all }\lambda,\mu\in\Lambda.$$
\end{prop}
\begin{proof}
Define the (involutive) bijection $\lambda\mapsto\lambda^\phi$ of $\Lambda$ by $\Upsilon_\phi(L(\lambda))\cong L(\lambda^\phi)$. It then follows that $\Upsilon_\phi(P(\lambda))\cong I(\lambda^\phi)$, with $I(\mu)$ the injective hull of $L(\mu)$, for all $\mu\in\Lambda$. We define $\nabla(\mu)$ as the maximal submodule of $I(\mu)$ such that $[\nabla(\mu):L(\nu)]=0$ unless $\nu\le \mu$.
We then find $\Upsilon_\phi(\Delta(\lambda))\cong \nabla(\lambda^\phi)$.
 By construction, we have
$$\dim\Hom_A(\Delta(\lambda),\nabla(\mu))\,=\,\delta_{\lambda,\mu},\qquad\mbox{for all}\;\; \lambda,\mu\in\Lambda,$$
which proves that $(\cdot)^\phi$ is the bijection in the proposition.
By \cite[Lemma~2.5]{DR}, we have the general reciprocity relation
$$(P(\lambda):\Delta(\mu))\;=\; [\nabla(\mu):L(\lambda)],\qquad\mbox{for all}\;\; \lambda,\mu\in\Lambda,$$
which implies the one in the proposition, by application of $\Upsilon_{\phi}$.
\end{proof}
%Consider the ordinary contravariant equivalence of categories
%$$\Upsilon:A\mbox{-mod}\;\stackrel{\sim}{\to}\; A^{\op}\mbox{-mod},\qquad M\mapsto M^\ast=\Hom_{\mk}(M,\mk).$$
%An algebra~$A$ is quasi-hereditary if and only if~$A^{\op}$ is, see~\cite{CPS}. Under~$\Upsilon^{-1}$, the standard modules of~$A^{\op}$ are mapped to the costandard modules $\{\nabla(\mu)\,|\,\mu\in \Lambda\}$ of~$A$. These satisfy 
%\begin{equation}\label{eqHomDN}
%\Hom_A(\Delta(\lambda),\nabla(\mu))\,=\,\delta_{\lambda,\mu}\mk,\qquad\mbox{for all}\;\; \lambda,\mu\in\Lambda,
%\end{equation} 
%see \cite[Lemma~1.2]{DR}. By \cite[Lemma~2.5]{DR}, we have the general reciprocity relation
 %\begin{equation}\label{eqBGG}(P(\lambda):\Delta(\mu))\;=\; [\nabla(\mu):L(\lambda)],\qquad\mbox{for all}\;\; \lambda,\mu\in\Lambda.\end{equation}

%Now 

\subsubsection{} Consider an abelian~$\mk$-linear category $\cC$ and a partially ordered set $(S,\le)$. 
A {\em standard system in~$\cC$ for $S$} is a set of objects $\{X(\alpha)\,|\, \alpha\in S\}$ in~$\cC$, such that for all $\alpha,\beta\in S$:
\begin{enumerate}
\item $\End_{\cC}(X(\alpha))=\mk$;
\item $\Hom_{\cC}(X(\alpha),X(\beta))=0$ unless $\alpha\le \beta$;
\item  $\Ext^1_{\cC}(X(\alpha),X(\beta))=0$ unless $\alpha< \beta$.
\end{enumerate}
For a quasi-hereditary algebra~$(A,\le)$, the standard modules $\{\Delta(\lambda)\,|\,\lambda\in\Lambda\}$ form a standard system in~$A$-mod for $(\Lambda,\le)$, see {\it e.g.} \cite{DR}.

\subsection{Standardly based algebras}
As observed in~\cite{Kujawa, Vera}, $A_n$ fails to be cellular by lack of a suitable involution. Omitting this involution in the definition of cellular algebras in~\cite[Definition~1.1]{CellAlg} leads to the concept of standardly based algebras in~\cite[Definition~1.1.2]{JieDu}. 

\subsubsection{}A standardly based algebra~$A$ has a special basis in terms of a poset $(\mathbf{L},\unlhd)$. The algebra $A$ admits (left) {\em cell modules} $W_A(\alpha)=W(\alpha)$, $\alpha\in\mathbf{L}$, introduced as `standard modules' in~\cite[\S2.1]{JieDu}.

By \cite[Theorem~2.4.1]{JieDu}, $\Lambda$ can be naturally identified with a subset of~$\mathbf{L}$. We consider~$\Lambda$ as a poset for the inherited partial order from~$\mathbf{L}$. For~$\lambda\in\Lambda\subset \mathbf{L}$, we have 
$$[W(\lambda):L(\lambda)]=1\qquad\mbox{and}\qquad [W(\alpha):L(\lambda)]=0\quad\mbox{unless } \lambda\unlhd \alpha,\;\;\mbox{for all $\alpha\in\LL$}.$$
By \cite[Proposition 2.4.4]{JieDu}, every indecomposable projective module $P(\lambda)$ with~$\lambda\in\Lambda$ has a filtration with sections given by modules $W(\alpha)$, with~$\alpha\in \mathbf{L}$, such that
$$(P(\lambda):W(\lambda))=1\qquad\mbox{and}\qquad (P(\lambda):W(\alpha))=0\quad\mbox{unless }\lambda\unlhd \alpha,\;\;\mbox{for all $\alpha\in\LL$}.$$
Note that, in general, the above multiplicities depend on the chosen filtration, see {\it e.g.} Section~\ref{SecA2}.

\subsubsection{}\label{StBQH} If~$\mathbf{L}=\Lambda$, the standardly based algebra is quasi-hereditary with standard modules~$W(\lambda)$. Conversely,
every quasi-hereditary algebra is standardly based for $\mathbf{L}=\Lambda$ (as posets) and $W(\lambda)=\Delta(\lambda)$, by \cite[Theorem~4.2.3]{JieDu}.

\subsection{Centraliser algebras}
Fix an algebra $C$ with idempotent $e\in C$ and set $A=eCe$.
\subsubsection{} We have a pair of adjoint functors $(F,G)$ given by
$$
\xymatrix{ 
C\text{-}\mathrm{mod}\ar@/^/[rrrrrr]^{F=e-\,\cong\,eC\otimes_C-}&&&&&& 
A\text{-}\mathrm{mod}.\ar@/^/[llllll]^{G=\Hom_{A}(eC,-)}
}
$$
We have $F\circ G\cong \Id$ on $A$-mod, and~$F$ is exact while $G$ is left exact. 
The following properties are well-known and immediate consequences of the properties of the above functors.
\begin{lemma}\label{LemTriv}
We can choose $\Lambda_A\subset\Lambda_C$ such that $FL_C(\lambda)$ is either isomorphic to $L_A(\lambda)$ or zero. For all $\lambda\in\Lambda_A$, we have $P_A(\lambda)\cong FP_C(\lambda)$.
\end{lemma}

\subsection{Partitions and Young diagrams}
We review some basic combinatorics of partitions and Young diagrams, to fix the conventions that we will follow.

\subsubsection{}\label{SecPn}Fix $n\in \mN$ and consider~$$\mathscr{P}_n:=\{\lambda=(\lambda_1,\lambda_2,\ldots)\,|\,\lambda \vdash n\},$$ the set of partitions of~$n$. For a partition~$\lambda\vdash n$, we set $|\lambda|=n$. %The partial order~$\unlhd$ on~$\Par_n$ is the dominance order.
We extend the dominance order on $\Par_n$ to a partial order on the set of partitions of all numbers, $\sqcup_{i\in\mN}\Par_i$, by setting
\begin{equation}\label{BigDO}
\mu \unlhd \lambda\;\Leftrightarrow \;\begin{cases}|\mu|>|\lambda| &\mbox{ or}\\ 
|\mu|=|\lambda| \mbox{ and }\sum_{j=1}^k \mu_j \,\le\, \sum_{j=1}^k\lambda_j,\qquad \forall\; k.\end{cases}
\end{equation}

\subsubsection{}For any partition~$\lambda$, we let $\lambda^t $ denote the {transpose}. %If $|\lambda|=|\mu|$, we have $\mu\unlhd \lambda$ if and only if~$\lambda^t \unlhd\mu^t $.
For $p\in\mZ_{>0}$, a partition~$\lambda$ is $p$-restricted if~$\lambda_{i}-\lambda_{i+1}<p$ for all $i$. Any partition is $0$-restricted. We denote the set of~$p$-restricted partitions by~$\Par_n^p$. 

\subsubsection{}\label{SecResi} We will identify a partition with its Young diagram, using English notation. For instance, the partition~$(3,1)$ is represented by the diagram {\tiny$\yng(3,1)$}. Each box or node in the diagram has coordinates $(i,j)$, meaning that the box is in row $i$ and column~$j$. The above diagram has boxes with coordinates $(1,1)$, $(1,2)$, $(1,3)$ and $(2,1)$.
The content of a box $b$ in position~$(i,j)$ in a Young diagram is $\con(b):=j-i\in\mZ$.
For a fixed field $\mk$, we define the residue $\resi(b)\in\mk$ of a box $b$ as the image of $\con(b)$ under the ring morphism~$\overline{\cdot}:\mZ\to \mk$. For a partition~$\lambda$ we also define $|\resi(\lambda)|\in\mk$ as the sum of all residues of the boxes in~$\lambda$.

\subsubsection{}\label{DefAR} Define the set $\RR(\lambda)$ for $\lambda\vdash t$ as the subset of~$\Par_{t-1}$ of all partitions which can be obtained by removing one box in the Young diagram of~$\lambda$. Similarly, $\AAAA(\lambda)$ consists of all partitions of~$t+1$ which can be obtained by adding a box to the Young diagram of~$\lambda$. We will write explicitly $\mu=\lambda\cup b$, if~$\mu$ is a partition for which its Young diagram can be obtained by adding the box $b$ to~$\lambda$.

\subsection{Representations of the symmetric group} Let $\mS_t$ be the symmetric group on $t$ symbols.

%\subsubsection{} For the algebra~$\mk\mS_t$, and any partition~$\lambda\vdash t$, we consider the {\em dual Specht module} $W^0(\lambda)$, following~\cite[\S 3.2]{Mathas}. When~$\lambda$ is $p$-restricted, with~$p=\charr(\mk)$, $W^0(\lambda)$ has simple top $L^0(\lambda)$ and we have $\Lambda_{\mk\mS_t}=\Par_t^p$.

\subsubsection{}\label{simplesSG}The group algebra~$\mk \mS_t$ is cellular and hence standardly based, see \cite[Example~1.2]{CellAlg} or \cite[Chapter~3]{Mathas}. For $p:=\charr(\mk)$, we have $$\mathbf{L}=(\Par_t,\unlhd)\quad\mbox{ and }\quad\Lambda=\Par_t^p\subseteq \mathbf{L}.$$ For each $\alpha\vdash t$, the cell module is the {\em dual Specht module} $W^0(\lambda)$, see~\cite[\S 3.2]{Mathas}. When~$\lambda$ is $p$-restricted, $W^0(\lambda)$ has simple top $L^0(\lambda)$.
When~$p\not\in[2,t]$, the algebra~$\mk \mS_t$ is semisimple and then we have $\Par^p_t=\Par_t$ and $L^0(\lambda)=W^0(\lambda)$ for all $\lambda\vdash t$.

The following proposition, stated as \cite[Lemma~8.4.6]{Borelic}, is due to Kleshchev and Nakano \cite{Nakano}.
\begin{prop}\label{PropHN}
The modules $\{W^0(\lambda)\}$ form a standard system for $(\Par_t,\unlhd)$ if and only if~$$\charr(\mk)\not\in \{2,3\}\,\,\,\mbox{ or}\quad \charr(\mk)=3 \,\mbox{ and }\,t=2.$$
\end{prop}

\subsubsection{}\label{psikS}Let $\psi :\mk\mS_t\to \mk \mS_t$ denote the algebra anti-automorphism given by~$w\mapsto (-1)^{l(w)}w^{-1},$
for all $w\in\mS_t$ of length $l(w)$. Under the corresponding equivalence $\Upsilon_{\psi}$ of equation~\eqref{eqTheta} we have
\begin{equation}\label{eqtrans}\Upsilon_{\psi}\left(W^0(\lambda)\right)\;\cong\; W^0(\lambda^t ),\quad\mbox{for all}\;\,\lambda\vdash t,\end{equation}
see {\it e.g.}~\cite[Ex. 3.14(iii)]{Mathas}.

\subsubsection{}\label{SecLR} Consider partitions $\lambda\vdash a$ and $\nu\vdash b$ and a field $\mk$ with~$\charr(\mk)=0$ or $\charr(\mk)>a+b$. Then there must be some $c^\nu_{\lambda,\mu}\in\mN$ for which
$$\ind^{\mk \mS_{a+b}}\left(L^0(\lambda)\boxtimes L^0(\mu)\right)\;\cong\; \bigoplus_{\nu\vdash a+b}L^0(\nu)^{\oplus c^\nu_{\lambda,\mu}}.$$
The multiplicities $c^\nu_{\lambda,\mu}$ are the Littlewood-Richardson (LR) coefficients, see~\cite[\S 16]{James}. By adjunction, the same coefficients appear, for any~$\nu\vdash a+b$, as 
$$\res_{\mk \mS_{a}\times \mS_b}L^0(\nu)=\bigoplus_{\lambda\vdash a,\mu\vdash b}\left(L^0(\lambda)\boxtimes L^0(\mu)\right)^{\oplus c_{\lambda,\mu}^\nu}.$$

%These result remain valid in arbitrary characteristic if 
%%%%%%%%%%%%%%%%%%%%%%%%%%%%%%%%%%%%%%%%%%%%%%%%%%%%%%%%%%%%%%%%%%%

\section{The covers $C_n$}\label{SecCover}

For $n\in\mZ_{\ge 2}$, we define the $\mk$-algebra
$$C_n:=\bigoplus_{i,j\in\JJJ(n)}\Hom_{\cA}(i,j), \qquad\mbox{with }\quad  \JJJ(n):=\{n-2i\,|\, 0\le i\le n/2\}.$$

 \subsection{Main results on~$C_n$}
 We assume $n\in\mZ_{\ge 2}$ and set $p:=\charr(\mk)$.
 \begin{thm}\label{Thmisocl}
The isoclasses of simple $C_n$-modules are labelled by 
 $\Lambda_{C}:=\bigsqcup_{i\in\JJJ(n)} \Par_i^p.$
 \end{thm}

 \begin{thm}\label{ThmQH}
Assume $p\not\in[2,n]$. 
 The algebra~$(C_n,\le)$ is quasi-hereditary, for partial order~$\le$ on~$\Lambda$ given by
 $\mu<\lambda$ if and only if~$|\mu|>|\lambda|$. The following reciprocity relation holds:
 \begin{equation}\label{eqBGGC}(P(\lambda): \Delta(\mu))\;=\;[{\Delta}(\mu^t ):L(\lambda^t )],\quad\mbox{for $\lambda,\mu\in\Lambda_{C}=\sqcup_{i\in\JJJ(n)}\Par_i$.}\end{equation}
 \end{thm}

  \begin{thm}\label{ThmCp}
Set $\mathbf{L}_{C}:=\bigsqcup_{i\in\JJJ(n)} \Par_i,$ with partial order~$\unlhd$ of equation~\eqref{BigDO}.
 \begin{enumerate} 
 \item The algebra~$C_n$ is standardly based for $(\mathbf{L}_C,\unlhd)$. 
 \item The cell modules form a standard system for $(\mathbf{L}_C,\unlhd)$ if and only if~$p\not\in\{2,3\}$ or $p=2=n$.
 \end{enumerate}
  \end{thm}
  
  The following sections are devoted to the proofs of the three theorems. We assume $n$ fixed and write $C$ for $C_n$

  \subsection{The triangular decomposition of $C$}
  We define three subalgebras of~$C$. 
  \begin{itemize}
  \item The subalgebra~$H$ is spanned by all diagrams with only propagating lines.
  \item The algebra~$N$ is spanned by all diagrams which consists only of caps and non-crossing propagating lines.
   \item The algebra~$\overline{N}$ is spanned by all diagrams which consists only of cups and non-crossing propagating lines.
  \end{itemize}
  
  \subsubsection{}\label{simplesH}With the identity morphisms $e_i^\ast\in \End_{\cA}(i)$ interpreted as idempotents in $C_n$, we have 
\begin{equation}\label{eqHS}H\;=\;\bigoplus_{i\in\JJJ(n)} He_i^\ast\;\cong\;  \bigoplus_{i\in\JJJ(n)} \mk\mS_i.\end{equation}
Thus, by~\ref{simplesSG}, the simple modules of~$H$ are labelled by~$\Lambda_C$ and we denote the simple $H$-modules by~$L^0(\lambda)$.
  
 \subsubsection{} We consider a $\mZ$-grading on~$C$, as $C_{(j)}=\bigoplus_{i} e^\ast_i C e^\ast_{i+j}.$
We set
$$C_{\plus}:=\bigoplus_{j>0}C_{(j)}\;\mbox{ and }\;C_{\minus}:=\bigoplus_{j<0}C_{(j)},$$
and use similar notation for any graded subalgebra of~$C$.
We set $B:= HN$ and $\overline{B}:=\overline{N}H$. It is easily checked that these are subalgebras, with~$B_{\minus}=0=\overline{B}_{\plus}$ and $B_{(0)}=H=\overline{B}_{(0)}$.

\begin{lemma}\label{LemPropB}
For the subalgebra~$B$ of~$C$, we have that
\begin{enumerate}
\item $C=B\oplus C_{\minus}B$;
\item $C=\overline{N}B$;
\item $B$ is projective as a left $H$-module;
\item $C$ is projective as a right $B$-module.
\end{enumerate}
\end{lemma}
 \begin{proof}
 The subspace $B$ of~$C$ is spanned by all diagrams without cups and the subspace $C_{\minus}B$ is spanned by all diagrams with at least one cup.
This proves claim (1).

We can decompose any diagram $d$ as the product $d_1d_2d_3$ of three diagrams. We let $d_1$ be made out of all the cups of~$d$ and non-crossing propagating lines, starting in the remaining dots on the upper line of~$d_1$. The diagram $d_3$ is similarly defined as the diagram containing the caps of~$d_3$. Finally, $d_2$ is the unique diagram consisting solely of propagating lines such that~$d=d_1d_2d_3$.
This implies part~(2). Furthermore, to each diagram we can associate the number of propagating lines. This leads to
 a decomposition
$$C=\bigoplus_{i\in\JJJ(n)}\overline{N} e_i^\ast B.$$
 For any diagram $d$ in~$\overline{N} e_i^\ast$ we have $dB\cong e_i^\ast B$, proving that~$C_B$ is projective. Similarly it follows that for any diagram $d$ in~$e^\ast_i B$ we have $Hd\cong He_i^\ast$, which concludes the proof.
 \end{proof}
 
 The reasoning in the proof of Lemma~\ref{LemPropB} has the following consequence.
\begin{cor}\label{CorTriang}
Multiplication provides an isomorphism
$$\bigoplus_{i\in\JJJ(n)} \overline{N}e_i^\ast\,\otimes\, e_i^\ast H\,\otimes\, e_i^\ast N\;\;\stackrel{\sim}{\to}\; \;C.$$
\end{cor}

\subsubsection{} In the terminology of \cite[Definition~3.2.3]{Borelic}, Lemma~\ref{LemPropB} states that~$B$ is a `graded pre-Borelic subalgebra' of~$C$. By \cite[Lemma~3.2.2 and Definition~3.1.1]{Borelic}, it follows that there is a correspondence between isoclasses of simple modules for $C$ and $H$. Theorem~\ref{Thmisocl} thus follows from equation \eqref{eqHS}. Theorem~\ref{ThmCp} then follows immediately from \cite[Corollary~5.2.2]{Borelic} and Proposition~\ref{PropHN}.

\subsubsection{} \label{SecDefMod}  In the following we will freely interpret $H$-modules as $B$-modules with trivial $B_{\plus}$-action.
By \cite[Proposition~5.2.3]{Borelic}, for any~$\mu\in \mathbf{L}_C$, the cell module is given by
\begin{equation}\label{eqWC}W(\mu):= C\otimes_B W^0(\mu).\end{equation}
For any~$\lambda\in \Lambda_C$, we also define
\begin{equation}
\label{defDover}
\overline{\Delta}(\lambda):= C\otimes_B L^0(\lambda)\qquad\mbox{and}\qquad \overline{\Delta}:=\bigoplus_{\lambda\in\Lambda}\overline{\Delta}(\lambda).
\end{equation}
The simple $C$-module $L(\lambda)$ is the simple top of the module $\overline{\Delta}(\lambda)$, see~\cite[Lemma 3.1.4]{Borelic}.

\subsubsection{}\label{Maschke} By Maschke's theorem, the algebra~$H$ in~\eqref{eqHS} is semisimple if~$p\not\in[2,n]$. The quasi-heredity in Theorem~\ref{ThmQH} thus follows from \cite[Corollary~4.5.4(3)]{Borelic}. The standard module are given by
$$\Delta(\lambda)\;:=\;\overline{\Delta}(\lambda)\;\cong\;W(\lambda)\qquad\mbox{if $p\not\in[2,n]$, $\;$ for $\lambda\in \Lambda_C=\mathbf{L}_C$.}$$
All of this also follows immediately from~\ref{StBQH}, by Theorems~\ref{Thmisocl} and~\ref{ThmCp}(1).

\subsection{Humphreys-BGG reciprocity}

In this section we assume that~$p=\charr(\mk)\not\in[2,n]$.

\subsubsection{}\label{StandNonsense}

 %By~\cite{CPS}, the quasi-heredity of~$(C,\le)$ is equivalent to the quasi-heredity of~$(C^{\op},\le)$ and the functor $\Hom_{\mk}(-,\mk)$ maps the costandard modules of~$C$ to the standard modules of~$C^{\op}$.
 
The algebra~$C$ inherits an anti-automorphism~$\phi$ from the anti-autoequivalence $\varphi$ of~$\cA$ in~\ref{anti}. We consider the corresponding
 contravariant equivalence $\Upsilon_{\phi}$ of~$C\mbox{-mod}$ of equation \eqref{eqTheta}. Equation~\eqref{eqBGGC} then follows from Proposition~\ref{NewProp} and the following lemma.

 % For any~$\lambda\in \Lambda$, we define $\lambda^\phi\in \Lambda$ by the requirement $L(\lambda)\cong \Upsilon_{\phi}(L(\lambda^\phi))$. Then we have $\Upsilon_{\phi}(\Delta(\lambda^\phi))\cong \nabla(\lambda)$ and hence, by equation \eqref{eqBGG}
 % $$(P(\lambda):\Delta(\mu))\;=\; [\Delta(\mu^\phi):L(\lambda^\phi)].$$
% It thus remains to show that~$\lambda^{\phi}=\lambda^t $ in order to prove the Humphreys-BGG reciprocity relation in Theorem~\ref{ThmQH}. This relation follows from equation \eqref{eqHomDN} and the following lemma. 

\begin{lemma}\label{LemDN}
For any~$\lambda,\mu\in\Lambda$, we have
$$\dim\Hom_C\left(\Delta(\lambda),\Upsilon_{\phi}(\Delta(\mu))\right)\;=\;\delta_{\lambda,\mu^t }.$$
 \end{lemma}
 \begin{proof}
 It is clear that~$\phi(H)=H$, $\phi(B_{\plus})=\overline{B}_{\minus}$ and $\phi(\overline{B}_{\minus})=B_{\plus}$. Moreover, the restriction of~$\phi$ to~$H$ corresponds to the anti-automorphisms $\psi$ of the $\mk\mS_i$ in~\ref{psikS}.

 By equations~\eqref{eqTheta} and~\eqref{defDover} we have an isomorphism of vector spaces
 $$\Upsilon_{\phi}(\Delta(\mu))\,\cong\,\Hom_{B}(L^0(\mu),\Hom_{\mk}(C,\mk))\,\cong\, \Hom_{H}(L^0(\mu), \Hom_{\mk}(C/CB_{\plus},\mk)).$$
 The inherited $C$-module structure on the right-hand side is described as follows. The action of~$x\in C$ on~$\alpha : L^0(\mu)\to \Hom_{\mk}(C/CB_{\plus},\mk)$ is given by
 $$((x\alpha)(v))(\overline{a})=(\alpha(v))(\phi(x)\overline{a}),\qquad\mbox{for all }\; v\in L^0(\mu)\mbox{ and }a\in C\mbox{ with }\overline{a}=a+CB_{\plus}.$$
 Furthermore, we have
\begin{equation}\label{eqLemDN}\Hom_C\left(\Delta(\lambda),\Upsilon_{\phi}(\Delta(\mu))\right)\,\cong\, \Hom_{H}(L^0(\lambda), \Upsilon_{\phi}(\Delta(\mu))^{B_{\plus}}).\end{equation}
Using $\phi(B_{\plus})=\overline{B}_{\minus}$, it follows that
 $$\Upsilon_{\phi}(\Delta(\mu))^{B_{\plus}}\,\cong\,\Hom_{H}\left(L^0(\mu), \Hom_{\mk}( \overline{B}_{\minus}C \backslash C/CB_{\plus},\mk)\right).$$
 
As~$\overline{B}_{\minus}C$ is the subspace of~$C$ spanned by all diagrams containing cups and $CB_{\plus}$ the subspace spanned by all diagrams containing caps, the right $H$-module $\overline{B}_{\minus}C \backslash C/CB_{\plus}$ is isomorphic to the right regular $H$-module.
This implies
 $$\Upsilon_{\phi}(\Delta(\mu))^{B_{\plus}}\cong\Hom_{H}\left(L^0(\mu), \Hom_{\mk}( H,\mk)\right)\cong \Hom_{\mk}(L^0(\mu),\mk)\cong \Upsilon_{\psi}(L^0(\mu)),$$
 as $H$-modules. The conclusion hence follows from equations~\eqref{eqLemDN} and~\eqref{eqtrans}.
 \end{proof}
 
%%%%%%%%%%%%%%%%%%%%%%%%%%%%%%%%%%%%%%%%%%%%%%%%%%%%%%%%%%%%%%%%%%%%%%%%%%%%%%%%%%%%%%

 \section{The periplectic Brauer algebras $A_n$}\label{SecAC}
 In this section we will transfer results from $C_n$ to~$A_n=e_n^\ast C_ne_n^\ast$ through the exact functor
 \begin{equation}\label{FunF}
 F=e_n^\ast-=e_n^\ast C\otimes_C-\,:\;\; C_n\mbox{-mod}\;\to\;A_n\mbox{-mod}.
 \end{equation}
 
 \subsection{Main results on~$A_n$}
 We fix $n\in \mZ_{\ge 2}$ and set $\JJJ^0(n)=\JJJ(n)\backslash\{0\}$ and $p:=\charr(\mk)$.
 \begin{thm}\label{ThmMor}${}$
 \begin{enumerate}
 \item If $n$ is odd, $A_n$ is Morita equivalent to~$C_n$ through $F$.
 \item If $n$ is even and $n\not\in \{2,4\}$, the functor $F$ induces isomorphisms
 $$\Ext^i_C(M,N)\;\stackrel{\sim}{\to}\;\Ext^i_A(FM,FN),\qquad\mbox{for all }\;M,N\in\cF(\overline{\Delta})\quad\mbox{and}\quad i\in\{0,1\}.$$
 If $n\not\in \{2,4\}$ and $p\not\in[2,n]$, $C_n$ is thus a quasi-hereditary $1$-cover of~$A_n$, in the sense of \cite[Definition~4.37]{Rou}, and a Schur algebra, see {\it e.g.}~\cite[Definition~2.9.4]{Borelic}.
 \end{enumerate}
 \end{thm}
 
 \begin{thm}\label{ThmSB}
${}$
 \begin{enumerate}
 \item The isoclasses of simple $A_n$-modules are labelled by $\Lambda_{A}:=\bigsqcup_{i\in\JJJ^0(n)}\Par_{i}^p.$
 \item  The algebra~$A_n$ is standardly based, for $(\mathbf{L}_A,\unlhd):=(\mathbf{L}_C,\unlhd)$ as in Theorem~\ref{ThmCp}.
 \item The cell modules $\{W(\mu)\}$ of~$A_n$ form a standard system for $(\mathbf{L},\unlhd)$ if and only if~$n\not\in\{2,4\}$ and $p\not\in\{2,3\}$. In this case, multiplicities in cell filtrations are unambiguous.
 \item If $n\not\in\{2,4\}$ and $p\not\in[2,n]$, we have
 $$(P(\lambda):W(\mu))\,=\, [W(\mu^t ):L(\lambda^t )],\qquad\mbox{for $\lambda\in\Lambda_A=\bigsqcup_{i\in\JJJ^0(n)}\Par_{i}\;$ and $\;\mu\in \mathbf{L}_A=\bigsqcup_{i\in\JJJ(n)}\Par_{i}$}.$$
  \end{enumerate}
 \end{thm}
 
 \begin{rem}\label{RemSB}
In a sense, theorem~\ref{ThmSB}(4) remains valid for $n\in\{2,4\}$, see Lemma~\ref{LemBGGA}. 
\end{rem}

 \begin{rem}\label{3rem}
 The twist by $\lambda\mapsto\lambda^t $ in the Humphreys-BGG reciprocity relation shows that
 the standardly based structure of $A$ is not a cell datum in the sense of \cite[Definition~1.1]{CellAlg}.
 
The anti-autoequivalence $\varphi$ of $\cA$ in~\ref{anti} restricts to an anti-automorphism $\varphi_n$ of $A_n$, such that ${}^{\varphi_n}L(\lambda)\cong L(\lambda^t)$.

 Any algebra admits at least one standardly based structure, see \cite[Corollary~1]{Borelic}. The reason the standardly based structure in Theorem~\ref{ThmSB}(2) is nonetheless interesting lies in the exceptional properties in Theorem~\ref{ThmSB}(3) and (4) and the fact that the cell modules are cyclic.
 \end{rem}
 
 The rest of this section is devoted to the proofs of the two theorems. 
  
\subsection{Some morphisms in~$\cA$}
\subsubsection{}\label{Defab}For any~$i\in\JJJ(n)$ we set
\begin{equation*}\label{defA}
\begin{tikzpicture}[scale=0.76,thick,>=angle 90]
\begin{scope}[xshift=8cm]
\node at (0,0.5) {$a_i:=$};
\draw (.7,0) -- +(0,1);
\draw (1.3,0) -- +(0,1);
\draw [dotted] (1.6,.5) -- +(1,0);
\draw (2.9,0) -- +(0,1);
\draw (3.5,0) -- +(0,1);
\draw (4.1,1) to [out=-90,in=-180] +(.3,-.3) to [out=0,in=-90] +(.3,.3);
\draw (5.3,1) to [out=-90,in=-180] +(.3,-.3) to [out=0,in=-90] +(.3,.3);
\draw [dotted] (6.2,.9) -- +(1,0);
\draw (8,1) to [out=-90,in=-180] +(.3,-.3) to [out=0,in=-90] +(.3,.3);
\node at (11.6,0.5) {$\in\, \Hom_{\cA}(i,n)=e_n^\ast Ce_i^\ast$.};
\end{scope}
\end{tikzpicture}
\end{equation*}
For any~$i\in\JJJ^0(n)$ we set
\begin{equation*}\label{defBB}
\begin{tikzpicture}[scale=0.76,thick,>=angle 90]
\begin{scope}[xshift=8cm]
\node at (0,0.5) {$b_i:=$};
\draw (.7,0) -- +(0,1);
\draw (1.3,0) -- +(0,1);
\draw [dotted] (1.6,.5) -- +(1,0);
\draw (2.9,0) -- +(0,1);
\draw (3.5,0) to [out=90,in=-180] +(.3,.3) to [out=0,in=90] +(.3,-.3);
\draw (4.7,0) to [out=90,in=-180] +(.3,.3) to [out=0,in=90] +(.3,-.3);
\draw [dotted] (5.7,.1) -- +(1,0);
\draw (7.4,0) to [out=90,in=-180] +(.3,.3) to [out=0,in=90] +(.3,-.3);
\draw (8.6,0) to [out=145,in=-35] +(-5.1,1);
\node at (11.6,0.5) {$\in\, \Hom_{\cA}(n,i)=e_i^\ast Ce_n^\ast$.};
\end{scope}
\end{tikzpicture}
\end{equation*}
Finally, in case $n$ is even, we set
\begin{equation*}\label{defOA}
\begin{tikzpicture}[scale=0.9,thick,>=angle 90]
\begin{scope}[xshift=8cm]
\node at (3,0.2) {${b}_0:=$};

\draw (4.1,0) to [out=90,in=-180] +(.3,.3) to [out=0,in=90] +(.3,-.3);
\draw (5.3,0) to [out=90,in=-180] +(.3,.3) to [out=0,in=90] +(.3,-.3);
\draw [dotted] (6.2,.1) -- +(1,0);
\draw (8,0) to [out=90,in=-180] +(.3,.3) to [out=0,in=90] +(.3,-.3);
\node at (11.6,0.2) {$\in\, \Hom_{\cA}(n,0)=e_0^\ast Ce_n^\ast$.};
\end{scope}
\end{tikzpicture}
\end{equation*}
A direct computation then proves the following lemma.
\begin{lemma}\label{Lemab} For $i\in \JJJ^0(n)$, we have $b_ia_i= e_i^\ast$.
\end{lemma}

For every $M\in C${\rm-mod} and $j\in\JJJ(n)$, we have a $\mk$-linear morphism
\begin{equation}\label{eqetaj}
\nu_j:\;e_j^\ast M\;\to\; \Hom_{A}(e_n^\ast Ce_j^\ast,e_n^\ast M),\qquad v\mapsto \alpha_v\quad\mbox{with }\;\alpha_v(x)=xv.\end{equation}

\begin{cor}\label{coretaj}
For $j\in \JJJ^0(n)$, the morphism $\nu_j$ in~\eqref{eqetaj}
is an isomorphism.
\end{cor}
\begin{proof}
By Lemma \ref{Lemab}, an inverse is given by mapping $\alpha\in \Hom_{A}(e_n^\ast Ce_j^\ast,e_n^\ast M)$ to~$b_j\alpha(a_j)$.
\end{proof}

\subsubsection{}By definition of~$L_C(\lambda)$ in~\ref{SecDefMod}, we have $e_j^\ast L_C(\lambda)\not=0$ with~$j=|\lambda|$, for all~$\lambda\in \Lambda_C$. Corollary~\ref{coretaj} thus implies that~$e_n^\ast L_C(\lambda)\not=0$ if~$\lambda\not=\varnothing$. On the other hand, $L_C(\varnothing )$ is the one-dimensional $C$-module which satisfies $e_0^\ast L_C(\varnothing )=L_C(\varnothing ),$ so $e_n^\ast L_C(\lambda)=0$. 
We choose the convention in Lemma~\ref{LemTriv}, which thus yields for $F=e_n^\ast-$
\begin{equation}
\label{eqCorLambda0}
FL_{C_n}(\lambda)=L_{A_n}(\lambda)\quad\mbox{and}\quad FP_{C_n}(\lambda)\cong P_{A_n}(\lambda),\qquad\mbox{for all $\lambda\in\Lambda_C\backslash\{\varnothing\}$,}
\end{equation}
This implies Theorem~\ref{ThmSB}(1). We will simplify the notation of $L_{A_n}(\lambda)$ to $L_A(\lambda)$, $L_{n}(\lambda)$ or $L(\lambda)$ depending on which information is not clear from context.

%\begin{equation}
%\label{eqCorLambda0}
%FL_{C_n}(\lambda)\;=\;\begin{cases}L_{A_n}(\lambda)& \mbox{if}\;\lambda\not=\varnothing,\\ 0& \mbox{if}\;\lambda=\varnothing,\end{cases}\qquad\mbox{for all~$\lambda\in \Lambda_C$.}
%\end{equation}

\subsection{A pair of adjoint functors}\label{SecAdj}
We consider the left exact functor
\begin{equation}\label{eqG}G=\Hom_{A_n}(e_n^\ast C_n,-):\;A_n\mbox{-mod}\to C_n\mbox{-mod},\end{equation}
which is right adjoint to the exact functor $F$ in equation \eqref{FunF}. We have the two corresponding adjoint natural transformations, the unit $\eta:\Id\Rightarrow G\circ F$ and the counit $\varepsilon: F\circ G\Rightarrow \Id$. The counit~$\varepsilon$ is clearly an isomorphism.

\begin{lemma}\label{lemeta}
${}$
\begin{enumerate}
\item If $n$ is odd, $F$ and $G$ are mutually inverse equivalences between~$C_n${\rm -mod} and $A_n${\rm-mod}.
\item If $n$ is even, $\eta_M: M\to G\circ F(M)$ for $M\in C_n${\rm -mod} is an isomorphism if and only if $\nu_0$ in~\eqref{eqetaj}
%$$e_0^\ast M\;\to\; \Hom_{A_n}(e_n^\ast C_ne_0^\ast,e_n^\ast M),\qquad v\mapsto \alpha_v,\qquad\mbox{with }\;\alpha_v(x)=xv,$$
is an isomorphism.
\end{enumerate}
\end{lemma}
\begin{proof}
For any~$M$ in~$C$-mod, we have that~$G\circ F(M)$ is given by~$\Hom_{A}(e_n^\ast C, e_n^\ast M)$. The morphism~$\eta_M$ decomposes into vector space morphisms from $e_j^\ast M$ to~$e_j^\ast G\circ F(M)$ for all $j\in \JJJ(n)$. These morphisms are precisely the ones in equation \eqref{eqetaj}. 

If $n$ is odd, we have $0\not\in\JJJ(n)$ and Corollary~\ref{coretaj} thus implies that~$\eta_M$ is an isomorphism for any module $M$, proving part~(1). Similarly, for $n$ even, Corollary~\ref{coretaj} implies that~$\eta_M$ is an isomorphism if and only if~$e_0^\ast M\to e_0^\ast G\circ F(M)$ is an isomorphism, proving part~(2).
\end{proof}

Lemma~\ref{lemeta}(1) implies Theorem~\ref{ThmMor}(1) and everything in Theorem~\ref{ThmSB} for the case $n$ odd, by Theorems~\ref{ThmQH} and~\ref{ThmCp}.

\subsubsection{}\label{sseqI}The subalgebra of~$A_n$ spanned by diagrams without cups is isomorphic to~$\mk\mS_n$. The algebra~$\mk\mS_n$ is also naturally a quotient of~$A_n$, with respect to the two-sided ideal $I$ spanned by all diagrams which contain at least one cup (or cap). As vector spaces $A_n=\mk\mS_n\oplus I$. 

\begin{cor}\label{CorLAn}
The module $L_A(\lambda)$, for $\lambda\in\Par_n^p\subset\Lambda_A$, is equal to~$L^0(\lambda)$ as a $\mk\mS_n$-module and has trivial action of~$I$.
\end{cor}
\begin{proof}
By equation~\eqref{eqCorLambda0}, we have 
$$L_A(\lambda)\cong e_n^\ast L_C(\lambda)\cong e_n^\ast \overline{\Delta}(\lambda).$$
The result thus follows from \eqref{defDover} and the fact that~$I\subset CB_{\plus}$, where~$A$ is interpreted inside~$C$.
\end{proof}

%\begin{rem}\label{CorLAnRem}
%Similarly, for the sequence of inclusions $\mS_i\hookrightarrow A_i\hookrightarrow A_n$, we obviously have
%$$[\res_{\mk\mS_i}L_{A_n}(\lambda):L^0(\lambda)]\not=0,\qquad\mbox{for any}\;\,\lambda\in\Par_i^p,\;i\in \JJJ^0(n).$$
%\end{rem}

\subsection{Some idempotents and a nilpotent}\label{SecIdNil} Since the results in Section~\ref{SecAdj} already prove everything in Theorems~\ref{ThmMor} and~\ref{ThmSB} for $n$ odd, from now on we will restrict to {\em the case $n$ even.}
We introduce elements $c_i^\ast\in A_n$, for $i\in \JJJ(n)$, defined as $c_i^\ast:= a_ib_i$, with $a_i,b_i$ as in \ref{Defab}. These are given explicitly by
$$
\begin{tikzpicture}[scale=0.9,thick,>=angle 90]
\begin{scope}[xshift=4cm]
\node at (0.4,0.5) {$c_0^\ast=$};
\draw (1.9,0) to [out=90,in=-180] +(.3,.3) to [out=0,in=90] +(.3,-.3);
\draw (3.1,0) to [out=90,in=-180] +(.3,.3) to [out=0,in=90] +(.3,-.3);
\draw (4.3,0) to [out=90,in=-180] +(.3,.3) to [out=0,in=90] +(.3,-.3);
\draw [dotted] (5.2,.1) -- +(1,0);
\draw (6.6,0) to [out=90,in=-180] +(.3,.3) to [out=0,in=90] +(.3,-.3);

\draw (1.9,1) to [out=-90,in=-180] +(.3,-.3) to [out=0,in=-90] +(.3,.3);
\draw (3.1,1) to [out=-90,in=-180] +(.3,-.3) to [out=0,in=-90] +(.3,.3);
\draw (4.3,1) to [out=-90,in=-180] +(.3,-.3) to [out=0,in=-90] +(.3,.3);
\draw [dotted] (5.1,.9) -- +(1,0);
\draw (6.6,1) to [out=-90,in=-180] +(.3,-.3) to [out=0,in=-90] +(.3,.3);

\node at (8.4,0.5) {$c_i^\ast=$};
\draw (9.7,0) -- +(0,1);
\draw [dotted] (10,.5) -- +(0.7,0);
\draw (10.9,0) -- +(0,1);
\draw (11.5,0) to [out=90,in=-180] +(.3,.3) to [out=0,in=90] +(.3,-.3);

\draw (12.7,0) to [out=90,in=-180] +(.3,.3) to [out=0,in=90] +(.3,-.3);
\draw [dotted] (13.6,.1) -- +(1,0);
\draw (15,0) to [out=90,in=-180] +(.3,.3) to [out=0,in=90] +(.3,-.3);
\draw (16.2,0) to [out=148,in=-32] +(-4.7,1);
\draw (12.1,1) to [out=-90,in=-180] +(.3,-.3) to [out=0,in=-90] +(.3,.3);
\draw (13.3,1) to [out=-90,in=-180] +(.3,-.3) to [out=0,in=-90] +(.3,.3);
\draw [dotted] (14.1,.9) -- +(1,0);
\draw (15.6,1) to [out=-90,in=-180] +(.3,-.3) to [out=0,in=-90] +(.3,.3);
\end{scope}
\end{tikzpicture}
$$
for $i>0$.
In particular $c_n^\ast=1_{A_n}$. We also write $c_{-2}^\ast=0$. We have
$$c_j^\ast c^\ast_i=c_i^\ast,\qquad\mbox{if } j\ge i\mbox{ and }(j,i)\not=(0,0).$$
So, for $i>0$, $c_i^\ast$ is an idempotent, and we also have $(c_0^\ast)^2=0$. 
\begin{lemma}\label{LemECA}
For each $i\in \JJJ(n)$, we have an isomorphism of left $A$-modules
\begin{equation}\label{eqinjiso}e_n^\ast Ce_i^\ast\;\stackrel{\sim}{\to}\; Ac_i^\ast,\quad y\mapsto yb_i,\end{equation}
which restricts to an isomorphism
$ e_n^\ast Ce_{i-2}^\ast Ce_i^\ast\;\cong Ac_{i-2}^\ast Ac_i^\ast.$
\end{lemma}
\begin{proof}
The form of the diagram $b_i$ shows that the morphism in~\eqref{eqinjiso} is injective.
From considering the bases of diagrams of both spaces, it then follows that it is an isomorphism.

Inside $e_n^\ast Ce_i^\ast $, the submodule $e_n^\ast Ce_{i-2}^\ast Ce_i^\ast$ is spanned by all diagrams with at least $(n-i+2)/2$ cups. The number of cups of a diagram is preserved under the isomorphism \eqref{eqinjiso}. This proves the second isomorphism.
\end{proof}

\begin{cor}\label{CorDeltaA}
For each $i\in \JJJ(n)$, we have an isomorphism of~$A$-modules
$$e_n^\ast C\otimes_B He_i^\ast\;\cong Ac_i^\ast/Ac_{i-2}^\ast Ac_i^\ast.$$
\end{cor}
\begin{proof}
As a special case of \cite[Lemma~3.2.8]{Borelic}, we have
$$C\otimes_B He_i^\ast \;\cong\; Ce_i^\ast/Ce_{i-2}^\ast Ce_i^\ast.$$
The result hence follows from Lemma~\ref{LemECA}.
\end{proof}
By equation \eqref{defDover}, this corollary, or Lemma~\ref{LemECA}, implies an isomorphism
\begin{equation}\label{enD0}e_n^\ast\overline{\Delta}(\varnothing)\;\stackrel{\sim}{\to}\; Ac_0^\ast,\qquad y\mapsto yb_0.\end{equation}

\begin{lemma}\label{NewLemmaNew} Assume that $n>2$.
For left ideals $N\subset M\subset A$, we have a monomorphism
$$\Hom_A(Ac_0^\ast, M/N)\;\hookrightarrow\; \left((c_0^\ast A\cap M)+N\right)/N,\quad \alpha\mapsto \alpha(c_0^\ast).$$
\end{lemma}
\begin{proof}
Take~$\alpha\in \Hom_A(Ac_0^\ast,M/N)$
with~$\alpha(c_0^\ast)=v+N$, for some $v\in M\subset A$. Consider the diagram
\begin{equation*}\label{Defw}
\begin{tikzpicture}[scale=0.9,thick,>=angle 90]

\begin{scope}[xshift=8cm]
\node at (-3,0.5) {$w:=$};
\draw (-1.7,1) to [out=-90,in=-180] +(.3,-.3) to [out=0,in=-90] +(.3,.3);
\draw (-1.7,0) to [out=50, in=-130] +(1.2,1);
\draw (-1.1,0) to [out=90,in=-180] +(.3,.3) to [out=0,in=90] +(.3,-.3);
\draw (.1,0) -- +(0,1);
\draw (.7,0) -- +(0,1);
\draw (1.3,0) -- +(0,1);
\draw [dotted] (1.6,.5) -- +(1,0);
\draw (2.9,0) -- +(0,1);
\draw (3.5,0) -- +(0,1);

\end{scope}
\end{tikzpicture}
\end{equation*}
Since $wc_0^\ast=-c_0^\ast$ we must have $v+ N=-wv+N$, so we can replace $v$ by $-wv$. In other words, we can assume $v$ is a linear combination of diagrams which have the same cup as $w$.
Using $c_2^\ast c_0^\ast=c_0^\ast$, we can then assume that $v$ is a linear combination of diagrams which have all the cups of $c_0^\ast$, from which the conclusion follows.
\end{proof}

\begin{thm}[Double centralisers]\label{ThmDC}
If $n\in\mZ_{>2}$, the $(A,C)$-bimodule $X:=e_n^\ast C$ satisfies
$$A\cong \End_C(X)\quad\mbox{ and }\quad C\cong \End_A(X)^{\op}.$$
\end{thm}
\begin{proof}
The first isomorphism follows from the definition~$A=e_n^\ast Ce_n^\ast$. We will prove that the canonical algebra morphism $C\to \End_A(e_n^\ast C)^{\op}$ is an isomorphism for $n>2$. This morphism restricts to $\mk$-linear morphisms for all $i,j\in \JJJ(n)$
$$e_i^\ast C e_j^\ast\;\to\;\Hom_A(e_n^\ast C e_i^\ast, e_n^\ast C e_j^\ast),\quad x\mapsto \beta_x\quad\mbox{with}\quad \beta_x(v)=vx.$$
We can compose these with the isomorphisms in Lemma~\ref{LemECA}, yielding morphisms
\begin{equation}\label{isoij}e_i^\ast Ce_j^\ast\, \to\, \Hom_A(Ac_i^\ast,Ac_j^\ast),\qquad x\mapsto \gamma_x\; \mbox{ with } \gamma_x(c_i^\ast)=a_i x b_j.
\end{equation}
It thus suffices to prove that these are all isomorphisms. It follows from the shape of the diagrams $a_i$ and $b_j$ that the morphisms in~\eqref{isoij} are injective.

Assume first that $i\not=0$. Then $c_i^\ast$ is an idempotent and it follows easily that 
$$\dim \Hom_A(Ac_i^\ast,Ac_j^\ast)\;=\;  \dim c_i^\ast Ac_j^\ast      \;=\;\dim e_i^\ast C e_j^\ast,$$ showing that \eqref{isoij} is an isomorphism.

Now assume $i=0$. We can use Lemma~\ref{NewLemmaNew} to show
$$\dim\Hom_A(Ac_0^\ast, Ac_j^\ast)\le \dim e_0^\ast Ce_j^\ast,$$
which concludes the proof.
\end{proof}

%for each cup in $c_0^\ast$ we can construct a diagram $v$ with one cup and cap and otherwise only non-intersecting propagating lines such that $vc_0^\ast=c_0^\ast$. The example for the left-most cup is given by $w$ in the proof of Lemma~\ref{AxAc}.

%For $j\in\{0,2\}$ it follows from direct computation that the dimensions in~\eqref{isoij} agree. For $j>2$
% we can use the arguments in the proof of Lemma~\ref{AxAc} below, for analogues of $w$ with the cup in all positions of the cups of $c_0^\ast$ to show that $c_0^\ast Ac_j^\ast\cong \Hom_A(Ac_0^\ast,Ac_j^\ast)$. The conclusion then follows again from the dimensions.

\subsection{Faithfulness of the covers}
We define $x\in A_n$ as
\begin{equation*}\label{Defx}
\begin{tikzpicture}[scale=0.9,thick,>=angle 90]
\begin{scope}[xshift=8cm]
\node at (1.4,0.5) {$x:=$};
\draw (2.9,0) -- +(0,1);
\draw (3.5,0) to [out=90,in=-180] +(.3,.3) to [out=0,in=90] +(.3,-.3);
\draw (4.7,0) to [out=90,in=-180] +(.3,.3) to [out=0,in=90] +(.3,-.3);
\draw [dotted] (5.6,.1) -- +(1,0);
\draw (6.9,0) to [out=90,in=-180] +(.3,.3) to [out=0,in=90] +(.3,-.3);
\draw (8.2,0) to [out=147,in=-33] +(-4.7,1);
\draw (4.1,1) to [out=-90,in=-180] +(.3,-.3) to [out=0,in=-90] +(.3,.3);
\draw (5.2,1) to [out=-90,in=-180] +(.3,-.3) to [out=0,in=-90] +(.3,.3);
\draw [dotted] (6.1,.9) -- +(1,0);
\draw (7.5,1) to [out=-90,in=-180] +(.3,-.3) to [out=0,in=-90] +(.3,.3);
\node at (9.5,0.5) {$+$};
\draw (10.9,0)  to [out=70,in=-110] +(1.2,1);
\draw (11.5,0) to [out=90,in=-180] +(.3,.3) to [out=0,in=90] +(.3,-.3);
\draw (12.7,0) to [out=90,in=-180] +(.3,.3) to [out=0,in=90] +(.3,-.3);
\draw [dotted] (13.6,.1) -- +(1,0);
\draw (14.9,0) to [out=90,in=-180] +(.3,.3) to [out=0,in=90] +(.3,-.3);
\draw (16.2,0) to [out=136,in=-44] +(-3.5,1);
\draw (10.9,1) to [out=-90,in=-180] +(.3,-.3) to [out=0,in=-90] +(.3,.3);
\draw (13.2,1) to [out=-90,in=-180] +(.3,-.3) to [out=0,in=-90] +(.3,.3);
\draw [dotted] (14.1,.9) -- +(1,0);
\draw (15.5,1) to [out=-90,in=-180] +(.3,-.3) to [out=0,in=-90] +(.3,.3);
\end{scope}
\end{tikzpicture}
\end{equation*}
One checks easily that the kernel of the epimorphism~$Ac_2^\ast\tto Ac_0^\ast$, given by~$a\mapsto ac_0^\ast$ for any~$a\in Ac_2^\ast$, is generated by~$x\in Ac_2^\ast$. This implies the following lemma.
\begin{lemma}\label{Lemx}
We have a short exact sequence
$$0\to Ax\to Ac_2^\ast\to Ac_0^\ast \to 0.$$
\end{lemma}
As a special case of Lemma~\ref{NewLemmaNew}, we have
\begin{lemma}\label{AxAc}
If $n\ge 4$, we have $$\Hom_A(Ac_0^\ast,A/Ac_0^\ast A)=0.$$
\end{lemma}

\begin{lemma}\label{Leminc}
For any~$\lambda\in\Lambda_C$, with~$|\lambda|=i$, we have a monomorphism of~$A$-modules
$$e_n^\ast \overline{\Delta}(\lambda)\hookrightarrow Ac_i^\ast/Ac_{i-2}^\ast Ac_i^\ast.$$
The cokernel has a filtration with sections of the form $e_n^\ast \overline{\Delta}(\mu)$ with~$\mu\in \Par^p_i$.
\end{lemma}
\begin{proof}
Since $He_i^\ast\cong\mk\mS_i$ is a group algebra and therefore self-injective, we have an inclusion~$L^0(\lambda)\hookrightarrow \mk\mS_i$ for each $\lambda\in\Par^p_i$.
Equation~\eqref{defDover} and Lemma~\ref{LemPropB}(4) thus imply that we have an inclusion
\begin{equation*}\overline{\Delta}(\lambda)\hookrightarrow C\otimes_B He_i^\ast,\end{equation*}
where the cokernel has a filtration with sections of the form $\overline{\Delta}(\mu)$ with~$\mu\in \Par^p_i$.
The conclusion then follows from the exactness of the functor $e_n^\ast-$ and Corollary~\ref{CorDeltaA}.
\end{proof}

\begin{prop}\label{PropHom}
If $n\ge 4$, the evaluation of the unit $\eta_{\overline{\Delta}(\lambda)}$ is an isomorphism for all~$\lambda\in\Lambda_C$.\end{prop}
\begin{proof}
By Lemma~\ref{lemeta}(2) and the isomorphism in equation~\eqref{eqinjiso}, it suffices to prove that 
$$e_0^\ast \overline{\Delta}(\lambda)\;\to\; \Hom_{A}(Ac_0^\ast,e_n^\ast \overline{\Delta}(\lambda)),\qquad v\mapsto \alpha_v,\qquad\mbox{with }\;\alpha_v(c_0^\ast)=a_0v$$
is an isomorphism. For $\lambda= \varnothing$, we have $\overline{\Delta}(\varnothing )\cong Ce_0^\ast$. By equation~\eqref{enD0}, the above morphism for this case can be rewritten as
$$ e_0^\ast Ce_0^\ast \to \Hom_{A}(Ac_0^\ast,Ac_0^\ast),\quad v\mapsto\beta_v\quad\mbox{with}\quad \beta_v(c_0^\ast)=a_0vb_0.$$
This is precisely the isomorphism in equation~\eqref{isoij}.

For $\lambda\vdash i$ with~$i>0$, we have $e_0^\ast \overline{\Delta}(\lambda)=0$, so it suffices to show that 
\begin{equation}\label{eqmb0}\Hom_{A}(Ac_0^\ast,e_n^\ast \overline{\Delta}(\lambda))=0,\qquad\forall i>0.
\end{equation} By Lemmata~\ref{Lemx} and~\ref{Leminc} we have monomorphisms
$$\Hom_A(Ac_0^\ast, e_n^\ast \overline{\Delta}(\lambda))\hookrightarrow \Hom_A(Ac_2^\ast, e_n^\ast \overline{\Delta}(\lambda) )\hookrightarrow  \Hom_A(Ac_2^\ast, Ac_i^\ast/Ac_{i-2}^\ast Ac_i^\ast).$$
As~$c_2^\ast$ is an idempotent included in~$Ac_{i-2}^\ast A$ when~$i>2$, this implies equation~\eqref{eqmb0} for $i>2$. 

Finally, we focus on~$i=2$.
By Lemma~\ref{Leminc} we have an inclusion
$$\Hom_A(Ac_0^\ast,e_n^\ast\overline{\Delta}(\lambda))\hookrightarrow \Hom_A(Ac_0^\ast,Ac_2^\ast/Ac_0^\ast Ac_2^\ast).$$
The right term is zero by Lemma~\ref{AxAc}. So \eqref{eqmb0} holds true, which concludes the proof.\end{proof}

\begin{lemma}\label{LemHomAx}
If $n>4$, we have
$$\Hom_A(Ax,Ac_0^\ast)\;=\;0\;=\;\Hom_A(Ax,Ac_{2j}^\ast/Ac_{2j-2}^\ast Ac_{2j}^\ast),\quad\mbox{for}\; j\ge 2.$$
\end{lemma}
\begin{proof}
We introduce diagrams $y_1,y_2\in A$ as
$$
\begin{tikzpicture}[scale=1,thick,>=angle 90]
\begin{scope}[xshift=4cm]

\node at (7,0.5) {$y_1:=$};

\draw (7.9,0) -- +(0,1);
\draw (8.5,0) -- +(0,1);
\draw (9.1,0) to [out=60,in=-110] +(1.2,1);
\draw (9.7,0) to [out=90,in=-180] +(.3,.3) to [out=0,in=90] +(.3,-.3);
\draw [dotted] (10.6,.1) -- +(1,0);
\draw (12,0) to [out=90,in=-180] +(.3,.3) to [out=0,in=90] +(.3,-.3);
\draw (13.2,0) to [out=120,in=-60] +(-2.3,1);

\draw (9.1,1) to [out=-90,in=-180] +(.3,-.3) to [out=0,in=-90] +(.3,.3);
\draw [dotted] (12.14,.9) -- +(0.4,0);
\draw (12.6,1) to [out=-90,in=-180] +(.3,-.3) to [out=0,in=-90] +(.3,.3);
\draw (11.5,1) to [out=-90,in=-180] +(.3,-.3) to [out=0,in=-90] +(.3,.3);

\node at (15,0.5) {$y_2:=$};

\draw (15.9,0) to [out=60,in=-110] +(1.2,1);
\draw (16.5,0) to [out=60,in=-110] +(1.2,1);
\draw (17.1,0) to [out=60,in=-110] +(1.2,1);
\draw (17.7,0) to [out=90,in=-180] +(.3,.3) to [out=0,in=90] +(.3,-.3);
\draw [dotted] (18.6,.1) -- +(1,0);
\draw (20,0) to [out=90,in=-180] +(.3,.3) to [out=0,in=90] +(.3,-.3);
\draw (21.2,0) to [out=120,in=-60] +(-2.3,1);

\draw (15.9,1) to [out=-90,in=-180] +(.3,-.3) to [out=0,in=-90] +(.3,.3);
\draw [dotted] (20.14,.9) -- +(0.4,0);
\draw (20.6,1) to [out=-90,in=-180] +(.3,-.3) to [out=0,in=-90] +(.3,.3);
\draw (19.5,1) to [out=-90,in=-180] +(.3,-.3) to [out=0,in=-90] +(.3,.3);

\end{scope}
\end{tikzpicture}
$$
We have $(y_1+y_2)x=-x$ and $c_4^\ast x=x$. If $a\in Ac_0^\ast$ is the image of~$x$ under a morphism~$Ax\to Ac_0^\ast$, then we must have $(y_1+y_2)a=-a$ and $c_4^\ast a=a$. This implies that~$a$ is the sum of diagrams which contain all the $n/2-2$ cups of~$c_4^\ast$ and either a cup connecting dots $1$ and $2$ (as in~$y_2$), or a cup connecting dots $3$ and $4$ (as in~$y_1$). As all diagrams in~$Ac_0^\ast$ contain~$n$ cups it follows that~$a$ is proportional to~$c_0^\ast$. However, we have $y_2c_0^\ast=c_0^\ast=-y_1c_0^\ast$. This implies that~$a=0$, proving the left equation.

Now we consider~$a\in Ac_{2j}^\ast$ such that~$a+ Ac_{2j-2}^\ast Ac_{2j}^\ast$ is the image of~$x$ under a morphism from $Ax$. Again we can assume $(y_1+y_2)a=-a$ and $c_4^\ast a=a$, which implies that~$a$ is a linear combination of diagrams with at least $n/2-1$ cups. Since $j\ge 2$, this implies that~$a\in Ac_{2j-2}^\ast A$.
\end{proof}

\begin{lemma}\label{noExt}
If $n>4$, we have
$$\Ext^1_A(Ac_0^\ast,Ac_{2j}^\ast/Ac_{2j-2}^\ast Ac_{2j}^\ast)=0,\quad\mbox{for}\;\,0\le j\le n/2.$$
\end{lemma}
\begin{proof}
The short exact sequence in Lemma~\ref{Lemx}, implies an exact sequence
\begin{equation}\label{esHE}0\to \Hom_A(Ac_0^\ast,M)\to \Hom_A(Ac_2^\ast,M)\to \Hom_A(Ax,M)\to\Ext^1_A(Ac_0^\ast,M)\to 0,\end{equation}
for an arbitrary $A$-module $M$. Lemma~\ref{LemHomAx} hence implies the vanishing of extensions with~$M= Ac_{2j}^\ast/Ac_{2j-2}^\ast Ac_{2j}^\ast$ for all $j\not=1$. We thus focus on~$M:= Ac_2^\ast/Ac_0^\ast Ac_2^\ast$.

For this $M$, the left space in~\eqref{esHE} is zero by Lemma~\ref{AxAc}, the second space from the left is isomorphic to~$c_2^\ast \left(A/Ac_0^\ast A\right)c_2^\ast\cong\mk\mS_2$, which has dimension 2. Using the same reasoning as in the proof of Lemma~\ref{LemHomAx}, the dimension of the third space is bounded by the dimension of the space of elements~$v$ in~$Ac_2^\ast/Ac_0^\ast Ac_2^\ast$ which satisfy $c_4^\ast v=v$ and $(y_1+y_2)v=-v$, which is also $2$. This shows that the extension vanishes, which concludes the proof.
\end{proof}

We denote the first right derived functor of the left exact functor $G$ in~\eqref{eqG} by $\cR_1G$.

\begin{prop}\label{PropExt}
For $n>4$, we have $\cR_1G\circ F(\overline{\Delta}(\lambda))=0$ for all $\lambda\in\Lambda_C$.
\end{prop}
\begin{proof}
Equation~\eqref{eqG} and Lemma~\ref{LemECA} imply that for any~$C$-module $N$ we have
$$\cR_1G\circ F(N)=0\quad\Leftrightarrow\quad \Ext^1_A(Ac_0^\ast,e_n^\ast N)=0.$$
We consider~$N=\overline{\Delta}(\lambda)$. For $\lambda=\varnothing$, the extension vanishes by equation~\eqref{enD0} and Lemma~\ref{noExt}. Consider $\lambda\vdash i$ with~$i>0$, Lemma~\ref{Leminc} yields an exact sequence
$$\Hom_A(Ac_0^\ast,K)\to \Ext^1_A(Ac_0^\ast,e_n^\ast \overline{\Delta}(\lambda))\to \Ext^1_A(Ac_0^\ast,Ac_i^\ast/Ac_{i-2}^\ast Ac_i^\ast), $$
where~$K$ has a filtration with sections of the form $e_n^\ast \overline{\Delta}(\mu)$ with~$\mu\vdash i$. The left-hand space is zero by equation~\eqref{eqmb0}, the right-hand space is zero by Lemma~\ref{noExt}.
Hence the middle space is zero, which concludes the proof.\end{proof}

The conclusion in Theorem~\ref{ThmMor}(2) then follows from Propositions~\ref{PropHom} and~\ref{PropExt}, precisely as in the proof of \cite[Theorem~9.2.2]{Borelic}. %Alternatively one can use the principle in \cite[Proposition~2.3]{Nakano}.

\subsection{Projective and cell modules}
By \cite[Proposition~3.5]{Yang} and Theorem~\ref{ThmCp}(1), $A$ is standardly based for poset $(\mathbf{L},\unlhd)$, with cell modules
\begin{equation}\label{eqWA}W_A(\lambda)\;=\;F(W_C(\lambda))\;=\;e_n^\ast C\otimes_B W^0(\lambda),\end{equation}
for any~$\lambda\in \mathbf{L}_C$. As none of these modules is zero we have $\mathbf{L}_A=\mathbf{L}_C$.
This proves Theorem~\ref{ThmSB}(2).

\subsubsection{}The cell modules of~$C$ are included in~$\cF(\overline{\Delta})$, by definitions in~\ref{SecDefMod} and Lemma~\ref{LemPropB}(4).
If $n>4$, Theorem~\ref{ThmMor}(2) thus implies that the cell modules of~$A_n$ form a standard system if and only if the cell modules of~$C_n$ form a standard system. Theorem~\ref{ThmCp}(2) thus implies that for $n>4$, the cell modules form a standard system if and only if~$\charr(\mk)\not\in\{2,3\}$. For $n\in\{2,4\}$, Corollaries~\ref{CorA2}(1) and~\ref{CorA4}(2) below imply that the cell modules do not form a standard system.
The fact that multiplicities in filtrations with sections given by modules forming a standard system are well-defined follows immediately from~\cite{DR}. This concludes the proof of \ref{ThmSB}(3).

%As a special case of Lemma~\ref{LemTriv} we have the following.
%\begin{lemma}\label{LemPCA}
%For all $\lambda\in\Lambda_A$, we have

%\end{lemma}
%\begin{proof}
%As~$|\lambda|\not=0$, we find that~$e_n^\ast P_C(\lambda)$ is projective by Lemma~\ref{LemECA}. 
%Furthermore, by equation~\eqref{eqCorLambda0}, for any~$\mu\in\Lambda_A$, we have
%$$\Hom_A(FP_C(\lambda), L_A(\mu))\;\cong\; \Hom_C(P_C(\lambda), GL_A(\mu)).$$
%As~$FG\cong \Id$ and hence $FGL_A(\mu)\cong L_A(\mu)$. We find that~$GL_A(\mu)$ is a module with~$[GL_A(\mu):L_C(\mu)]=1$ and $[GL_A(\mu):L_C(\lambda)]=0$ for any other $\lambda\in \Lambda_A$. This concludes the proof.
%\end{proof}

\begin{lemma}\label{LemBGGA}
If $p\not\in [2,n]$ and $\lambda\in\Lambda_A$, the module $P_A(\lambda)$ has a filtration with sections given by cell modules with multiplicities
$$(P(\lambda):W(\mu))\,=\, [W(\mu^t ):L(\lambda^t )],\qquad\mbox{for $\mu\in \mathbf{L}_A$}.$$
\end{lemma}
\begin{proof}
We can apply the exact functor $F$ to the filtration of $P_C(\lambda)$ in equation~\eqref{eqBGGC}. The conclusion follows from equations~\eqref{eqWA} and~\eqref{eqCorLambda0}.
\end{proof}
The above lemma implies in particular Theorem~\ref{ThmSB}(4).
%As~$F$ is exact, by \ref{Maschke} and equations~\eqref{eqWA} and~\eqref{eqCorLambda0}, we have
 %\begin{equation}\label{eqWDeltaL}
 %[W_A(\lambda):L_A(\mu)]=[\Delta_C(\lambda):L_C(\mu)],\quad\mbox{ for }\lambda\in \mathbf{L}_A\;\mbox{ and }\;\mu\in\Lambda_A,\quad\mbox{if }\,\charr(\mk)\not\in [2,n].
 %\end{equation}
%Theorem~\ref{ThmSB}(4) then follows from the reciprocity relation in Theorem~\ref{ThmQH} and Corollary~\ref{multAC2}.

\begin{lemma}\label{LemConnectP}
Consider $\lambda\in \Par_i^p$ with~$i\in\JJJ^0(n)$. The idempotent
$a_i e^{(i)}_\lambda b_i\in A_n,$
with~$e_\lambda^{(i)}$ a primitive idempotent in~$A_i=e_i^\ast C e_i^\ast$ corresponding to~$L_{A_i}(\lambda)$ and $a_i,b_i$ as defined in~\ref{Defab}, is primitive and corresponds to $L_{A_n}(\lambda)$.
\end{lemma}
\begin{proof}
It is obvious that $e^{(i)}_\lambda\in A_i=e_i^\ast C_ne_i^\ast$ remains primitive as an idempotent in $C_n$.
By construction, the isomorphism $e_n^\ast C_n e_i^\ast\cong A_n c_i^\ast$ of~Lemma~\ref{LemECA} restricts to
$$e_n^\ast Ce_\lambda^{(i)}\cong A_n a_i e_\lambda^{(i)} b_i.$$
This module is isomorphic to $P_{A_n}(\lambda)$, by equation~\eqref{eqCorLambda0}, which concludes the proof.
\end{proof}

%%%%%%%%%%%%%%%%%%%%%%%%%%%%%%%%%%%%%%%%%%%%%%%%%%%%%%%%%%%%%%%%%%%%%%%%

\section{The Bratteli diagram and Murphy bases}\label{SecBasis}
We construct the Bratteli diagram related to the sequence of standardly based algebras given by the inclusion of the periplectic Brauer algebras in~\ref{embed}. The analogue for the Brauer algebras can be found in~\cite[\S6]{LeducRam} or \cite[Proposition~2.7]{blocks}. This will allow us to construct a Murphy basis for each cell module of~$A_n$, similarly to the case of the Brauer algebra in~\cite[\S 9]{Enyang} or the Iwahori-Hecke algebra in~\cite[\S 3.2]{Mathas}. We could use this construction to prove the standardly based structure of~$A_n$ more directly than in Section~\ref{SecAC}, but we do not pursue this.

\subsection{Bratteli diagrams} First we introduce our Bratteli diagram, and some general terminology.
\subsubsection{}\label{SecBraPro}Consider a chain 
$$R_1\subset R_2\subset \cdots\subset R_i\subset R_{i+1}\subset \cdots$$ of standardly based algebras, with~$R_1=\mk$. Assume that the restriction of an arbitrary cell module of~$R_{i+1}$ to~$R_{i}$ admits a filtration with sections given by~$R_i$-cell modules, such that no cell module of~$R_i$ appears more than once for a fixed cell module of~$R_{i+1}$. Corresponding to that chain of algebras (and that choice of filtrations) we can then define a multiplicity free Bratteli diagram as follows. On the $i$th row we place the elements of~$\mathbf{L}_{R_i}$, these are the vertices of the diagram. Then we draw an edge between an element of~$\mathbf{L}_{R_i}$ and one of~$\mathbf{L}_{R_{i+1}}$ if the corresponding cell module of~$R_i$ appears in the filtration of the restriction of the cell module of~$R_{i+1}$.

\subsubsection{}
The set of vertices on row $i$ in the Bratteli diagram which appears in the study of BMW and Brauer algebras in~\cite{Enyang, LeducRam} is $\mathbf{L}_{A_i}=\mathbf{L}_{C_i}$. The edges in the Bratteli diagram are then given by connecting any partition~$\lambda$ on the $i$th row with all partitions on the $i+1$-th row which are in~$\RR(\lambda)\sqcup\AAAA(\lambda)$, see~\ref{DefAR}. The top part of this diagram diagram, see \cite[Figure 2]{LeducRam}, is given by
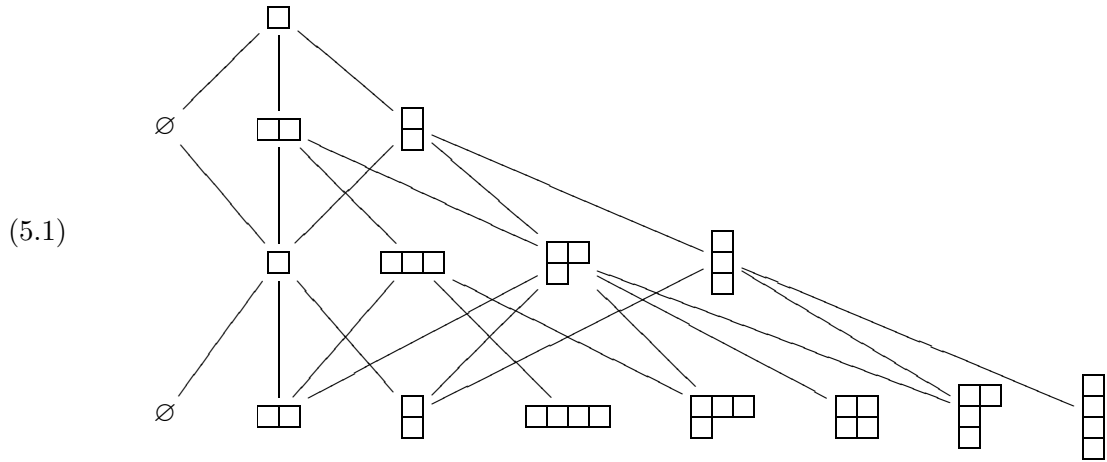
\begin{align}\label{Brat}
\begin{matrix}
\xymatrix{ 
 & \text{\tiny\Yvcentermath1$\yng(1)$}\ar@{-}[dl]\ar@{-}[d]\ar@{-}[dr] & & & \\
 \varnothing\ar@{-}[dr] & \text{\tiny\Yvcentermath1$\yng(2)$}\ar@{-}[d]\ar@{-}[dr]\ar@{-}[drr]&
 \text{\tiny\Yvcentermath1$\yng(1,1)$}\ar@{-}[dl]\ar@{-}[dr]\ar@{-}[drr]& & \\
& \text{\tiny\Yvcentermath1$\yng(1)$}\ar@{-}[dl]\ar@{-}[d]\ar@{-}[dr] &
\text{\tiny\Yvcentermath1$\yng(3)$}\ar@{-}[dl]\ar@{-}[dr]\ar@{-}[drr]&
 \text{\tiny\Yvcentermath1$\yng(2,1)$}\ar@{-}[dll]\ar@{-}[dl]\ar@{-}[dr]\ar@{-}[drr]\ar@{-}[drrr]
 & \text{\tiny\Yvcentermath1$\yng(1,1,1)$}\ar@{-}[dll]\ar@{-}[drr]\ar@{-}[drrr]\\
  \varnothing & \text{\tiny\Yvcentermath1$\yng(2)$}&
 \text{\tiny\Yvcentermath1$\yng(1,1)$}& \text{\tiny\Yvcentermath1$\yng(4)$} & \text{\tiny\Yvcentermath1$\yng(3,1)$} &\text{\tiny\Yvcentermath1$\yng(2,2)$} &\text{\tiny\Yvcentermath1$\yng(2,1,1)$}&\text{\tiny\Yvcentermath1$\yng(1,1,1,1)$}}
\end{matrix}
\end{align}

\subsubsection{}\label{SecPath}
A {\em path} in the Bratteli diagram is a sequence of vertices $\mathfrak{t}=(\mathfrak{t}^{(1)},\mathfrak{t}^{(2)},\cdots, \mathfrak{t}^{(k)})$ such that~$\mathfrak{t}^{(l)}$ is on the $l$th row and there is an edge between~$\mathfrak{t}^{(l)}$ and $\mathfrak{t}^{(l+1)}$.
Denote the set of all paths in the Bratteli diagram \eqref{Brat} ending in a partition~$\lambda$ on the $i$th row by  $\St_i(\lambda)$. Note that in case $\lambda\vdash i$, the set $\St_i(\lambda)$ for diagram~\ref{Brat} also corresponds to a set of paths in the Young lattice. In this case, $\St_i(\lambda)$ can thus be identified with the set $\St(\lambda)$ of standard tableaux of shape $\lambda$, see \cite[\S4]{LeducRam}, justifying the notation. For any~$\mathfrak{t}\in \St_i(\lambda)$, we let $\mathfrak{t}'\in\St_{i-1}(\mathfrak{t}^{(i-1)})$ be defined by~$\mathfrak{t}=(\mathfrak{t}',\lambda)$. 

An example of a path in~\eqref{Brat} is given as
$$(\text{    \tiny{$\yng(1)$}        },\text{    \tiny{$\yng(2)$}        },\text{    \tiny{$\yng(1)$}        },\text{    \tiny{$\yng(1,1)$}        })\;\in\; \St_4(\text{    \tiny{$\yng(1,1)$}        })$$

\subsubsection{}Each row in the Bratteli diagram~\eqref{Brat} is equipped with the partial order~$\unlhd$ in equation~\eqref{BigDO}. We have drawn the Bratteli diagram in such a way that if~$ \lambda\rhd\mu$, then~$\lambda$ appears to the left of~$\mu$.

For a fixed partition~$\lambda$ on the $n$th row we introduce a lexicographic partial order~$\unlhd$ on the paths $\St_n(\lambda)$, by setting $\mathfrak{t}\lhd \mathfrak{s}$ if and only if there is a $k$, with~$1\le k\le n$, such that
$$\mathfrak{t}^{(k)}\lhd \mathfrak{s}^{(k)}\qquad\mbox{and}\qquad \mathfrak{t}^{(j)}=\mathfrak{s}^{(j)},\mbox{ for all }k<j\le n.$$
The partial order can equivalently be defined by iteration, for $\mathfrak{t},\mathfrak{s}\in\St_n(\lambda)$, we have
$$
\mathfrak{t}\lhd \mathfrak{s}\quad\Leftrightarrow\quad\begin{cases}\mathfrak{t}^{(n-1)}\lhd \mathfrak{s}^{(n-1)},&\mbox{or}\\ \mathfrak{t}^{(n-1)}= \mathfrak{s}^{(n-1)}\;\mbox{ and }\; \mathfrak{t}'\lhd \mathfrak{s}'.\end{cases}
$$

When~$|\lambda|=n$, the order~$\unlhd$ on~$\St_n(\lambda)=\St(\lambda)$ does not reduce to the partial order on standard $\lambda$-tableaux in~\cite[\S 3.1]{Mathas}, but to an extension of that order.
\subsection{Restriction of cell modules}\label{SecRes} We make no assumption on $\charr(\mk)$. We will use the short-hand notation~$W_n(\lambda):=W_{A_n}(\lambda)$, with~$\lambda\in\mathbf{L}_A$, for the cell modules of $A_n$.

\begin{thm}\label{ThmRes}
The restriction of the cell modules of $A_n$ to~$A_{n-1}$ have filtrations which lead to the Bratteli diagram~\eqref{Brat} under the procedure in \ref{SecBraPro}.

% with sections given by cell modules of~$A_{n-1}$ such that the multiplicities correspond to .
\end{thm}

\subsubsection{}\label{SecNni}By equation~\eqref{eqWA} and Corollary~\ref{CorTriang}, for $\lambda\vdash i$, we have
$$W_n(\lambda)\;=\; e_n^\ast \overline{N} e_i^\ast \otimes W^0(\lambda),$$
as a vector space. To describe the restriction to~$A_{n-1}$, we will decompose $e_n^\ast \overline{N}e_i^\ast$. We write
$$e_n^\ast \overline{N}e_i^\ast\;=\; N_{n-1,i-1}\;\oplus N_{n-1,i+1},$$
where~$N_{n-1,i-1}$, resp. $N_{n-1,i+1}$, is spanned by all diagrams having a propagating line, resp. cup, ending in the right-most dot on the upper line. We have an isomorphism of vector spaces 
\begin{equation}\label{isomni-1}e_{n-1}^\ast \overline{N} e_{i-1}^\ast\;\stackrel{\sim}{\to}\; N_{n-1,i-1},\quad \alpha\mapsto \alpha\otimes I.\end{equation}
Define $a\in e_{i+2}^\ast \overline{N}e_i^\ast$ as 
\begin{equation*}
\begin{tikzpicture}[scale=0.6,thick,>=angle 90]
\begin{scope}[xshift=8cm]
\node at (0,0.5) {$a:=$};
\draw (.7,0) -- +(0,1);
\draw (1.3,0) -- +(0,1);
\draw [dotted] (1.6,.5) -- +(1,0);
\draw (2.9,0) -- +(0,1);
\draw (3.5,1) to [out=-90,in=-180] +(.3,-.3) to [out=0,in=-90] +(.3,.3);
\end{scope}
\end{tikzpicture}
\end{equation*} Any diagram $d\in N_{n-1,i+1}$ can be uniquely decomposed as
\begin{equation}\label{Defd1d2}d=(d^{(1)}\otimes I)\circ (d^{(2)}\otimes I)\circ a,\quad \mbox{with}\quad d^{(1)}\in e_{n-1}^\ast \overline{N}e^\ast_{i+1}\;\mbox{ and }\; d^{(2)}\in \mS_{i+1}=e_{i+1}^\ast He_{i+1}^\ast.\end{equation}
By construction, $d^{(2)}$ is a shortest representative of the coset $\mS_{i+1}/(\mS_i\times \mS_1)$. Note also that~$d^{(1)}$ is either a diagram, or a diagram multiplied with~$-1$. It can be easily observed graphically that this actually yields an isomorphism of vector spaces
$$N_{i-1,i+1}\;\stackrel{\sim}{\to}\; e_{n-1}^\ast \overline{N}e^\ast_{i+1}\otimes \mk  \mS_{i+1}/\mk\mS_i, \quad d\mapsto d^{(1)}\otimes d^{(2)}.$$

For all $k\in\mZ_{>0}$, we consider the exact functor 
$$I_k\;=\; e_k^\ast C_k\otimes_B-:\; H\mbox{-mod}\to A_k\mbox{-mod},$$
where modules of $H\cong\oplus_{i\in\JJJ(k)}\mk\mS_i$ are interpreted as $B$-modules with trivial $B_{\plus}$-action. In particular, equation~\eqref{eqWA} implies $W(\lambda)=I_k(W^0(\lambda))$, for all $\lambda\in\LL_{A_k}$.
If $\lambda\vdash i=n$, we assume the convention~$I_{n-1}(\ind^{\mk\mS_{n+1}} W^0(\lambda))=0$.
\begin{lemma}\label{LemRes}
The restriction to~$A_{n-1}$ of $W_n(\lambda)$, with $\lambda\vdash i\in\JJJ(n)$, satisfies a short exact sequence
$$0\;\to\; I_{n-1} (\res_{\mk\mS_{i-1}} W^0(\lambda))\;\to\; \res_{A_{n-1}}W_n(\lambda)\;\to\;I_{n-1}(\ind^{\mk\mS_{i+1}} W^0(\lambda)) \;\to\; 0.$$
\end{lemma}
\begin{proof}
The subspace $N_{n-1,i-1}\otimes W^0(\lambda)$ of~$W_n(\lambda)$ constitutes an~$A_{n-1}$-submodule. Using the isomorphism in~\eqref{isomni-1}, this submodule is easily identified with
the left term in the short exact sequence.

Now we use~\eqref{Defd1d2} to construct the vector space isomorphism
\begin{equation}\label{eqddd}N_{n-1,i+1}\otimes W^0(\lambda)\;\stackrel{\sim}{\to}\; e_{n-1}^\ast \overline{N}e_{i+1}^\ast \otimes \ind^{\mk\mS_{i+1}}W^0(\lambda),\quad d\otimes v\mapsto d^{(1)}\otimes d^{(2)}v,\end{equation}
where~$v\in W^0(\lambda)$ is interpreted as $1\otimes v\in \mk\mS_{i+1}\otimes_{\mk\mS_i}W^0(\lambda)$. Note that $\mk\mS_{i+1}$ is free as a (right) $\mk\mS_i$-module.
This corresponds to the proposed isomorphism from the quotient of~$W_n(\lambda)$ with respect to the submodule of the previous paragraph.\end{proof}

As things simplify greatly, we will first briefly focus on the case $\charr(\mk)\not\in[2,n]$, before moving on to arbitrary characteristic. We assume $W_{n-1}(\nu)=0$ when $|\nu|>n-1$.
\begin{cor}\label{CorSS}
Assume $\charr(\mk)\not\in[2,n]$, we have a short exact sequence
$$0\;\to\;\bigoplus_{\mu\in \RR(\lambda)}W_{n-1}(\mu)\;\to\;\res_{A_{n-1}}W_n(\lambda)\;\to\; \bigoplus_{\nu\in\AAAA(\lambda)}W_{n-1}(\nu)\;\to\;0.$$
\end{cor}
\begin{proof}
This follows immediately from Lemma~\ref{LemRes} and the description of the restriction and induction in~\ref{SecLR}.
\end{proof}

\subsubsection{} If we fix $n$ and we assume that~$\charr(\mk)\not\in[2,n]$, we can define the Murphy basis of~$A_n$ easily, by iteration. For $A_1=\mk$ we let $v_{\mathfrak{t}}$, with~$\mathfrak{t}=(\Box)$, be any non-zero vector. Assume that we have defined a basis $\{v_{\mathfrak{t}}\,|\, \mathfrak{t}\in \St_{n-1}(\mu)\}$ of~$W_{n-1}(\mu)$ for any~$\mu\in \mathbf{L}_{A_{n-1}}$.

Now fix $\lambda\in \mathbf{L}_{A_n}$. For any~$\mathfrak{t}\in \St_n(\lambda)$ with~$|\mathfrak{t}^{(n-1)}|<|\lambda|$, we let $v_{\mathfrak{t}}\in W_n(\lambda)$ be the image of~$v_{\mathfrak{t}'}\in W_{n-1}(\mathfrak{t}^{(n-1)})$ under the $A_{n-1}$-monomorphism in Corollary~\ref{CorSS}. Similarly, if~$|\mathfrak{t}^{(n-1)}|>|\lambda|$, we define $v_{\mathfrak{t}}\in W_n(\lambda)$ to be the unique vector in~$N_{n-1,i+1}\otimes W^0(\lambda)$, which is in the preimage of~$v_{\mathfrak{t}'}\in W_{n-1}(\mathfrak{t}^{(n-1)})$ under the $A_{n-1}$-epimorphism in Corollary~\ref{CorSS}.

\subsubsection{}\label{SecRevS} Now we return to arbitrary $\charr(\mk)$. We formulate the result in~\cite[Theorem~9.3]{James} or \cite[Proposition~6.1]{Mathas} in a way which will be convenient later. For any~$\lambda\vdash i$, we have a vector space decomposition
$$W^0(\lambda)\;=\;\bigoplus_{\mu\in\RR(\lambda)}W^0(\lambda)^{\mu}.$$
We set 
$$W^0(\lambda)^{\widehat\mu}:=\bigoplus_{\kappa\rhd\mu} W^0(\lambda)^{\kappa}.$$ Then both $W^0(\lambda)^{\widehat\mu}$ and $W^0(\lambda)^{\widehat\mu}\oplus W^0(\lambda)^{\mu}$ are $\mk\mS_{i-1}$-submodules of~$W^0(\lambda)$ and 
$$(W^0(\lambda)^{\widehat\mu}\oplus W^0(\lambda)^{\mu})/W^0(\lambda)^{\widehat\mu}\;\cong\; W^0(\mu).$$
Similarly, by \cite[\S 17]{James}, we have a vector space decomposition
$$\ind^{\mk\mS_{i+1}}W^0(\lambda)\;=\;\bigoplus_{\nu\in\AAAA(\lambda)}IW^0(\lambda)^{\nu}, $$
such that, with similar notation, we have $\mk\mS_{i+1}$-module isomorphisms
\begin{equation}\label{isoIW}(IW^0(\lambda)^{\widehat\nu}\oplus IW^0(\lambda)^{\nu})/IW^0(\lambda)^{\widehat\nu}\;\cong\; W^0(\nu).\end{equation}

\begin{proof}[Proof of Theorem~\ref{ThmRes}]
As the functor $I_{n-1}$ is exact, the statement follows immediately from Lemma~\ref{LemRes} and the restriction and induction of Specht modules in~\ref{SecRevS}.
\end{proof}

\subsection{The Murphy bases for $A_n$}
Fix an arbitrary $\mk$. For each $\lambda\in \mathbf{L}_{A_n}$, we will fix a basis 
$$\{v_{\mathfrak{t}}\;|\; \mathfrak{t}\in\St_n(\lambda)\}\qquad\mbox{of}\qquad W_n(\lambda),$$
for $W_n(\lambda)$ the cell module over~$A_n$. This basis will be defined iteratively. For $A_1$ we take an arbitrary non-zero vector of the trivial module.
\subsubsection{}
First we define a vector space decomposition
$$W_n(\lambda)\;=\;\bigoplus_{\mu\in \RR(\lambda)}W_{n}(\mu)^{\mu}\;\oplus\;\bigoplus_{\nu\in\AAAA(\lambda)}W_{n}(\lambda)^{\nu}.$$
For $\mu\in\RR(\lambda)$, we set
$$W_{n}(\lambda)^{\mu}:=N_{n-1,i-1}\otimes W^0(\lambda)^{\mu}.$$
For $\nu\in\AAAA(\lambda)$, we set
\begin{equation}\label{eqCompl}W_n(\lambda)^{\nu}:= (e_{n-1}^\ast \overline{N}e_{i+1}^\ast)\otimes IW^0(\lambda)^\nu,\end{equation}
where the equality should be interpreted under isomorphism~\eqref{eqddd}. For any~$\mu\in \RR(\lambda)\sqcup \AAAA(\lambda)$ we set
$$W_n(\lambda)^{\widehat\mu}=\bigoplus_{\kappa\in \RR(\lambda)\sqcup \AAAA(\lambda)\,|\, \kappa\rhd\mu}W_n(\lambda)^{\kappa}.$$
By Section~\ref{SecRes} and exactness of $I_{n-1}$, we have an isomorphism of~$A_{n-1}$-modules
\begin{equation}\label{isoMB}(W_n(\lambda)^{\widehat\mu}\oplus W_n(\lambda)^{\mu})/W_n(\lambda)^{\mu}\;\cong \;W_{n-1}(\mu),\qquad \mbox{for $\mu\in \RR(\lambda)\sqcup\AAAA(\lambda)$}.\end{equation}

\subsubsection{} For any~$\mathfrak{t}\in\St_n(\lambda)$, the Murphy basis element $v_\mathfrak{t}$ is defined as the unique element of~$W_n(\lambda)^{\mathfrak{t}^{(n-1)}}$ such that its image under \eqref{isoMB} is $v_{\mathfrak{t}'}$. Consequently, we have
\begin{equation}\label{FPMB}W_n(\lambda)^\mu\;=\;\bigoplus_{\mathfrak{t}\in \St_n(\lambda)\,|\,\mathfrak{t}^{(n-1)}=\mu}\mk v_{\mathfrak{t}}\qquad\mbox{and}\qquad W_n(\lambda)^{\widehat\mu}\;=\;\bigoplus_{\mathfrak{s}\in \St_n(\lambda)\,|\,\mathfrak{s}^{(n-1)}\rhd\mu}\mk v_{\mathfrak{s}}.
\end{equation}

%%%%%%%%%%%%%%%%%%%%%%%%%%%%%%%%%%%%%%%%%%%%%%%%%%%%%%%%%%%%%%%

\section{Jucys-Murphy elements and the centre of~$A_n$}\label{SecJM}
In this section we will construct a ``family of JM elements'' in the sense of~\cite{MathasCrelle}, which is compatible with the Murphy basis. For appropriate characteristic of the ground field, this family will even be `complete' in the sense that it will separate between the different blocks of $A_n$.

\subsection{Definition}\label{SecDefJM} We introduce analogues of the Jucys-Murphy elements of the symmetric group, see {\it e.g.}~\cite[\S3.3]{Mathas}, or of the Brauer algebra, see {\it e.g.}~\cite[\S2]{Naz}.

\subsubsection{} We define the {\em Jucys-Murphy elements} $\{x_i\,|\, 1\le i\le n\}$ in~$A_n$ as $x_1:=0\in A_n$, and
 $$x_i:=\sum_{j=1}^{i-1} (j,i) + \overline{(j,i)}, \quad\mbox{for }\;2\le i\le n.$$
 We will also use the notation~$x^0_i:=\sum_{j=1}^{i-1} (j,i)$ for the corresponding Jucys-Murphy element of~$\mk\mS_n$.
  
 As an example, we depict the non-trivial Jucys-Murphy elements of~$A_3$:
 $$\begin{tikzpicture}[scale=0.8,thick,>=angle 90]
\begin{scope}[xshift=4cm]
%\node at (8.4,0.5) {$x_1=$};
%\draw (9.1,0) -- + (0,1);
%\draw (9.7,0) -- +(0,1);
%\draw (10.3,0) -- +(0,1);
%\node at (10.45,0) {,};

\node at (11.8,0.5) {$x_2=$};
\draw (12.5,0) to [out=65,in=-115] +(0.6,1);
\draw (13.1,0) to [out=115,in=-65] +(-0.6,1);
\draw (13.7,0) -- +(0,1);
\node at (14.3,0.5) {$+$};
\draw (14.9,0 ) to [out=90,in=90] +(0.6,0);
\draw (14.9,1 ) to [out=-90,in=-90] +(0.6,0);
\draw (16.1,0) -- +(0,1);
\node at (16.25,0) {,};

\node at (17.6,0.5) {$x_3=$};
\draw (18.3,0) -- +(0,1);
\draw (18.9,0) to [out=65,in=-115] +(0.6,1);
\draw (19.5,0) to [out=115,in=-65] +(-0.6,1);

\node at (20.1,0.5) {$+$};
\draw (21.3,0 ) to [out=90,in=90] +(0.6,0);
\draw (21.3,1 ) to [out=-90,in=-90] +(0.6,0);
\draw (20.7,0) -- +(0,1);

\node at (22.5,0.5) {$+$};
\draw (23.1,0) to [out=55, in =-125]+(1.2,1);
\draw (24.3,0) to [out=125, in =-55]+(-1.2,1);
\draw (23.7,0) -- +(0,1);

\node at (24.9,0.5) {$+$};
\draw (25.5,0) to [out=90,in=90] +(1.2,0);
\draw (25.5,1) to [out=-90,in=-90] +(1.2,0);
\draw (26.1,0) -- +(0,1);
\node at (26.85,0) {.};
\end{scope}
\end{tikzpicture}
$$
 
 \begin{lemma}\label{LemJM}
 For each $1\le i \le n$, the element $x_i\in A_n$ commutes with every element in~$A_{i-1}\subset A_n$.
  Consequently, the elements $\{x_i\,|\, 1\le i\le n\}$ generate an abelian subalgebra $\Gamma_n$ of~$A_n$.
 \end{lemma}
 \begin{proof}
 The statement is trivial for $i\in\{1,2\}$, so we assume $i>2$.
 By construction, we have $wx_i w^{-1}=x_i$ for all $w\in \mS_{i-1}$. As~$A_{i-1}$ is generated by~$\mS_{i-1}$ and $\varepsilon_1$, it suffices to prove that~$\varepsilon_1x_i=x_i\varepsilon_1$. Clearly $\varepsilon_1$ commutes with~$\sum_{j=3}^{i-1} (j,i) + \overline{(j,i)}$, so we only need to prove that~$\varepsilon_1$ commutes with 
 $$(1,i)+(2,i)+\overline{(1,i)}+\overline{(2,i)}.$$
 This is essentially the claim that~$\varepsilon_1x_3=x_3\varepsilon_1$. A quick calculation shows that we even have
\begin{equation}\label{eq13}\varepsilon_1x_3=0=x_3\varepsilon_1,\end{equation}
 concluding the proof.
 \end{proof}
 
%Denote by $\Gamma_n:=\langle x_2,x_3,\ldots,x_n\rangle$ the commutative subalgebra of~$A_n$ generated by the Jucys-Murphy elements.

 \subsection{Action of the Jucys-Murphy elements on the Murphy basis}

 \subsubsection{} Fix a partition~$\lambda\vdash i$ for $i\in\JJJ(n)$. To any~$\mathfrak{t}\in \St_n(\lambda)$ as in~\ref{SecPath}, we associate a vector
 $$c_{\mathfrak{t}}:=\left(c_{\mathfrak{t}}(2),c_{\mathfrak{t}}(3),\ldots, c_{\mathfrak{t}}(n)\right)\in \mk^{n-1},\qquad\mbox{with}$$
  \begin{equation}\label{defct}c_{\mathfrak{t}}(l):=\begin{cases}\resi(b)& \mbox{if }\; \mathfrak{t}^{(l)}=\mathfrak{t}^{(l-1)}\cup b,\\ \resi(b)+1& \mbox{if }\; \mathfrak{t}^{(l-1)}=\mathfrak{t}^{(l)}\cup b,\end{cases}\qquad  \mbox{for }\;1<l\le n.
 \end{equation}
Here we used the residue of a box $b$ in a Young diagram as in~\ref{SecResi}. In particular, we have $c_{\mathfrak{t}}\in\mathbb{F}^{n-1}_p$ if $\charr(\mk)=p>0$, and $c_{\mathfrak{t}}\in\mathbb{Z}^{n-1}$ if $\charr(\mk)=0$. For the example in \ref{SecPath}, we have $c_{\mathfrak{t}}=(\overline{1},\overline{2},\overline{-1})$.

 \begin{thm}\label{ThmFamJM}
 Let $\mk$ be arbitrary and $W_n(\lambda)$ the cell module over~$A_n$, with~$\lambda\in \mathbf{L}_A$.
  For the Murphy basis $\{v_{\mathfrak{t}}\,|\, \mathfrak{t}\in\St_n(\lambda)\}$ and the Jucys-Murphy elements $\{x_l\,|\, 1< l\le n\}$, we find
 $$x_l v_{\mathfrak{t}}-c_{\mathfrak{t}}(l)v_{\mathfrak{t}}\;\in\; \bigoplus_{\mathfrak{s}\rhd\mathfrak{t}} \mk v_{\mathfrak{s}}.$$
 \end{thm}
 
 We start the proof with the following lemma, for which we use the notation of \ref{SecNni}.
 
 \begin{lemma}\label{Lemd2}${}$
 \begin{enumerate}
 \item For a diagram $d\in N_{n-1,i-1}\subset e_n^\ast \overline{N}e_i^\ast$, we have $x_nd-dx_i^0\in e_n^\ast CB_{\plus}e_i^\ast$.
 \item For a diagram $d\in N_{n-1,i+1}\subset e_n^\ast \overline{N}e_i^\ast$, we have
 $$(x_n-1)d-(d^{(1)}\otimes I)\circ (d^{(2)}x_{i+1}^0\otimes I)\circ a\;\in\; N_{n-1,i-1}\mS_{i}.$$
 \end{enumerate}
 \end{lemma}
 \begin{proof}
If $d\in N_{n-1,i-1}$ has a cup ending in dots $l$ and $j$, then $l,j<n$ and $(j,n)d=-\overline{(l,n)}d$. Hence
 $$x_n\,d\;=\; \left(\sum_{k\in D} (k,n) +\overline{(k,n)}\right)d,$$
 where~$D$ is the subset of~$\{1,2,\ldots,n-1\}$ corresponding to dots on the upper line which have a propagating line. Clearly $\overline{(k,n)}d$ is a diagram with one cap, so $\overline{(k,n)}d\in e_n^\ast CB_{\plus}e_i^\ast$. Part~(1) then follows from
 $\sum_{k\in D} (k,n)d=dx_i^0.$

Take $d\in N_{n-1,i+1}$ and $l_0$ such that the $l_0$th dot from the left on the upper line of $d$ is connected to the right-most (the $n$th) dot by a cap. Take $1\le l<n$ arbitrarily with~$l\not=l_0$. We have $(l,n)d\in N_{n-1,i-1}\mS_{i}$ unless $l$ is on a cap. On the other hand, if~$l$ is connected to some~$j$ by a cap, we have $(l,n)d+\overline{(j,n)}d=0$. Furthermore, we have $\overline{(l,n)}d=(l,l_0)d$ if~$l$ is not the end point of a cap. Finally, we have $(l_0,n)d=d$ and $\overline{(l_0,n)}d=0$. Hence we obtain
$$x_n d-d-\sum_{j\in D}(j,l_0)d\;\in\; N_{n-1,i-1}\mS_{i},$$
where~$D$ corresponds again to all dots which are not on the cups. 
By definition, $(j,l_0)(d^{(1)}\otimes I)=(d^{(1)}\otimes I)(j',l_0')$, where~$k'$ for $k\in D\cup\{l_0\}$ denotes the position, starting from the left, of the dot corresponding to~$k$ on the upper line, when ignoring the dots not in~$D\cup\{l_0\}$. 
Part (2) then follows.
 \end{proof}
 
We recall some basic facts concerning the Jucys-Murphy elements $\{x_j^0\}$ for the symmetric group.
 \begin{lemma}\label{JMSi}Fix $\lambda\vdash i$.
 \begin{enumerate}
 \item For any~$v\in W^0(\lambda)$, we have $\sum_{j=2}^{i} x_j^0v\;=\; |\resi(\lambda)| v.$
 \item For any~$\mu\in \RR(\lambda)$ and $v\in W^0(\lambda)^\mu$, we have
 $$x^0_i v-\resi(b)v\;\in \;W^0(\lambda)^{\widehat\mu},\qquad\mbox{where}\,\;\lambda=\mu\cup b.$$
 \end{enumerate}
 \end{lemma}
 \begin{proof}
 The element $\sum_{j=2}^{i} x_j^0\in\mZ\mS_i$ is central in $\mk\mS_i$ by \cite[Corollary~3.27]{Mathas} and thus acts as a constant on (the simple) Specht modules in characteristic zero. This remains true in arbitrary characteristic as this claim clearly does not depend on the field. The constant through which it acts then follows from \cite[Theorem~3.32]{Mathas}.
   Part (2) is a weaker version of a special case of \cite[Theorem~3.32]{Mathas}.
 \end{proof}

 We can now extend Lemma~\ref{JMSi}(2) to periplectic Brauer algebras.
 \begin{lemma}\label{LemFamJM}
Consider~$\lambda\in \mathbf{L}_{A_n}$ and $\mu\in\left(\RR(\lambda)\sqcup \AAAA(\lambda)\right)\cap \mathbf{L}_{A_{n-1}}$. For any~$v\in W(\lambda)^{\mu}$, we have
 $$x_n v -cv\in W(\lambda)^{\widehat\mu},\;\,\mbox{ with $c\in\mk$ given as}$$
 \begin{enumerate}
\item $c=\resi(b)$, if~$\lambda=\mu\cup b$;
\item $c=1+\resi(b)$, if~$\mu=\lambda\cup b$.
\end{enumerate}
 \end{lemma}
  \begin{proof}
Set $i:=|\lambda|$. For $\mu$ as in part~(1), the space $W_n(\lambda)^\mu$ is spanned by
 elements $d\otimes u$, with~$u\in W^0(\lambda)^{\mu}$ and $d\in N_{n-1,i-1}$. By Lemma~\ref{Lemd2}(1), we have
 $$x_n(d\otimes u)\;=\; d\otimes x_i^0u,$$
By Lemma~\ref{JMSi}(2), this is equal to~$\resi(b)d\otimes u$, up to elements $d\otimes u'$ with~$u'\in W^0(\lambda)^{\hat\mu}$, proving part~(1).

 For part~(2) we start by considering arbitrary elements in~$\oplus_{\nu\in \AAAA(\lambda)}W(\lambda)^{\nu}$. These are spanned by~$d\otimes u$, with~$d\in N_{n-1,i+1}$ and $u\in W^0(\lambda)$. By Lemma~\ref{Lemd2}(2) we have
 $$(x_n-1)(d\otimes u)-\left( (d^{(1)}\otimes I)\circ (d^{(2)}x^0_{i+1}\otimes I)\circ a\right) \otimes u\;\;\in\;\bigoplus_{\kappa\in\RR(\lambda)}W(\lambda)^{\kappa}.$$
Set $T^0:=\sum_{j=2}^{i+1}x_j^0$.
Using Lemma~\ref{JMSi}(1) and the fact that~$T^0$ is central in~$\mk\mS_{i+1}$, we have
 $$(d^{(2)}x^0_{i+1}\otimes I)\circ a \otimes u\;=\; (T^0d^{(2)}\otimes I)\circ a\otimes u - |\resi(\lambda)|(d^{(2)}\otimes I)\circ a \otimes u.$$
So far we have thus proved that
$$(x_n-1+|\resi(\lambda)|)(d\otimes u)-\left( (d^{(1)}\otimes I)\circ (T^0\otimes I)\circ (d^{(2)}\otimes I)\circ a\right) \otimes u\;\in\;\bigoplus_{\kappa\in\RR(\lambda)}W(\lambda)^{\kappa}.$$
By our definition~\eqref{BigDO} of the dominance order on partitions, the space in the right-hand side is contained in~$W(\lambda)^{\widehat\nu}$ for all~$\nu\in\AAAA(\lambda)$. By equations~\eqref{isoIW}, \eqref{eqCompl} and Lemma~\ref{JMSi}(1) we then find
$$\left(x_n-1+|\resi(\lambda)|-|\resi(\nu)|\right)v\;\in\; W(\lambda)^{\widehat\nu},$$
for all $v\in W(\lambda)^{\nu}$,
 %Now $W(\lambda)^\nu$, with~$\nu\in\AAAA(\lambda)$ is realised as in~\eqref{eqCompl}, by the span of~$v:=(d^{(1)}\otimes I)\otimes z$ with~$d^{(1)}\in e_{n-1}^\ast \overline{N}e_{i-1}^\ast$ and $z\in IW^0(\lambda)^{\nu}$. By the previous paragraph we have
 %$$(x_n-1+|\resi(\lambda)|)(d^{(1)}+I)\otimes z\,-\, (d^{(1)}+I) \otimes T_0 z\;\in\; W(\lambda)^{\widehat\nu}.$$
 %Now, by Lemma~\ref{JMSi}(1), we have $T_0z-|\resi(\nu)|z\in IW^0(\lambda)^{\widehat{\nu}}$.
%We thus find
%$$\left(x_n-1+|\resi(\lambda)|-|\resi(\nu)|\right)(d^{(1)}+I)\otimes z\;\in\; W(\lambda)^{\widehat\nu},$$
which concludes the proof.
  \end{proof}

 \begin{proof}[Proof of Theorem~\ref{ThmFamJM}]
 For $n=1$ the statement is trivial. Now we assume it is true for $n-1$. For $l<n$, the action of~$x_l$ on~$W(\lambda)$ is determined by~$\res_{A_{n-1}} W(\lambda)$. For $v_{\mathfrak{t}}$ with~$\mu:=\mathfrak{t}^{(n-1)}$, its image in~$(W_n(\lambda)^\mu\oplus W_n(\lambda)^{\hat{\mu}}) /W_n(\lambda)^{\hat{\mu}}$ is given by~$v_{\mathfrak{t'}}$. By the induction hypothesis, we have
 $$x_l v_{\mathfrak{t'}}-c_{\mathfrak{t}'}(l)v_{\mathfrak{t'}}\;\in\; \bigoplus_{\mathfrak{u}\in\St_{n-1}(\mu)\,|\, \mathfrak{u}\rhd \mathfrak{t}'}\mk v_{\mathfrak{u}}.$$
 As~$c_{\mathfrak{t}'}(l)=c_{\mathfrak{t}}(l)$, this implies that
 $$x_l v_{\mathfrak{t}}-c_{\mathfrak{t}}(l)v_{\mathfrak{t}}\;\in\; \bigoplus_{\mathfrak{s}\in\St_{n}(\lambda)\,|\, \mathfrak{s}^{(n-1)}=\mu\,,\, \mathfrak{s}'\rhd \mathfrak{t}'}\mk v_{\mathfrak{s}}\;\oplus \; W_n(\lambda)^{\hat{\mu}}\;\subset\; \bigoplus_{\mathfrak{s}\in\St_{n}(\lambda)\,|\, \mathfrak{s}\rhd \mathfrak{t}}\mk v_{\mathfrak{s}}.$$
The last inclusion follows from the definition of~$\unlhd$ on~$\St_n(\lambda)$ and equation~\eqref{FPMB}.

The case $l=n$ follows from Lemma~\ref{LemFamJM}.
 \end{proof}

 \begin{cor}\label{CorResi}
For $\lambda\in \mathbf{L}_A$ and $\mu\in \Lambda_A$, if
 $[W_n(\lambda):L_n(\mu)]\not=0$, there exist $\mathfrak{t}\in\St_n(\lambda)$ and $\mathfrak{s}\in\St_n(\mu)$ with~$c_\mathfrak{t}=c_{\mathfrak{s}}$.
 \end{cor}
 \begin{proof}
 The proof follows \cite[\S 2]{Mathas}.
 Consider~$W(\lambda)$ as a module over~$\Gamma_n$ from Lemma~\ref{LemJM}. All simple finite dimensional modules over~$\Gamma_n$ are one-dimensional. Since $L(\mu)$ is the top of~$W(\mu)$, the property $[W(\lambda):L(\mu)]\not=0$ implies in particular that there is a simple $\Gamma_n$-module which appears as a subquotient both in~$W(\mu)$ and $W(\lambda)$.
  Theorem~\ref{ThmFamJM} thus completes the argument.
 \end{proof}

 \begin{cor}
 \label{CorResi2}
Assume $\charr(\mk)\not\in[2,n]$. For $\lambda\in \mathbf{L}_{A_n}$ and $\mu\vdash n$, if
 $$[W_n(\lambda):L_n(\mu)]\not=0,$$
 then~$\lambda\subset \mu$ and the boxes in~$\mu\backslash\lambda$ can be paired in a way that the contents of each pair differ by one.
 \end{cor} 
 \begin{proof}
 The condition~$\lambda\subset\mu$ is immediate by the LR rule, since
 $$\res_{\mk\mS_n}W_n(\lambda)\;\cong\;\res_{\mk\mS_n} e_n^\ast \Delta_{C_n}(\lambda)\;\cong\;\ind^{\mS_n}\left(L^0(\lambda)\boxtimes K\right),$$
 with~$K$ the $\mk\mS_{n-i}$-module $\Hom_{\cA}(0,n-i)$, for $i=|\lambda|$.

 By definition \eqref{defct}, any~$c_{\mathfrak{t}}$ for $\mathfrak{t}\in\St_n(\lambda)$ consists of the residues of $\lambda$ (excluding that of the box in position (0,0)), together with pairs of elements in $\mk$ which differ by one. Similarly $c_{\mathfrak{s}}$, for $\mathfrak{s}\in \St_n(\mu)=\St(\mu)$, consists only of the residues of $\mu$. By Corollary~\ref{CorResi}, the residues in~$\mu$ which do not appear in~$\lambda$ must thus pair up into pairs of elements which differ by one.
 
 Furthermore, as the difference between the largest and smallest content of $\mu$ is strictly lower than~$n$, the condition~$\charr(\mk)\not\in[2,n]$ implies that when two residues differ by one, the corresponding contents must also differ by one. This concludes the proof.
 \end{proof}

 \subsection{Some commutation relations}
 In order to investigate which elements of~$\Gamma_n$ from Lemma~\ref{LemJM} belong to the centre of~$A_n$, we calculate some relations with the generators.
 
  \begin{lemma}\label{Lemexx}
 For $1\le k<n~$, we have
 \begin{enumerate}
 \item $\varepsilon_k(x_{k}-x_{k+1})=\varepsilon_k=-(x_{k}-x_{k+1})\varepsilon_k,$
 \item $s_kx_ks_k =x_{k+1}-s_k-\varepsilon_k,$
 \item $s_k(x_{k}-x_{k+1})s_k=-2s_k-(x_{k}-x_{k+1})$,
 \end{enumerate}
 \end{lemma}
 \begin{proof}
 We have $\varepsilon_k(i,k)=\varepsilon_k\overline{(i,k+1)}$ and $\varepsilon_k(i,k+1)=\varepsilon_k\overline{(i,k)}$ for all $i<k$. Furthermore, $\varepsilon_k s_k=-\varepsilon_k$ and $\varepsilon_k\varepsilon_k=0$, yielding the first equation in part~(1), with the second following similarly.

For $i<k$, we find $s_k(i,k)s_k=(i,k+1)$ and $s_k\overline{(i,k)}s_k=\overline{(i,k+1)}$.
Part (2) then follows immediately and part~(3) is immediate consequence of part~(2).
 \end{proof}
% These relations imply the following corollary.
 
\begin{cor}\label{Corx2}
 For $1\le k<n~$, we have 
 \begin{enumerate}
 \item $s_k(x_{k}-x_{k+1})^2=(x_{k}-x_{k+1})^2s_k$,
 \item $\epsilon_k(x_{k}-x_{k+1})^2=\epsilon_k=(x_{k}-x_{k+1})^2\epsilon_k$,
 \item $s_k(x_{k}+x_{k+1})s_k=(x_{k}+x_{k+1}) -2 \varepsilon_k$,
\item $s_k(x_kx_{k+1}) = (x_kx_{k+1}) s_k +x_k \varepsilon_k+\varepsilon_k x_k.$
 \end{enumerate}
 \end{cor}

 \begin{lemma}\label{Lemkl}
 For $1\le k< n$ and $l\not\in \{k,k+1\}$, we have $\varepsilon_k x_{l}=x_l\varepsilon_k,$ and $s_k x_{l}=x_ls_k$.
 \end{lemma}
 \begin{proof}
 For $l>k+1$, we have $s_k,\varepsilon_k\in A_{l-1}\subset A_n$, so the equations follow from Lemma~\ref{LemJM}. For $l<k$, it follows immediately that~$x_l$, being an element of~$ A_l$, commutes with~$s_k$ and $\varepsilon_k$, concluding the proof.
 \end{proof}
 
 %\begin{lemma}\label{Lem12}${}$
 %\begin{enumerate}
 %\item We have $\varepsilon_1 x_3=0=x_3\varepsilon_1$. 
 %\item We have $s_1x_2s_1=x_2-2\varepsilon_1$.
 %\end{enumerate}
 %\end{lemma}
 
 %Let $I$ be the ideal in~$A_n$ spanned by all diagrams which contain at least one cup (or cap).
 %\begin{lemma}
 %For $1\le k\le n$, we have
 %\begin{enumerate}
 %\item $s_k(x_k+x_{k+1}) + I\;=\; (x_k+x_{k+1})s_k+I$
  %\item $s_k(x_kx_{k+1}) + I\;=\; (x_kx_{k+1})s_k+I$
 %\end{enumerate}
 %\end{lemma}
 %\begin{proof}
 %The image of~$x_k$ in~$A_n/I\cong\mk\mS_n$ is precisely the corresponding Jucys-Murphy element for the symmetric group. The result is thus precisely \cite[Proposition~3.26(iv)]{Mathas}.  \end{proof}

\begin{cor}
Assume $\charr(\mk)\not=2$ and let $f$ be an arbitrary element of~$\Gamma_n$, we have
$$\varepsilon_k \,f\,\varepsilon_k=0,\qquad\;\mbox{ for all $\;1\le k< n$.}$$
\end{cor}
 \begin{proof}
 By Lemma~\ref{Lemkl}, it suffices to prove that the claim is true for $f$ an element in the subalgebra of~$A_n$ generated by~$x_{k}$ and $x_{k+1}$. Such an element is a linear combination of elements
 $$f_{ab}:=(x_{k}-x_{k+1})^a(x_k+x_{k+1})^b\quad \mbox{ for }\; a,b\in\mN.$$
 From Lemma~\ref{Lemexx}(1), we find $\varepsilon_k f_{ab}=\varepsilon_k f_{0b}$. To deal with~$f_{0b}$, we proceed by induction on~$b$. We thus assume $\varepsilon_k f_{0j}\varepsilon_k=0$ for $j<b$. Using Corollary~\ref{Corx2}(3), we calculate
 $$\varepsilon_k (x_k+x_{k+1})^b \varepsilon_k\,=\, \varepsilon_k (x_k+x_{k+1})^b s_k\varepsilon_k\,=\,\varepsilon_ks_k (x_k+x_{k+1}-2\varepsilon_k)^b \varepsilon_k\,=\,-\varepsilon_k (x_k+x_{k+1}-2\varepsilon_k)^b\varepsilon_k.$$ 
 The induction hypothesis can be used to show that the right-hand side is equal to~$-\varepsilon_k (x_k+x_{k+1})^b\varepsilon_k$, which concludes the proof.
 \end{proof}
 
 The properties in Corollary~\ref{Corx2} motivate the introduction of the following element,
 \begin{equation}\label{eqThetaT}
 \Theta:=\prod_{2\le i\not=j\le n}\left((x_i-x_j)-1\right)=\prod_{2\le i<j\le n}\left(1-(x_i-x_j)^2\right)\;\in\; \Gamma_n\,\subset\,A_n,
 \end{equation}
assuming that~$n>2$. For $n=2$, we just set $\Theta=0$.

\subsection{The centre}
 
 \subsubsection{}
 For the Hecke and Brauer algebras, many polynomials in the Jucys-Murphy elements are central see~\cite[Corollary~3.27]{Mathas} and~\cite[Corollary~2.4]{Naz}. For $A_n$ we find that the natural sufficient condition for elements of~$\Gamma_n$ to be central in~$A_n$ is far more restrictive. This is logical, as the centre of~$A_n$ is expected to be very small, as a consequence of the trivial centre of the universal enveloping algebra of the periplectic superalgebra, see \cite{Go}.
Let $\mk[\underline{x}]^{\mS_{n\minus 1}}\subset\Gamma_n$ denote the symmetric polynomials in~$\{x_2,x_3,\ldots,x_n\}$. In particular, $\Theta\in \mk[\underline{x}]^{\mS_{n\minus 1}}$.
 \begin{thm}\label{ThmGZ}
 The subalgebra
 $$\mk1 \;\oplus\; \Theta\, \mk[\underline{x}]^{\mS_{n\minus 1}},$$
with~$\Theta$ introduced in~\eqref{eqThetaT}, belongs to the centre of~$A_n$.
\end{thm} 
In general, the central elements $\Theta\, \mk[\underline{x}]^{\mS_{n\minus 1}}$ will belong to the Jacobson radical of~$A_n$, and might well be zero for $n>3$.
\begin{prop}\label{PropTheta0}
We have $\Theta L_{A_n}(\lambda)=0$, unless $n=3$, $\charr(\mk)\not=3$ and $\lambda=(2,1)$.
\end{prop}

Now we start the proofs of the theorem and proposition. Recall the definition of~$I$ in~\ref{sseqI}. The following proposition implies Theorem~\ref{ThmGZ}.

 \begin{prop}\label{PropThm}
 Let $f=f(x_2,x_3,\ldots,x_n)$ be a symmetric polynomial in~$n-1$ indeterminates evaluated in~$\{x_2,x_3,\ldots,x_n\}$. Then 
$f\,\Theta\,\in \,A_n$
belongs to the centre of~$A_n$. More precisely, we have
$$w\, f\Theta\;=\; f\Theta \, w\;\,\mbox{ for }\;w\in \mS_n\qquad\mbox{and}\qquad a\, f\Theta\;=\;0\;=\; f\Theta\, a\;\,\mbox{ for }\;a\in I.$$
 \end{prop}
  \begin{proof}
 Set $g:= f\Theta$. First we prove that~$\varepsilon_kg=0=g\varepsilon_k$. For $2\le k<n$, $g$ is of the form
 \begin{equation}\label{eqgh}g\;=\; \left((x_{k}-x_{k+1})^2 -1\right)h\left(x_{k},x_{k+1}\right)\end{equation}
 with~$h$ a symmetric polynomial in two variables, with values in the algebra generated by~$\{x_l\,|\, l\not\in\{k,k+1\}\}.$ The claim thus follows from Corollary~\ref{Corx2}(2). Using $x_2^2=1$, we can also write $g$ as
\begin{equation}\label{eqghh}((x_{2}-x_{3})^2 -1)h'(x_2)=(x_3^2-2x_3x_2)h'(x_2),\end{equation}
 for some polynomial $h'$ in one variable with values in the algebra generated by~$\{x_l\,|\, l>2\}$. That $g$ is annihilated by left and right multiplication with~$\varepsilon_1$ then follows from equation~\eqref{eq13}. 
 
 The above proves that~$ag=0=ga$ for all $a\in I$. Now we prove that~$s_k g=gs_k$ for all $1\le k<n$. For $2\le k$, we consider again the expression in~\eqref{eqgh}. The first factor commutes with~$s_k$ by Corollary~\ref{Corx2}(1). Lemma~\ref{Corx2}(3) and (4) imply that~$s_k$ commutes with any symmetric polynomial in~$x_k,x_{k+1}$, up to terms containing $\varepsilon_k$. This and Lemma~\ref{Lemkl} imply that the second factor in~\eqref{eqgh} commutes with~$s_k$ up to terms which cancel the first factor. Hence, we find indeed $s_k g=gs_k$. For $k=1$, we consider again the expression in~\eqref{eqghh}. By Lemma~\ref{Lemkl}, $s_1$ commutes with all factors in all terms except with~$x_2$. By Lemma~\ref{Lemexx}(2), we have $s_1x_2=x_2s_1+2\varepsilon_1$. Hence,~$s_1$ commutes with~$x_2$ up to a term which cancels $x_3$ by equation~\eqref{eq13}, so $s_1g=gs_1$.
 \end{proof}

\begin{proof}[Proof of Proposition~\ref{PropTheta0}]
By Proposition~\ref{PropThm}, $\Theta I=0$. We can thus consider the action on~$A_n/I\cong \mk\mS_n$, on which $\Theta$ acts as 
$$\Theta^0:=\prod_{2\le i<j\le n}\left(1-(x^0_i-x^0_j)^2\right).$$
The results thus follow immediately from \cite[Corollary~3.7]{Mathas} and \cite[Theorem~3.32]{Mathas}.\end{proof}

In sharp contrast to the symmetric group and Brauer algebra, there is no linear term in~$\Gamma_n$ contained in the centre of~$A_n$.  This will be conceptually explained in~\ref{Cas}.
 \begin{lemma}
If $\charr(\mk)=2$, the element
 $\sum_{i=1}^n x_i\,\in A_n$
 is central. When~$\charr(\mk)\not=2$, the only linear combinations of the Jucys-Murphy elements which are central in~$A_n$ are zero.
 \end{lemma}

 %%%%%%%%%%%%%%%%%%%%%%%%%%%%%%%%%%%%%%%%%%%%%%%%%%%%%%%%%%
 
 \section{On composition multiplicities and blocks}\label{SecDM}
 In this section we determine the blocks of the periplectic Brauer algebra~$A_n$ over fields with characteristic $0$ or higher than~$n$. The result is very different from the corresponding one for Brauer algebras in~\cite[Corollary~6.7]{blocks}. As an extra result we obtain several decomposition multiplicities. The latter result will be completed in~\cite{PB2}.
 
{\em For the entire section, we assume $\charr(\mk)\not\in[2,n]$.}
 \subsection{The blocks of~$A_n$}
 The main result will be stated in terms of~$2$-cores.
 \subsubsection{}\label{2core}The $2$-core of a partition, see \cite[\S~5.3]{Mathas}, is the partition obtained by iteratively removing rim $2$-hooks (2-ribbons) from its Young diagram, until no more can be removed. Rim 2-hooks are just two adjacent boxes 
$${ {\tiny\begin{ytableau}{}&     \end{ytableau}}}\qquad\mbox{or}\qquad {\tiny\begin{ytableau}
 {}\\
 {}     
\end{ytableau}}\;,$$
such that both boxes lie on the lower-right edge of the Young diagram of the partition. The possible 2-cores are given by~$\partial^{i}:=(i,i-1,\ldots,1)\,\in\Par_{\frac{1}{2}i(i+1)}$, for $i\in\mN$, so 
$$\partial^{0}:=\varnothing,\quad\partial^{1}:=\yng(1),\quad\partial^{2}:=\yng(2,1),\quad\partial^{3}:=\yng(3,2,1),\ldots.$$

\begin{thm}\label{ThmBlock}
Assume $\charr(\mk)\not\in[2,n]$. 
\begin{enumerate}
\item For $\lambda,\mu\in \Lambda_A$, the simple $A_n$-modules $L_A(\lambda)$ and $L_A(\mu)$ belong to the same block in~$A_n${\rm -mod} if and only if~$\lambda$ and $\mu$ have the same $2$-core.
\item For $\lambda,\mu\in \Lambda_C$, the simple $C_n$-modules $L_C(\lambda)$ and $L_C(\mu)$ belong to the same block in~$C_n${\rm -mod} if and only if~$\lambda$ and $\mu$ have the same $2$-core.
\end{enumerate}
\end{thm}
 The rest of this section is devoted to the proof.
 
 \subsection{Decomposition multiplicities for $C$}
\subsubsection{} For the standardly based algebra~$A_n$, we study the multiplicities
\begin{equation}\label{eqWDeltaL}[W_A(\lambda):L_A(\mu)]=[\Delta_C(\lambda):L_C(\mu)]\quad\mbox{
 for }\;\lambda\in\mathbf{L}_A=\Lambda_C\;\mbox{ and }\;\mu\in\Lambda_A,\end{equation} 
where the equality follows from \ref{Maschke} and equations~\eqref{eqWA} and~\eqref{eqCorLambda0}.
 We will thus focus on~$C_n$, which allows to work with quasi-hereditary algebras, and introduce the short-hand notation
 $$\Delta_n(\lambda)=\Delta_{C_n}(\lambda)\qquad\mbox{and}\qquad L_n(\lambda)=L_{C_n}(\lambda),$$
 for standard and simple modules over~$C_n$.
 %The following lemma reduces the quest for multiplicities to those in the maximal degree for the $\mZ$-grading on~$C_n$. 

 \begin{lemma}\label{LemRed1}
Consider $\lambda,\mu\in \Lambda_{C_n}=\mathbf{L}_{C_n}$.
\begin{enumerate}
\item If $i=|\mu|$, we have
 $$[\Delta_n(\lambda):L_n(\mu)]=\begin{cases} [\Delta_i(\lambda):L_i(\mu)]&\mbox{if }\; |\lambda|\le i\\
 0&\mbox{if }\; |\lambda|> i.\end{cases}$$
 \item If $[\Delta_n(\lambda):L_n(\mu)]\not=0$, then~$\lambda\subset\mu$ such that the boxes in~$\mu\backslash\lambda$ can be paired up in a way that the contents of each pair differ by one.
 \end{enumerate}
 \end{lemma}
 \begin{proof}
 The case $|\lambda|>i$ of part (1) is immediate by Theorem~\ref{ThmQH}, as then~$\lambda< \mu$. Assume therefore that~$|\lambda|\le i$.
 We have an exact Schur functor 
 $$C_n\mbox{-mod}\;\to \; C_i\mbox{-mod},$$
 corresponding to the idempotent $e^\ast:=\sum_{k\in\JJJ(i)}e_k^\ast$, since by construction~$C_i\cong e^\ast C_ne^\ast$. By equation~\eqref{defDover}, we have
 $e^\ast \Delta_n(\nu)\cong \Delta_i(\nu)$
 for each $\nu\in \Lambda_{C_i}$. As~$\Delta_n(\nu)$ has simple top $L_n(\nu)$ and $\Delta_i(\nu)$ has simple top $L_i(\nu)$, we furthermore find 
 $e^\ast L_n(\nu)\cong L_i(\nu).$
 This concludes the proof of part~(1).
 
 Part (2) follows immediately from part~(1) and Corollary~\ref{CorResi2}.
 \end{proof}
 
 Now we determine some properties of the $\mk\mS_n$-module structure of the extremal degree in~$\Delta_n(\lambda)$, which will be applied throughout the rest of the paper.
 
  \begin{lemma}\label{rednj}
 For $\lambda\vdash n-2$, we have
 $$\res_{\mk \mS_n}e_n^\ast\Delta_n(\lambda)\;\cong\;\bigoplus_{\mu\in P} L^0(\mu),$$
 where~$P$ is the set of $\nu\vdash n$ with~$\lambda\subset \nu$, such that the two boxes added to the Young diagram of~$\lambda$ to create $\nu$ are not in the same column.
 \end{lemma}
 \begin{proof}
We have
$$\res_{\mk \mS_n}e_n^\ast\Delta_n(\lambda)\;\cong\; \ind^{\mk \mS_n} (L^0(\lambda)\boxtimes L^0((2)))   \;\cong\;\bigoplus_{\mu\vdash n}L^0(\mu)^{\oplus c^\mu_{\lambda,(2)}},$$
so the claim then follows from the LR rule.
 \end{proof}

 Although we will only use the following lemma in the subsequent example, we note that it can be proved
by a direct computation as in the proof of Lemma~\ref{Lemd2}.
\begin{lemma}
For $n=2k$, if~$[\res_{\mk\mS_{n}}e_{n}^\ast \Delta_{n}(\varnothing):L^0(\lambda)]\not=0$, then~$|\resi(\lambda)|=k$.
\end{lemma}
%\begin{proof}
 %We use Lemma~\ref{JMSi}(1).
%The $\mk\mS_{2k}$-module $\Hom_{\cA}(0,2k)$ is generated by any diagram in that space. As~$\sum_{j=2}^{2k} x_j^0 $ is central in~$\mk\mS_{2k}$, by \cite[Corollary~3.27]{Mathas}, it suffices to prove $\sum_{j=2}^{2k} x_j^0 d=kd,$ for one diagram $d$. We take the diagram $a_0$ from~\ref{Defab}. For any~$1\le i< j \le k$, we have
%$$(2i,2j-1)a_0=-(2i-1,2j)a_0\qquad\mbox{and}\qquad (2i-1,2j-1)a_0=-(2i,2j)a_0.$$
%As we also have $(2l-1,2l)a_0=a_0$ for any~$1\le l\le k$, this implies that
%$$(x^0_{2l-1} +x^0_{2l})a_0=a_0,$$
%which concludes the proof.
%\end{proof}

There is only one partition~$\lambda\vdash4$ for which $|\resi(\lambda)|=2$, if~$\charr(\mk)\not\in\{2,3\}$, which is $\lambda=(3,1)$. 
This leads to the following example.
\begin{ex}\label{ExLem4}
For $n=4$, we have
$e_4^\ast\Delta_4(\varnothing)\cong L^0(3,1)$. For $n=5$, we have $$\quad e_5^\ast\Delta_5(1)\cong L^0(4,1)\oplus L^0(3,2)\oplus L^0(3,1,1) .$$
\end{ex}
The decomposition of~$e_5^\ast\Delta_5(1)$ follows immediately from $e_4^\ast\Delta_4(\varnothing)$ and the LR rule.

  \begin{prop}\label{PropMult1}
 Assume $\lambda\vdash i$ and $i+2\in\JJJ(n)$. We have
 $$[\Delta_n(\lambda):L_n(\mu)]=1$$
 for any~$\mu\vdash i+2$ such that its Young diagram is obtained from that of~$\lambda$ by adding two boxes in the same row.
 \end{prop}
 \begin{proof}
By Lemma~\ref{LemRed1}(1) we can assume $n=i+2$. By Lemma~\ref{rednj}, we have 
$$[\res_{\mk \mS_{n}}e_n^\ast\Delta_n(\lambda):L^0(\mu)]=1.$$
It thus remains to prove that the subspace $L^0(\mu)$ is annihilated by left multiplication with any diagram containing a cap and hence constitutes a $C_n$-submodule. As~$L^0(\mu)$ forms a $\mk\mS_n$-submodule it actually suffices to prove that it is annihilated by~$d\in e_{n-2}^\ast C_ne_n^\ast$, defined as
\begin{equation*}\label{defBB}
\begin{tikzpicture}[scale=0.6,thick,>=angle 90]
\begin{scope}[xshift=8cm]
\node at (0,0.5) {$d:=$};
\draw (.7,0) -- +(0,1);
\draw (1.3,0) -- +(0,1);
\draw [dotted] (1.6,.5) -- +(1,0);
\draw (2.9,0) -- +(0,1);
\draw (3.5,0) to [out=90,in=-180] +(.3,.3) to [out=0,in=90] +(.3,-.3);
\end{scope}
\end{tikzpicture}
\end{equation*}

This left multiplication actually yields a $\mk\mS_{n-2}\times \mS_1\times \mS_1$-module morphism from $L^0(\mu)$ to~$L^0(\lambda)$. Furthermore, $d=-d s_{n-1}$, which implies that left multiplication with~$d$ annihilates any element of~$L^0(\mu)$ which is $\mS_1^{\times n-2}\times \mS_2$-invariant. Hence, left multiplication with~$d$ will be zero on~$L^0(\mu)$ if
$$[\res_{\mk\mS_{n-2}\times \mS_2} L^0(\mu): L^0(\lambda)\boxtimes L^0(1,1)]=0.$$
For the proposed choices of~$\mu$ this is satisfied, as the LR rule implies $c_{\lambda,(1,1)}^\mu=0$.
 \end{proof}
 
In the following lemma and corollary, we take the convention that~$L_n(\lambda)=0$ when~$\lambda$ can not be sensibly interpreted as a partition.
\begin{lemma}\label{Corn2}
We have short exact sequences
$$0\,\to\, L_n(n)\oplus L_n(n-2,2)\,\to\, \Delta_n(n-2)\,\to\, L_n(n-2)\,\to\, 0,$$
$$0\,\to\,  L_n(3,1^{n-3})\,\to\, \Delta_n(1^{n-2})\,\to\, L_n(1^{n-2})\,\to\, 0.$$
\end{lemma}
\begin{proof}
By Lemma~\ref{rednj}, we have
$$e_n^\ast \Delta_n(n-2)\;\cong\; L^0(n)\oplus L^0(n-1,1)\oplus L^0(n-2,2)\quad\, \mbox{ and}$$
$$e_n^\ast \Delta_n(n-2,1^{n-2})\;\cong\; L^0(3,1^{n-3})\oplus L^0(2,1^{n-2}).$$
By Proposition~\ref{PropMult1}, $L^0(n)\oplus L^0(n-2,2)$ forms a $C_n$-submodule of~$\Delta_n(n-2)$, while $L^0(3,1^{n-3})$ forms a $C_n$-submodule of~$\Delta_n(1^{n-2})$. In both cases, the given subspace in~$e_n^\ast \Delta_n(\lambda)$ which complements this submodule is simple as a $\mk\mS_n$-module. When $n=2$ it is actually zero, when $n>2$, we have $e_n^\ast L(\lambda)\not=0$ for $\lambda$ equal to $(n-2)$ or $1^{n-2}$, by equation~\eqref{eqCorLambda0}. This shows that the remaining subspace in~$e_n^\ast \Delta_n(\lambda)$ is not in the radical of $\Delta_n(\lambda)$, which concludes the proof.
\end{proof}

\begin{cor}\label{Lemn2}
If~$j+2\in\JJJ(n)$, we have
$$[\Delta_n(j):L_n(j+2)]=1=[\Delta_n(j):L_n(j,2)]=[\Delta_n(1^j):L_n(3,1^{j-1})],$$
and all other multiplicities for $L_n(\lambda)$ with~$\lambda\vdash j+2$ vanish.
\end{cor}
\begin{proof}
This follows immediately from Lemmata~\ref{Corn2} and~\ref{LemRed1}(1).
\end{proof}

\subsection{Proof of Theorem~\ref{ThmBlock}}
First we observe that, by definition and the fact that all 2-cores are their own transpose, the $2$-cores of~$\lambda$ and $\lambda^t $ are identical. We will use this property freely.

\begin{prop}\label{PropBlock}
If $\lambda,\mu\in\Lambda_C$ are partitions with the same 2-core, the simple modules $L_n(\lambda)$ and $L_n(\mu)$ belong to the same block in~$C_n${\rm -mod}.
\end{prop}
\begin{proof}
Assume that~$\kappa$ is obtained from a partition~$\nu$ by removing a rim two-hook, we claim that
\begin{equation}\label{eqrim2}[P_n(\nu):L_n(\kappa)]\not=0\qquad\mbox{or}\qquad [P_n(\kappa):L_n(\nu)]\not=0.\end{equation}
From this it follows that any~$L_n(\nu)$ is in the same block as the simple module for its $2$-core, proving the proposition.

Now we prove equation~\eqref{eqrim2}. Assume first that~$\nu$ is obtained from $\kappa$ by adding two boxes in the same row. Proposition~\ref{PropMult1} implies that~$[\Delta_n(\kappa):L_n(\nu)]\not=0$, so in particular $[P_n(\kappa):L_n(\nu)]\not=0$. Now assume that~$\nu$ is obtained from $\kappa$ by adding two boxes in the same column. Then we have $[\Delta_n(\kappa^t):L_n(\nu^t)]\not=0$, so by equation~\eqref{eqBGGC} we have $(P_n(\nu):\Delta_n(\kappa))\not=0$, so in particular $[P_n(\nu):L_n(\kappa)]\not=0$. This concludes the proof.
\end{proof}

Theorem~\ref{ThmBlock}(1) now follows from the following two lemmata. Theorem~\ref{ThmBlock}(2) follows similarly from Proposition~\ref{PropBlock} and the analogue of Lemma~\ref{LemBlock2}.
\begin{lemma}\label{LemBlock1}
If $\lambda,\mu\in\Lambda_A$ have the same $2$-core, $L(\lambda)$ and $L(\mu)$ belong to the same block in~$A_n${\rm-mod}.
\end{lemma}
\begin{proof}
If $\lambda$ and $\mu$ have the same $2$-core, which is not $\partial^0=\varnothing$, the result follows immediately from equations~\eqref{eqrim2} and~\eqref{eqCorLambda0}.

If $\lambda$ has $2$-core $\varnothing$, the above reasoning shows that~$L_A(\lambda)$ will be in the block of either $L_A(2)$ or $L_A(1,1)$. 
The claim then follows frome 
$$[P_A(1,1):L_A(2)]\ge (P_A(1,1):W_A(\varnothing))[W_A(\varnothing):L_A(2)]=1,$$ see Proposition~\ref{PropMult1}.
 \end{proof}
 
 \begin{lemma}\label{LemBlock2}
If $L(\lambda)$ and $L(\mu)$ belong to the same block in~$A_n${\rm-mod}, then one of the following equivalent conditions must hold:
\begin{enumerate}
\item the number of even minus the number of odd contents of~$\lambda$ equals that of~$\mu$;
\item $\lambda$ and $\mu$ have the same $2$-core.
\end{enumerate}
\end{lemma}
\begin{proof}
Assume first that~$[W(\nu):L(\kappa)]\not=0$ for some $\nu,\kappa\in\Lambda_A$. By Lemma~\ref{LemRed1}(2) and equation~\eqref{eqWDeltaL}, the number of even minus the number of odd contents of~$\nu$ equals that of~$\kappa$. Note also that this difference between the number of odd and even contents is the same for $\nu$ and $\nu^t$.

If $L(\lambda)$ and $L(\mu)$ are in the same block we must be able to construct a sequence $\lambda=\nu_0,\nu_1,\ldots,\nu_k=\mu$ of partitions such that~$\nu_i$ corresponds to a simple subquotient in the projective cover corresponding to~$\nu_{i-1}$. Theorem~\ref{ThmSB}(4) (with Remark~\ref{RemSB}) and the previous paragraph imply that condition~(1) must be satisfied.

Now we prove the equivalence of (1) and (2). Denote by~$\gamma(\kappa)$ the number of boxes in the Young diagram of $\kappa$ with even content minus the number of boxes with odd content. By construction, $\gamma(\nu)=\gamma(\lambda)$ when $\nu$ is obtained from $\lambda$ by removing rim 2-hooks. We also have
$$\gamma(\partial^{2i-1})=i \qquad\mbox{and}\qquad \gamma(\partial^{2i})=-i,\qquad\mbox{for $i\in\mN$}.$$
In conclusion, $\gamma$ of a partition is the same as $\gamma$ of its 2-core, and~$\gamma$ is different for each $2$-core. Hence, conditions (1) and~(2) are equivalent. 
\end{proof}

\begin{rem}
It follows from Lemma~\ref{lemeta}(2) that any block in~$A_n$-mod which does not correspond to the 2-core $\varnothing$ is equivalent to one in~$C_n$-mod. Hence, $A_n$ decomposes into a number of blocks, such that at most one is not quasi-hereditary.
\end{rem}

%%%%%%%%%%%%%%%%%%%%%%%%%%%%%%%%%%%%%%%%%%%%%%%%%%%%%%%%%%%%%%%% 
 
 \section{Connections with the periplectic Lie superalgebra}\label{Secpm}
 Fix $m\in\mN$. We set $V=\mk^{m|m}$, the $m|m$-dimensional superspace ($\mF_2=\mZ/2$-graded space) and we choose an {\em odd} supersymmetric non-degenerate bilinear form $\langle\cdot,\cdot\rangle$ on~$V$. In this section, we will {\em always assume that~$\charr(\mk)=0$}, even though several statements are independent of characteristic. For any superspace $W$ and $v\in W_i$, with $i\in\mF_2$, we write $|v|=i$.

  \subsection{Actions of $\mathfrak{gl}(V)$ and $\pe(V)$ on tensor space}\label{acAn}
 
 \subsubsection{}\label{SecAnV} The space $\End_{\mk}(V)$, with multiplication given by the super commutator is the Lie superalgebra~$\mathfrak{gl}(V)$. We define the Killing form $(\cdot|\cdot)$ on~$\mathfrak{gl}(V)$ as
 $$(X|Y):=\Str_{V} (XY),\qquad\mbox{for }\;X,Y\in\mathfrak{gl}(V),$$
 the super trace of~$XY$ interpreted as an operator on~$V$. For a (homogeneous) basis $\{X_a\}$ of~$\mathfrak{gl}(V)$, we denote by~$\{X_a^\dagger\}$ the basis satisfying $(X_a|X_b^\dagger)=\delta_{ab}$.
 Note that~$(X^\dagger_a)^\dagger =(-1)^{|X_a|}X_a$.

% For $1\le i,j\le 2m$, we denote the elementary matrix with~$1$ on row $i$ and column~$j$ by~$E_{ij}$. We use the supersymmetric, non-degenerate, invariant bilinear form $(\cdot|\cdot)$ on~$\mathfrak{gl}(m|m)$ given by
 %$$(E_{ij}| E_{k,l})=(-1)^{[j]}\delta_{jk}\delta_{i,l}.$$  
 For $X\in\mathfrak{gl}(V)$ and $0\le l< n$, we let $\Id^{\otimes l}\otimes X\otimes\Id^{\otimes n-l-1}$ act from the {\em left} on~$V^{\otimes n}$ by
$$v_1\otimes v_2\otimes\cdots \otimes v_n\;\mapsto\; (-1)^{|X|(|v_1|+\cdots+|v_{l}|)} v_1\otimes \cdots \otimes v_{l}\otimes Xv_{l+1}\otimes v_{l+2}\otimes\cdots \otimes v_n.$$
For $1\le k\le n$, we have the corresponding tensor product action, leading to a morphism of associative algebras
$$\Delta_k:\; U(\mathfrak{gl}(V)) \to \End_{\mk}(V^{\otimes n});\;\quad \Delta_k(X)=\sum_{0\le l< k}\Id^{\otimes l}\otimes X\otimes \Id^{\otimes (n-l-1)},\quad\mbox{for }X\in \mathfrak{gl}(V).$$

\subsubsection{} 
We define an involutive anti-algebra automorphism~$\theta$ of~$\mathfrak{gl}(V)$, which is determined by the property
 $$\langle Xv,w\rangle\;=\;- (-1)^{|v||X|}\langle v,\theta(X)w\rangle,$$
 for all $v,w\in V$, for a fixed $X\in\mathfrak{gl}(V)$.
The {\em periplectic Lie superalgebra} $\mathfrak{pe}(V)$ is the subalgebra of~$\mathfrak{gl}(V)$ of elements satisfying $\theta(X)=X$. It follows that~$\dim_{\mk}\mathfrak{pe}(V)=m^2|m^2$ and $\mathfrak{pe}(V)$ is a maximal degenerate subspace for $(\cdot|\cdot)$.

We consider~$V^{\otimes n}$ as a $\mathfrak{pe}(V)$-module, for the restriction of~$\Delta_n$. Denote by $\bT_V$ the category with objects $\{V^{\otimes i}\,|\, i\in\mN\}$, with convention $V^{\otimes 0}=\mk$, and as morphisms all $\mathfrak{pe}(V)$-linear morphisms of vector spaces. Note that we do not require morphisms to respect the $\mF_2$-grading.

\subsubsection{}We set
$$T=\sum_{j=1}^{2m}(-1)^{|u_j|} u_j\otimes u_j^\ast\; \in \;V\otimes V,$$ for $\{u_j\}$ some homogeneous basis of~$V$, with dual basis $u_j^\ast$ with respect to~$\langle\cdot,\cdot\rangle$. Note that left and right dual basis mean the same thing for an odd supersymmetric form, and $(u_j^\ast)^\ast=u_j$.

The following follows immediately from work of Deligne - Lehrer - Zhang and Kujawa - Tharp, and extends earlier work by Moon.

\begin{lemma}\cite{DLZ, Kujawa}\label{LemDLZ}
There exists a full contravariant functor
$F:\cA\to \bT_V$, which satisfies $F(i)=V^{\otimes i}$ for all $i\in\mN$ and $F(f\otimes g)=F(f)\otimes F(g)$, for all morphisms $f,g$. Furthermore, we have $F(\cup)=\langle\cdot,\cdot\rangle\in \Hom_{\mathfrak{pe}(V)}(V^{\otimes 2},\mk)$ and $F(\cap)\in \Hom_{\pe(V)}(\mk,V^{\otimes 2})$ is given by $1\mapsto T$. Finally, $F(X)\in \Hom_{\mathfrak{pe}(V)}(V^{\otimes 2},V^{\otimes 2})$ is given by $v\otimes w\mapsto (-1)^{|v||w|}w\otimes v$.
\end{lemma}
\begin{proof}
The functor $F$ is defined in \cite[Theorem~5.2.1]{Kujawa}. Note that {\it loc. cit.}
 the functor is considered as covariant because contravariant composition of morphisms in $\cA$ is used. By \cite[\S 5.3]{Kujawa}, the functor $F$ is full if and only if $F$ induces epimorphisms $\Hom_{\cA}(0,i)\to \Hom_{\mathfrak{pe}(V)}(V^{\otimes i},\mk)$, for all $i\in\mN$. The latter is precisely~\cite[Proposition~4.10]{DLZ}.
\end{proof}

In particular, $F$ restricts to surjective algebra morphisms
\begin{equation}\label{eqpi2}\pi_n:\,A_n\;\tto\;\End_{\pe(V)}(V^{\otimes n})^{\op}\quad\mbox{and}\quad
\widetilde{\pi}_n:\,C_n\;\tto\; \End_{\pe(V)}(\bigoplus_{i\in\JJJ(n)}V^{\otimes i})^{\op}.\end{equation}

\subsubsection{}\label{Isom}If $ n\le m$, then~$\pi_n$ in~\eqref{eqpi2} is even an isomorphism by \cite[Theorems~4.1 and~4.5]{Moon}. This has recently been extended to the case $n<\frac{1}{2}(m+1)(m+2)$ in \cite[Theorem~8.3.1]{TenIde}.

 The algebra~$A_n$ hence represents the full centraliser of the action of~$U(\mathfrak{pe}(m))$ on~$V^{\otimes n}$ under those conditions. Whether we have a double centraliser property is an interesting open question.
\begin{qu}\label{question}
For which $m,n\in\mN$ is the algebra morphism~$U(\mathfrak{pe}(m))\to \End_{A_n}(V^{\otimes n})$ surjective?
\end{qu}
For $n\in\{2,3\}$ and $m\ge n$, this morphism is a surjection, as can be checked directly from the explicit description of the $\mathfrak{pe}(m)$-module $V^{\otimes n}$ in~\cite[Section~6]{Moon}.

 \subsubsection{}\label{CommD}By construction, we have a commutative diagram 
\[
\xymatrix{
 *+\txt{$A_n$} \ar[rr]^{\pi_n}&& *+\txt{$\End_{\mathfrak{pe}(V)}(V^{\otimes n})^{\op}$} \\
*+\txt{$\mk\mS_n$}\ar[rr]^{\pi_n}\ar@{^{(}->}[u]&&*+\txt{$\End_{\mathfrak{gl}(V)}(V^{\otimes n})^{\op}$.}\ar@{^{(}->}[u]
}
\]

\subsection{The category $\cF_m$} Set $\pe(m):=\pe(V)$.
\subsubsection{} Let $\cF_m$ be the category of finite dimensional integrable modules over~$\mathfrak{pe}(m)$. Such modules are automatically $\mF_2$-gradable, see {\it e.g.}~\cite[\S 2.3]{Chen}. However, we consider all  $\mathfrak{pe}(m)$-linear morphisms, not only those preserving grading. It is well-known that this category is abelian and has enough projective and injective objects, where the latter two classes of modules coincide, see {\it e.g.}~\cite[\S 2]{Chen}.

\subsubsection{} We consider the standard triangular decomposition~$\mathfrak{pe}(m)=\fn^{\minus}\oplus\fh\oplus\fn^{\plus}$, of {\it e.g.} \cite[\S 2]{Vera}, which satisfies
$$\dim_{\mk}\fn^{\plus}=\frac{1}{2}m(m-1)|\frac{1}{2}m(m+1)\quad\mbox{and}\quad \dim_{\mk}\fn^{\minus}=\frac{1}{2}m(m-1)|\frac{1}{2}m(m-1).$$
The positive odd roots are given by~$\delta_i+\delta_j$, for $1\le i\le j\le m$, and the negative odd roots are $-\delta_i-\delta_j$, for $1\le i< j\le m$.

With respect to this system of positive roots, we denote the simple module with highest weight $\alpha\in\fh^\ast$ by~$\cL(\alpha)$. For instance, we have $V\cong \cL(\delta_1)$ and $\mk\cong\cL(0)$. The necessary and sufficient condition for $\cL(\alpha_1\delta_1+\cdots+\alpha_m\delta_m)$ to be finite dimensional is $\alpha_i-\alpha_j\in\mN$. In order to understand all finite dimensional $\cL(\alpha)$ it is sufficient to restrict to those satisfying $\alpha_1\in \zeta+\mZ$ for one fixed $\zeta\in[0,1[$. We choose $\zeta=0$ and hence $\alpha_j\in\mZ$. We call the corresponding highest weights integral dominant and denote the set by~$X^+\subset\fh^\ast$.
The simple objects of~$\cF_m$ are, up to isomorphism, given by~$\{\cL(\alpha)\,|\,\alpha\in X^+\}$. 

To a partition $\lambda\vdash d$ which satisfies $\lambda_{m+1}=0$, we associate an integral dominant weight
$$\underline\lambda\,:=\,\sum_{i=1}^m\lambda_i\delta_i\;\in X^+.$$

\subsection{The blocks of~$\mathfrak{pe}(m)$}
In \cite{Chen} some interesting partial results concerning the block decomposition of~$\cF_m$ were obtained, which we complete in the following theorem. We recall the equivalence relation $\sim$ on $X^+$ introduced in~\cite[Definition~5.1]{Chen}. This relation is transitively generated by
$$\begin{cases}\alpha\sim \alpha +2\delta_k, &\mbox{ for $1\le k\le m$, provided $\alpha+ 2\delta_k\in X^+$;}\\
\alpha\sim \alpha + (\delta_k+\delta_l),&\mbox{ for $1\le k\le l\le m$, provided $\alpha + (\delta_k+\delta_l)\in X^+$ and $\alpha_k=\alpha_l$.}
\end{cases}$$
Recall also the $2$-cores $\{\partial^{i}\,|\,i\in\mN\}$ from \ref{2core}.
\begin{thm}\label{ThmBlocksP}
The modules $\cL(\alpha)$ and $\cL(\beta)$ are in the same block of~$\cF_m$ if and only if $\alpha\sim\beta$. The block decomposition of~$\cF_m$ is thus given by~$$\cF_m\;=\;\bigoplus_{i=0}^{m} \cF_m(\underline{\partial}^i)$$
\end{thm}
The rest of this section is devoted to the proof.

\begin{lemma}\label{CorIdAnn}
For any~$\lambda\in\Lambda_{A_n}$, the corresponding primitive idempotent $e_\lambda\in A_n$ is not annihilated by~$\pi_n$ if $\lambda_{m+1}=0$.
Moreover, we have $[V^{\otimes n}e_\lambda:\cL(\underline\lambda)]\not=0$.
\end{lemma}
\begin{proof}
Assume first that~$\lambda\vdash n$ and let $e$ be a primitive idempotent in $\mk\mS_n$ corresponding to $L^0(\lambda)$. Under the inclusion~$\mk\mS_n\hookrightarrow A_n$, the idempotent $e$ does not necessarily remain primitive, but in its decomposition as a sum of orthogonal primitive idempotents in~$A_n$, $e_\lambda$ (up to conjugation) must appear, by Corollary~\ref{CorLAn}. Considering $A_n\tto \mk\mS_n$ as in \ref{sseqI} thus shows that~$e-e_\lambda\in I$.

Now assume that~$\pi_n(e_\lambda)=0$, which thus implies that~$\pi_n(e)\in \pi_n(I)$. By Lemma~\ref{LemDLZ}, every $\mathfrak{pe}(m)$-morphism of~$V^{\otimes n}$ in~$\pi_n(I)$ factors through $V^{\otimes n-2}$. Now \cite{Berele} implies that $V^{\otimes n}e$ is a simple $\mathfrak{gl}(m|m)$-module with highest weight $\underline{\lambda}$.  Since $V^{\otimes n-2}$ does not contain (non-zero) weight vectors of weight $\underline{\lambda}$, we obtain a contradiction. Hence $\pi_n(e_\lambda)\not=0$. Moreover, since the simple $\mathfrak{gl}(m|m)$-module wight highest weight $\underline{\lambda}$ restricts to a $\mathfrak{pe}(m)$-module which contains $\cL(\underline\lambda)$ as a constituent, we have~$[V^{\otimes n}e_\lambda:\cL(\underline\lambda)]\not=0$.

Now we consider the general case $\lambda\vdash i\in\JJJ^0(n)$. By Lemma~\ref{LemConnectP}, we can take $e_\lambda =a_i e_\lambda^{(i)} b_i$.
By Lemma~\ref{LemDLZ}, $F(a_i):V^{\otimes n}\to V^{\otimes i}$ is surjective and $F(b_i): V^{\otimes i}\to V^{\otimes n}$ injective. We can therefore reduce the case $i<n$ to the above paragraph. %This shows that~$\pi_i(e_\lambda^{(i)})$ is zero if and only if $\pi_n(e_\lambda)$ is zero. The fact that~$\pi_n(e_\lambda)\not=0$ thus follows the previous paragraph. This reasoning also shows that the property $[V^{\otimes n}e_\lambda:\cL(\underline\lambda)]\not=0$ extends.
\end{proof}

\begin{prop}\label{Propim}
Let $n$ be strictly bigger than~$\frac{1}{2}m(m+1)$. The algebra
$\End_{\mathfrak{pe}(m)}(\bigoplus_{j=0}^nV^{\otimes j})$
has at least $m+1$ blocks. Hence, we have a decomposition of $\mathfrak{pe}(m)$-modules
$$\bigoplus_{j=0}^nV^{\otimes j}=\bigoplus_{i=0}^mM_i^n,\qquad\mbox{with}\quad \Hom_{\pe(m)}(M_i^n,M_j^n)=0,\;\mbox{ if $i\not=j$.}$$
Furthermore, $[M^n_i:\cL(\underline{\partial}^{i})]\not=0$, for $0\le i\le m$.
\end{prop}
\begin{proof}
We clearly have
$$\End_{\mathfrak{pe}(m)}(\bigoplus_{j=0}^nV^{\otimes j})\;=\;\End_{\mathfrak{pe}(m)}(\bigoplus_{j\in\JJJ(n)}V^{\otimes j})\oplus \End_{\mathfrak{pe}(m)}(\bigoplus_{j\in\JJJ(n-1)}V^{\otimes j}),$$
as $\mathfrak{h}$ already separates between even and odd tensor powers.
We consider the surjective morphism
$$\widetilde{\pi}_n\oplus\widetilde{\pi}_{n-1}:\;C_n\oplus C_{n-1}\;\tto\; \End_{\mathfrak{pe}(m)}(\bigoplus_{j=0}^nV^{\otimes j}).$$

For $0\le i\le m$, the idempotents $e_{\partial^i}\in A_n\oplus A_{n-1}\subset C_n\oplus C_{n-1}$, 
are not in the kernel of $\widetilde{\pi}_n\oplus\widetilde{\pi}_{n-1}$, by Lemma~\ref{CorIdAnn}. Theorem~\ref{ThmBlock}(2) then implies that~$m+1$ blocks of~$C_n\oplus C_{n-1}$ are not annihilated by~$\widetilde{\pi}_n\oplus \widetilde{\pi}_{n-1}$.
The conclusion about the non-vanishing multiplicities for $\cL(\underline{\partial}^i)$ then follows from Lemma~\ref{CorIdAnn}.
\end{proof}

The following observation about the injective modules in~$\cF_m$ is a well-known general property of faithful representations of affine algebraic supergroup schemes. Alternatively it can be obtained from the corresponding $\mathfrak{gl}(V)$-property by exploiting restriction and induction functors. 
\begin{lemma}\label{LemITP}
Every injective envelope in~$\cF_m$ is a direct summand of some power~$V^{\otimes n}$.
\end{lemma}
%\begin{proof}
%We consider the exact restriction functor $\res^{\mathfrak{gl}(m|m)}_{\mathfrak{pe}(m)}$. By the PBW theorem, the induction from $U(\mathfrak{pe}(m))$ to~$U(\mathfrak{gl}(m|m))$ yields an exact functor. Hence, $\res^{\mathfrak{gl}(m|m)}_{\mathfrak{pe}(m)}$ maps injective modules to injective modules. It then follows from Frobenius reciprocity that any injective hull in~$\cF_m$ is a direct summand of an injective $\mathfrak{gl}(m|m)$-module, after restriction. By~\cite[Theorem~3.6]{BS} and \cite[\S3.1]{Vera}, any injective $\mathfrak{gl}(m|m)$-module is a direct summand of some $V^{\otimes r}\otimes (V^\ast)^{\otimes s}$, with~$V=\mk^{m|m}$ the natural $\mathfrak{gl}(m|m)$-module and $V^\ast$ its dual. As~$V\cong V^\ast$ for $\mathfrak{pe}(m)$, this concludes the proof.
%\end{proof}

%The following follows immediately from results of Deligne, Lehrer and Zhang, and Kujawa and Tharp.
%\begin{lemma}
%The morphism $\pi_n$ in equation~\eqref{eqpi2} is always surjective.
%\end{lemma}
%\begin{proof}
%The intertwiners of \cite[Section~5.3]{Kujawa} imply that the requested surjectivity is equivalent to surjectivity of
%$$\Hom_{\cA}(0,2r)\;\to\; \Hom_{\mathfrak{pe}(m)}(V^{\otimes 2r},\mk).$$
%The latter is precisely~\cite[Section~4.9]{DLZ}.
%\end{proof}
 We let $\cF_m(\alpha)$ denote the block containing $\cL(\alpha)$, for any~$\alpha\in X^+$.

\begin{prop}\label{Propm1}
$\cF_m$ contains at least $m+1$ blocks, given by~$\cF_m(\underline\partial^i)$, for $0\le i\le m$.
\end{prop}
\begin{proof}
Let $I_i$ denote the injective envelope of $\cL(\underline{\partial}^i)$.
Assume now that $I_j$ and $I_{j'}$ belong to the same block in~$\cF_m$, for $j\not=j'$.
There must be a finite collection of injective modules $\{K_1,K_2,\ldots,K_k\}$ in~$\cF_m$ such that there is always a morphism~$K_i\to K_{i+1}$ or $K_{i+1}\to K_i$ and such that there is some morphism between~$I_j$ and $K_1$, and between~$I_{j'}$ and $K_k$. By Lemma~\ref{LemITP}, there exists $n\in\mN$ (which we take to satisfy $n>\frac{1}{2}m(m+1)$) such that all of these injective modules are direct summands of
$\oplus_{l\le n}V^{\otimes l}$. We use the notation and results of Proposition~\ref{Propim}. By the existence of morphisms between the injective hulls, they must all be contained in~$M_{j_0}^{n}$ for one fixed $j_0$. However, we clearly have that $I_j$, resp. $I_{j'}$, is a direct summand of $M_j^n$, resp. $M_{j'}^n$. This implies $j=j_0=j'$, a contradiction.
\end{proof}

Theorem~\ref{ThmBlocksP} now follows from Proposition~\ref{Propm1} and the
following result of~\cite[\S5]{Chen} and \cite[Theorem~A.2]{Chen}.
\begin{lemma}[C.W. Chen]\label{LemChen}
If $\alpha\sim\beta$, for $\alpha,\beta\in X^+$, the modules $\cL(\alpha)$ and $\cL(\beta)$ belong to the same block in $\cF_m$. Hence, there are at most $m+1$ blocks in $\cF_m$.\end{lemma}

 \subsection{Alternative realisation of $A_n$ and the Jucys-Murphy elements}\label{JMpn}
 Using $\mathfrak{gl}(m|m)$ and $\mathfrak{pe}(m)$, we introduce some operators on~$ V^{\otimes n}$. 
 
 \subsubsection{} We choose a homogeneous basis $\{X_a\,|\, a=1,\ldots,m^2\}$ of~$\mathfrak{pe}(m)$. Recall $\theta$ and $\dagger$ from Section~\ref{acAn} and take the homogeneous basis 
 $$\{X_{m^2+b}:= X^\dagger_{b}\,|\, 1\le b\le m^2\}$$
 of the eigenspace of~$\mathfrak{gl}(m|m)$ for $\theta$ with eigenvalue $-1$. For clarity, we will use $i$ as an index to sum from $1$ to~$2m^2$ and $a$ or $b$ to sum from $1$ to~$m^2$, so $X_{a}$ and $X_{m^2+b}^{\dagger}$ will be elements in~$\mathfrak{pe}(m)$.

 \subsubsection{}\label{altes} For $1\le k<n$, we set
 $$\sigma_k:=\sum_{1\le i \le 2m^2}\,\Id^{\otimes k-1}\otimes X_i^\dagger\otimes X_i\otimes \Id^{\otimes n-k-1}\;\in\,\End_{\mk}(V^{\otimes n}).$$
 It is easily checked, for instance by considering a basis of~$V$, that~$\sigma_k=\pi_n(s_k)$. We also set
 $$c_k:=\sum_{1\le a \le 2m^2}\,\Id^{\otimes k-1}\otimes\left((-1)^{|X_a|} X_a\otimes X_a^\dagger-X_a^\dagger\otimes X_a\right)\otimes \Id^{\otimes n-k-1} \;\in\,\End_{\mk}(V^{\otimes n}).$$
 It follows from a direct computation that~$c_k=\pi_n(\varepsilon_k)$.
 
 We have thus found an alternative realisation of the algebra $\pi_n(A_n)$. When we take $m>>$ (so that $n<\frac{1}{2}(m+1)(m+2)$), this thus realises $A_n$.
 
\subsubsection{} Now, for $2\le k\le n$, we introduce
 $$\xi_k:=2\sum_{1\le a \le m^2}(-1)^{|X_a|}\,\Delta_{k-1}( X_a)\otimes X_a^\dagger\otimes \Id^{\otimes n-k}.$$
It follows from~\ref{altes} that~$\xi_k=\pi_n(x_k)$, so $\{\xi_k\,|\,2\le k\le n\}$ are the image of the Jucys-Murphy elements of~$A_n$.
 Consider the embedding of~$A_{k-1}$ in~$A_n$ of \ref{embed}. We have commuting subalgebras
 $$\pi_n(A_{k-1})=\pi_{k-1}(A_{k-1})\otimes \Id_V^{\otimes n-k+1}\quad\mbox{and}\quad \Delta_{k-1}(U(\mathfrak{pe}(m)))\otimes \End_{\mk}(V^{\otimes n-k+1})$$
 in $\End_{\mk}(V^{\otimes n})$. In particular, $\pi_n(A_{k-1})$ commutes with $\xi_k=\pi_n(x_k)$.
  For $\charr(\mk)=0$, we hence find an alternative proof of Lemma~\ref{LemJM} by taking $m>>$.

\subsubsection{}\label{Cas} The restriction of the Killing form from $\mathfrak{gl}(m|m)$ to~$\mathfrak{pe}(m)$ is zero, instead of non-degenerate. Consequently, the elements $X_a^\dagger$ do not belong to~$\mathfrak{pe}(m)$, meaning that~$\sum_{k=2}^n\xi_k$ does not correspond to~$\Delta_n(\cC_2)$, with~$\cC_2$ a Casimir operator, as would be the case for $\mathfrak{osp}(p|2q)$, see {\it e.g.}~\cite[\S2]{Naz}.

 %%%%%%%%%%%%%%%%%%%%%%%%%%%%%%%%%%%%%%%%%%%%%%%%%%%%%%%%%%%%%%%%%%%
 \section{Some examples}\label{SecEx}
We will determine algebra structures and decomposition multiplicities for $A_n$ with~$n\le 5$. %We do this without using Theorem~\ref{ThmBlock}.
 \subsection{The algebras $A_2$ and $C_2$}\label{SecA2}

 \begin{thm}\label{ThmA2}
If $\charr(\mk)\not=2$, the algebra~$A_2$ is the hereditary algebra given by the path algebra of the quiver~$\cQ_2$:
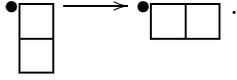
\begin{displaymath}
    \xymatrix{
         \bullet_{{\tiny\begin{ytableau}{}\\{}     \end{ytableau}}}\ar[r] &\bullet_{{\tiny\begin{ytableau}{}&     \end{ytableau}}}   }.
\end{displaymath}
\end{thm} 
\begin{proof}
The algebra~$A_2$ has a basis given by~$1$, $s=s_1$ and $\varepsilon=\varepsilon_1$. We have orthogonal primitive idempotents
$$e_1= \frac{1}{2}(1-s)\quad\mbox{and}\quad e_2=\frac{1}{2}(1+s),$$
for which we find $\varepsilon=e_2\varepsilon e_1$ and $1=e_1+e_2$. The identification with the labelling set $\Par_2$ of simple modules in Theorem~\ref{ThmSB}(1) follows easily, which
completes the proof.
\end{proof}

Similarly one proves the following theorem.
 \begin{thm}\label{ThmC2}
If $\charr(\mk)\not=2$, the algebra~$C_2$ is the hereditary algebra given by the path algebra of the quiver~$\overline{\cQ}_2$:
\begin{displaymath}
    \xymatrix{
         \bullet_{{\tiny\begin{ytableau}{}\\{}     \end{ytableau}}}\ar[r] &\bullet_{\varnothing }\ar[r] &\bullet_{{\tiny\begin{ytableau}{}&     \end{ytableau}}}   }.
\end{displaymath}
\end{thm} 
 
\subsubsection{}Continue assuming $\charr(\mk)\not=2$. By Theorem~\ref{ThmA2}, we find that~$A_2$ is quasi-hereditary for the two linear orders on its poset $\Lambda_A=\Par_2$. Each case gives a corresponding standardly based structure, see \ref{StBQH}, for which the cell modules are the standard modules and hence form a standard system. However, the standardly based structure of Theorem~\ref{ThmSB}(2) is given in terms of the set
 $L=\{\varnothing,(2),(1,1)\}=\Lambda_C$, totally ordered by~$\unlhd$. Theorem~\ref{ThmC2} allows to calculate the standard modules of~$C_2$ for this order, from which we find the corresponding cell modules of~$A_2$:
\begin{equation}\label{cellA2}W_A(\varnothing)\cong L_A(2)\cong W_A(2)\quad\mbox{and}\quad W_A(1,1)\cong L_A(1,1).\end{equation}

\begin{cor}\label{CorA2}
Set $p:=\charr(\mk)\ge 0$.
\begin{enumerate}
\item The cell modules of~$A_2$ from Theorem~\ref{ThmSB}(2) do not form a standard system.
\item The algebra~$A_2$ admits a cellular datum as in~\cite[Definition~1.1]{CellAlg} if and only if~$p=2$.
\item If $p\not=2$, the double centraliser property in Theorem~\ref{ThmDC} is not true for $A_2$ and $C_2$.
\end{enumerate}
\end{cor}
\begin{proof}
Assume $p\not=2$. By equation~\eqref{cellA2}, there are two cell modules which are isomorphic, so they do not form a standard system. Assume $p=2$, then the conclusion in part~(1) follows from \cite[Lemma~9.3.2(1)]{Borelic}, or directly from Remark~\ref{Rem2}. This proves part~(1). 

By \cite[(C1)]{CellAlg}, the dimension of a cellular algebra is the sum of squares of integers. As the dimension of~$A_2$ is $3$, the only possible such sum is $3=1^2+1^2+1^2$. However, \cite[(C2)]{CellAlg} then implies that~$A_2$ must admit an anti-involution which is the identity ({\it i.e.} $A_2$ must be commutative). As~$A_2$ is not commutative when~$p\not=2$, it does not admit a cellular datum. When~$p=2$, it is well-known that~$A_2\cong B_2(0)$ admits a cellular datum, see~\cite[Theorem~4.10]{CellAlg}. This proves part (2).

For part (3), equations~\eqref{eqinjiso} and~\eqref{enD0} show that $X=e_2^\ast C\cong A\oplus W(\varnothing)$, which is a projective $A$-module. Hence $\End_A(X)^{\op}$ is Morita equivalent to $A$, so different from $C$.
\end{proof}

\begin{rem}\label{Rem2}
For completeness, we mention that, when~$\charr(\mk)=2$, we have
$$A_2\,\cong\, B_2(0)\cong\, \mk[x,y]/(x^2,xy,y^2).$$
\end{rem}

\subsection{The algebra~$A_3$}

\begin{thm}\label{ThmA3}
If $\charr(\mk)\not\in\{2,3\}$, the algebras $A_3$ and $C_3$ are hereditary algebras, Morita equivalent to the path algebra of the quiver~$\cQ_3$:
\begin{displaymath}
    \xymatrix{
         \bullet_{{\tiny\begin{ytableau}{}\\{}\\{}     \end{ytableau}}}\ar[r] &\bullet_{{\tiny\begin{ytableau}{}     \end{ytableau}}}\ar[r]&\bullet_{{\tiny\begin{ytableau}{}&&     \end{ytableau}}} &&\bullet_{{\tiny\begin{ytableau}{}&\\{}     \end{ytableau}}}  }.
\end{displaymath}
\end{thm}
\begin{proof}
By Theorems \ref{ThmQH} and \ref{ThmMor}(1), $A_3$ is Morita equivalent to~$C_3$ and quasi-hereditary.
We work with~$C_3$. If $\alpha\not=(1)$, then~$\Delta(\alpha)$ is simple. Lemma~\ref{Corn2} implies that the unique proper submodule of
$\Delta(1)$ is $L(3).$
Using the reciprocity in Theorem~\ref{ThmQH} then implies
$$P(2,1)\cong L(2,1),\quad P(3)\cong L(3)\quad\mbox{and}\quad P(1)\cong \Delta(1),$$
as well as a short exact sequence
$$0\to \Delta(1)\to P(1,1,1)\to L(1,1,1)\to 0.$$
As~$(C_3,\le)$ is quasi-hereditary, there is no first extension between~$L(1,1,1)$ and $L(3)$. Hence the structure of all indecomposable projective modules is completely determined and corresponds to the path algebra. %As the quiver has no closed paths, the algebras are hereditary.
\end{proof}

\subsection{The algebras $A_4$ and $C_4$}

\begin{thm}\label{ThmC4}
If $\charr(\mk)\not\in\{2,3\}$, the algebra~$C_4$ is Morita equivalent to the path algebra of the quiver~$\overline{\cQ}_4$:
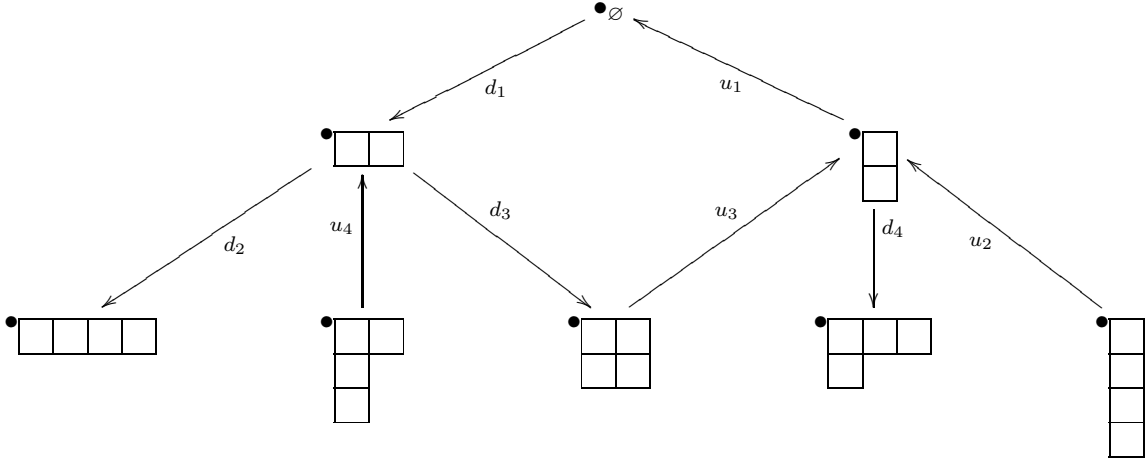
\begin{displaymath}
    \xymatrix{
 &&&&   \bullet_{{{\varnothing }}}\ar[lld]^{d_1} &\\
       &&  \bullet_{{\tiny\begin{ytableau}{}&     \end{ytableau}}}\ar[dll]^{d_2}\ar[drr]^{d_3} &&&&\bullet_{{\tiny\begin{ytableau}
 {}\\
 {}     
\end{ytableau}}}\ar[ull]^{u_1}\ar[d]^{d_4} \\
         \bullet_{{\tiny\begin{ytableau}
 {}&   &  &\end{ytableau}}} &&\bullet_{{\tiny\begin{ytableau}
 {}&    \\
 {}\\
 {}\end{ytableau}}}\ar[u]^{u_4}  && \bullet_{{\tiny\begin{ytableau}
 {}&     \\
 {}&\end{ytableau}}}\ar[rru]^{u_3}&&\bullet_{{\tiny\begin{ytableau}
 {}& &  \\
 {}\end{ytableau}}} &&\bullet_{{\tiny\begin{ytableau}
 {}   \\
 {}\\
 {}\\{}\end{ytableau}}}\ar[ull]^{u_2}
}
\end{displaymath}
with relations $d_2\circ d_1=0$, $d_3\circ d_1=0$, $u_1\circ u_2=0$, $u_3\circ d_3=0$ and $u_1\circ u_3=0$.
\end{thm}
 \begin{proof}
 By Corollary~\ref{Lemn2}, we have $[\Delta(\varnothing):L(2)]=1$ and $[\Delta(\varnothing):L(1,1)]=0$. By Example~\ref{ExLem4}, $e_4^\ast\Delta(\varnothing)$ is a simple $\mk\mS_4$-module and hence belongs to one simple $C_4$-subquotient. As~$e_4^\ast L(2)\not=0$, by equation~\eqref{eqCorLambda0}, we find $e_4^\ast\Delta(\varnothing)=e_4^\ast L(2)$, proving~$[\Delta(\varnothing):L(\lambda)]=0$, for any~$\lambda\vdash 4$.

By Lemma~\ref{Corn2}, the maximal submodule of~$\Delta(2)$ decomposes into~$L(4)$ and $L(2,2)$. Also by Lemma~\ref{Corn2}, the unique proper submodule of~$\Delta(1,1)$ is  $L(3,1)$. Using the reciprocity in Theorem~\ref{ThmQH} then determines the standard filtrations of the projective modules.
The requirement that~$(C_4,\le)$ be quasi-hereditary then allows to derive the explicit structure of the projective modules and these correspond to the projective modules of the proposed path algebra. \end{proof}

Theorem~\ref{ThmC4} and equation~\eqref{eqCorLambda0} then immediately lead to the following result.

\begin{thm}\label{ThmA4}
If $\charr(\mk)\not\in\{2,3\}$, the algebra~$A_4$ is Morita equivalent to the path algebra of the quiver~$\cQ_4$:
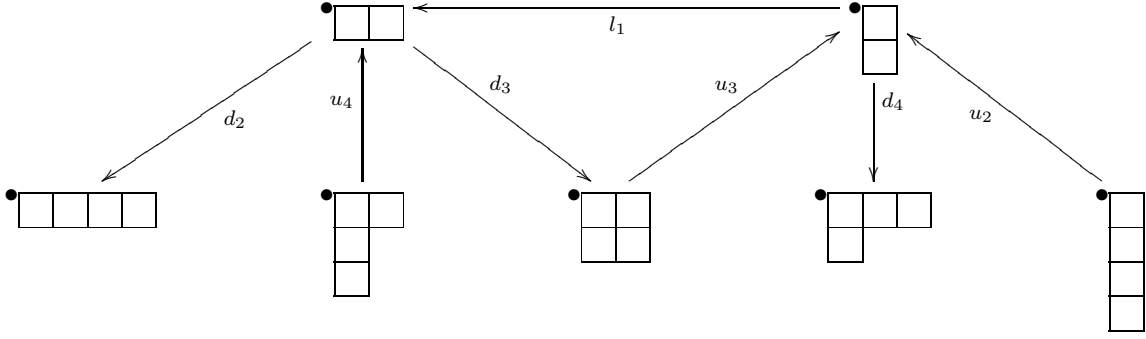
\begin{displaymath}
    \xymatrix{
       &&  \bullet_{{\tiny\begin{ytableau}{}&     \end{ytableau}}}\ar[dll]^{d_2}\ar[drr]^{d_3} &&&&\bullet_{{\tiny\begin{ytableau}
 {}\\
 {}     
\end{ytableau}}}\ar[llll]^{l_1}\ar[d]^{d_4} \\
         \bullet_{{\tiny\begin{ytableau}
 {}&   &  &\end{ytableau}}} &&\bullet_{{\tiny\begin{ytableau}
 {}&    \\
 {}\\
 {}\end{ytableau}}}\ar[u]^{u_4}  && \bullet_{{\tiny\begin{ytableau}
 {}&     \\
 {}&\end{ytableau}}}\ar[rru]^{u_3}&&\bullet_{{\tiny\begin{ytableau}
 {}& &  \\
 {}\end{ytableau}}} &&\bullet_{{\tiny\begin{ytableau}
 {}   \\
 {}\\
 {}\\{}\end{ytableau}}}\ar[ull]^{u_2}
}
\end{displaymath}
with relations $d_2\circ l_1=0$, $d_3\circ l_1=0$, $l_1\circ u_2=0$, $u_3\circ d_3=0$ and $l_1\circ u_3=0$.
\end{thm}

Part (1) in the following corollary shows that there is no partial order $\preceq$ for which $(A_4,\preceq)$ is quasi-hereditary.
\begin{cor}\label{CorA4}
Assume $\charr(\mk)\not\in\{2,3\}$.
\begin{enumerate}
\item The algebra~$A_4$ has infinite global dimension.
\item The cell modules of~$A_4$ from Theorem~\ref{ThmSB} do not form a standard system.
\end{enumerate}
\end{cor}
\begin{proof}
From Theorem~\ref{ThmA4}, it follows that we can construct an exact sequence
\begin{equation}\label{eqA4}0\to L(2)\to P(1,1)\to P(4)\oplus P(2,2)\to P(2)\to L(2)\to 0.\end{equation}
From this it easily follows that
$$\Ext^{3k}_{A_4}(L(2),L(2))\,\not=\,0,\quad\mbox{for all }\; k\in\mN,$$
proving part~(1).

From the structure of the standard modules for $C_4$ in the proof of Theorem~\ref{ThmC4}, it follows that~$W(\varnothing )\cong L(2)$ and $W(4)\cong L(4).$
The Gabriel quiver in Theorem~\ref{ThmA4} thus clearly shows that 
$$\Ext^1_{A_4}(W(\varnothing ),W(4))\not=0,$$
even though $(4)\unlhd \varnothing $, concluding the proof of part~(2).
\end{proof}

\subsection{Koszulity}
Let~$\cQ$ be a quiver and $J$ an ideal in the (free) path algebra~$\mk\cQ$ generated by elements $\beta\circ\alpha$ for arrows $\alpha$ and $\beta$ in~$\cQ$. The algebra~$\mk\cQ/J$ is then automatically quadratic in the sense of \cite[Definition~1.2.2]{BGS}, so in particular positively graded. All quotients of free path algebras constructed above are of this form. They are even Koszul, in the sense of \cite[Definition~1.2.1]{BGS}.

\begin{prop}
For $n<5$ and $\charr(\mk)\not\in\{2,3\}$, the algebras $A_n$ and $C_n$ are Morita equivalent to Koszul algebras. More precisely:
\begin{enumerate}
\item If $\charr(\mk)\not=2$, the algebras $A_2\cong \mk\cQ_2$ and $C_2\cong\mk\overline{\cQ}_2$ are Koszul. The algebra~$A_2$ is Koszul self-dual, but $C_2$ is not.
\item If $\charr(\mk)\not\in\{2,3\}$, the algebra~$\mk\cQ_3$ is a Koszul, but not Koszul self-dual.
\item Let $J$, resp. $\overline{J}$, be the ideals generated by the relations in Theorems~\ref{ThmA4}, resp.~\ref{ThmC4}. If $\charr(\mk)\not\in\{2,3\}$, the algebras $\mk\cQ_4/J$ and $\mk\overline{\cQ}_4/\overline{J}$ are Koszul, but not Koszul self-dual.
\end{enumerate}
\end{prop}
\begin{proof}
The algebra~$\mk\cQ_2$ is quadratic with 
$$(\mk\cQ_2)_{1}\otimes_{(\mk\cQ_2)_{0}}(\mk\cQ_2)_{1}=0.$$
Such an algebra its own quadratic dual, see~\cite[Definition~2.8.1]{BGS}. The algebra is clearly Koszul and the Koszul self-duality follows from~\cite[Theorem~2.10.1]{BGS}. The Koszulity of~$\mk\overline{\cQ}_2$ is also obvious. Since
$$(\mk\overline{\cQ}_2)_2=(\mk\overline{\cQ}_2)_1\otimes_{(\mk\overline{\cQ}_2)_0}(\mk\overline{\cQ}_2)_0\not=0$$ the algebra is not quadratic self-dual and hence not Koszul self-dual. This concludes part~(1). We have $\mk\cQ_3\cong \mk\overline{\cQ}_2\oplus \mk$, so part~(2) follows from part~(1).

The Koszulity in part~(3) follows from careful construction of the minimal projective resolutions of the simple modules. The algebra~$A_4$ has infinite global dimension, by Corollary~\ref{CorA4}(1). The Koszul dual of~$\cQ_4/J$ is therefore infinite dimensional, disproving Koszul self-duality. The projective resolution of~$L(\varnothing)$ shows that~$\Ext^4_A(L(\varnothing),L(\varnothing))$ does not vanish, while the graded length of~$C_4$ is only three. This prevents Koszul self-duality and thus concludes the proof of part~(3).
\end{proof}

\begin{rem}
It is easily checked that the Koszul dual algebra of~$\mk\cQ_3$ is actually isomorphic to the Ringel dual of~$\mk\cQ_3$, for the quasi-hereditary structure of Theorem~\ref{ThmQH}.
\end{rem}

\subsection{Dimensions} Recall that the global dimension of an algebra $A$, denoted by $\gd A\in\mN\cup\{\infty\}$, is the highest $i$ for which $\Ext^i_A(-,-)$ is non-trivial. The dominant dimension, $\dd A\in\mN\cup\{\infty\}$, is the maximal $i$ for which all $I^j$ with $0\le j<i$ are projective in the minimal injective coresolution $I^\bullet$ of the left regular representation. The injective dimension, $\id A\in\mN\cup\{\infty\}$, of the left regular representation is the maximal $i$ for which $I^i$ is non-zero. With these conventions we have $\dd A\minus 1\le\id A\le \gd A$. An algebra $A$ is Iwanaga-Gorenstein if both $\id A$ and $\id A^{\op}$ are finite.

\begin{prop} Assume that $\charr(\mk)\not\in\{2,3\}$.
\begin{enumerate}
\item We have $\gd A_n=\dd A_n= \id A_n=1$, for $n\in\{2,3\}$.
\item We have $\gd A_4=\infty$, $\dd A_4=0$ and $\id A_4=\infty$. In particular, $A_4$ is not Iwanaga-Gorenstein. 
\end{enumerate}
\end{prop}
\begin{proof}
The dimensions of $A_2$ and $A_3$ are easily calculated using Theorems~\ref{ThmA2} and~\ref{ThmA3}. Now we consider $n=4$. From Theorem~\ref{ThmA4} we can construct an exact sequence
$$0\to L(2)\to P(1,1)\oplus P(3,1)\to P(2,2)\oplus P(1,1,1,1)\to I(1,1)\to 0.$$
By equation~\eqref{eqA4}, this implies that the injective hull $I(1,1)$ has infinite projective dimension. Applying the anti-automorphism $\varphi_4$ of $A_4$ in Remark~\ref{3rem} shows that this is equivalent to $\id P(2)=\infty$. This implies $\gd A_4=\id A_4=\infty$.
Constructing the indecomposable injective modules of $A_4$ from Theorem~\ref{ThmA4} shows that there are no modules simultaneously projective and injective, which immediately implies $\dd A_4=0$.
\end{proof}

\subsection{The algebras $A_5$ and $C_5$}
 We determine all composition multiplicities for $A_5$.
\begin{prop}\label{PropA5}
Assume $\charr(\mk)\not\in\{2,3,5\}$. The standard modules for $A_5$ or $C_5$ have the following Jordan-H\"older multiplicities: 
$$[\Delta(1)]\;=\;[L(1)]+[L(3)]+[L(3,2)],\;\quad [\Delta(2,1)]\;=\; [L(2,1)]+[L(4,1)],$$
$$[\Delta(3)]\;=\; [L(3)]+[L(5)]+[L(3,2)],\;\quad [\Delta(1,1,1)]\;=\; [L(1,1,1)]+[L(3,1,1)],$$
$$\mbox{and}\;\, [\Delta(\lambda)]=[L(\lambda)],\quad\mbox{ for all }\; \lambda\vdash 5.$$
\end{prop}
\begin{proof}
The statements for $\Delta(3)$ and $\Delta(1,1,1)$ follow from Lemma~\ref{Corn2}. Now we consider~$\Delta(2,1)$. By Lemma~\ref{rednj}, we have
$$e_5^\ast \Delta(2,1)\;\cong\; L^0(4,1)\oplus L^0(3,2)\oplus L^0(3,1,1)\oplus L^0(2,2,1).$$
By Proposition \ref{PropMult1}, $L^0(4,1)$ constitutes a $C_5$-submodule. The other spaces are prohibited to form subquotients by Theorem~\ref{ThmBlock}.

By Corollary~\ref{Lemn2}, we have $[\Delta(1):L(3)]=1$ and $[\Delta(1):L(3,\lambda)]=0$ if~$|\lambda|=3$ and $\lambda\not=(3)$. By Example~\ref{ExLem4} we have
$$e_5^\ast\Delta(1)\;\cong\; L^0(4,1)\oplus L^0(3,2)\oplus L^0(3,1,1).$$
From the proof of Lemma~\ref{Corn2}, we find that~$L^0(4,1)$ belongs to~$e_5^\ast L(3)$. Using the LR rule it follows that there is a two-dimensional space of~$\mS_3\times \mS_2$-invariants in~$e_5^\ast\Delta(1)$, contained in~$L^0(4,1)\oplus L^0(3,2)$. One easily constructs such an invariant $v$ which is annihilated by all diagrams containing a cap. Then~$\mS_5 v$ forms a $C_5$-submodule. As~$L^0(4,1)$ does not constitute a submodule, $ L^0(3,2)$ constitutes a submodule. Equation~\eqref{eqCorLambda0} implies that~$e_5^\ast L(1)\not=0$, so the remaining space $L^0(3,1,1)$ belongs to~$L(1)$. This implies that~$\Delta(1)$ does not contain any further simple $C_5$-subquotients.
\end{proof}

%The category $\cF_m$ decomposes as
%$$\cF_m\;=\;\cF_m(\lambda^{(m)})\,\oplus \,\cF_{m}(\lambda^{(m-1)}) \, \oplus\, \cF_m^1\,\oplus\, \cF_m^2,$$
%where~$\cF_m^l$ contains all $\cL(\lambda^{(i)})$ with~$i\in S_l$ for $l\in\{1,2\}$. Furthermore
%$$\cF_m^l=\bigoplus_{i\in T_l}\cF_m(\lambda^{(i)}),$$
%for certain subsets $T_1\subseteq S_1$ and $T_2\subseteq S_2$.

\subsection*{Acknowledgement}
The research was supported by Australian Research Council Discover-Project Grant DP140103239 and an FWO postdoctoral grant. The author thanks Andrew Mathas and Ruibin Zhang for useful discussions.

\begin{flushleft}
	K. Coulembier\qquad \url{kevin.coulembier@sydney.edu.au}
	
	School of Mathematics and Statistics, University of Sydney, NSW 2006, Australia

\end{flushleft}


\begin{thebibliography}
	{CDM}
	
	\bibitem[B+]{gang} M. Balagovic, Z. Daugherty, I. Entova Aizenbud, I. Halacheva, J. Hennig, M.S. Im, G. Letzter, E. Norton, V. Serganova, and C. Stroppel: Translation functors and decomposition numbers for the periplectic Lie superalgebra~$\mathfrak{p}(n)$. arXiv:1610.08470.

	
	\bibitem[BGS]{BGS}
A.~Beilinson, V.~Ginzburg, W.~Soergel:
Koszul duality patterns in representation theory. 
J. Amer. Math. Soc. {\bf 9} (1996), no. 2, 473--527. 
	
	\bibitem[BSR]{Ram}
	G.~Benkart, C.~L. Shader, A.~Ram:
Tensor product representations for orthosymplectic Lie superalgebras.
J. Pure Appl. Algebra {\bf130} (1998), no. 1, 1--48. 

\bibitem[BR]{Berele}
A.~Berele, A.~Regev: Hook Young diagrams with applications to combinatorics and to representations of Lie superalgebras. Adv. in Math. {\bf64} (1987), no. 2, 118--175.
	
	
	\bibitem[BS]{BS}
	J.~Brundan, C.~Stroppel:
	Highest weight categories arising from Khovanov's diagram algebra IV: the general linear supergroup. J. Eur. Math. Soc. (JEMS) {\bf14} (2012), no. 2, 373--419.
	
	\bibitem[Ch]{Chen}
	C.W.~Chen:
Finite-dimensional representations of periplectic Lie superalgebras.
J. Algebra {\bf443} (2015), 99--125.
	

	\bibitem[CPS]{CPS}
	E. Cline, B. Parshall, L. Scott:
Finite-dimensional algebras and highest weight categories. 
J. Reine Angew. Math. {\bf391} (1988), 85--99. 


\bibitem[Co]{TenIde}
K.~Coulembier:
Tensor ideals, Deligne categories and invariant theory.
arXiv:1712.06248.

\bibitem[CE]{PB2}
K.~Coulembier, M.~Ehrig: The periplectic Brauer algebra II: Decomposition multiplicities.  J. Comb. Algebra 2 (2018), no. 1, 19--46.


\bibitem[CZ]{Borelic}
K.~Coulembier, R.B.~Zhang:
Borelic pairs for stratified algebras. arXiv:1607.01867.


\bibitem[CDM]{blocks}
A.~Cox, M.~De~Visscher, P.~Martin:
The blocks of the Brauer algebra in characteristic zero. 
Represent. Theory {\bf13} (2009), 272--308.

	
	%\bibitem[De]{Deligne}
	%P. Deligne:
%La cat\'egorie des repr\'esentations du groupe sym\'etrique $S_t$, lorsque t n'est pas un entier naturel.  Algebraic groups and homogeneous spaces, 209--273, 
%Tata Inst. Fund. Res. Stud. Math., Tata Inst. Fund. Res., Mumbai, 2007. 

\bibitem[DLZ]{DLZ}
P.~Deligne, G.I.~Lehrer, R.B.~Zhang:
	The first fundamental theorem of invariant theory for the orthosymplectic super group. Adv. Math. 327 (2018), 4--24.

\bibitem[DR]{DR}
V.~Dlab, C.M.~Ringel:
The module theoretical approach to quasi-hereditary algebras. Representations of algebras and related topics (Kyoto, 1990), 200--224, 
London Math. Soc. Lecture Note Ser., {\bf 168}, Cambridge Univ. Press, Cambridge, 1992. 

\bibitem[DuR]{JieDu}
J.~Du, H.~Rui:
Based algebras and standard bases for quasi-hereditary algebras.
Trans. Amer. Math. Soc. {\bf350} (1998), no. 8, 3207--3235.

\bibitem[ES]{ES}
M.~Ehrig, C.~Stroppel:
Schur-Weyl duality for the Brauer algebra and the ortho-symplectic Lie superalgebra.  Math. Z. 284 (2016), no. 1-2, 595--613.



\bibitem[En]{Enyang}
J.~Enyang:
Specht modules and semisimplicity criteria for Brauer and Birman-Murakami-Wenzl algebras. 
J. Algebraic Combin. {\bf26} (2007), no. 3, 291--341.



\bibitem[Go]{Go}
M.~Gorelik:
The center of a simple P-type Lie superalgebra.
J. Algebra {\bf246} (2001), no. 1, 414--428.
	
	\bibitem[GL]{CellAlg}
	J.J. Graham, G.I. Lehrer:
Cellular algebras.
Invent. Math. {\bf123} (1996), no. 1, 1--34. 


\bibitem[HP]{Paget}
R.~Hartmann, R.~Paget:
Young modules and filtration multiplicities for Brauer algebras.
Math. Z. {\bf254} (2006), no. 2, 333--357. 


\bibitem[Ja]{James}
G.~James:
The representation theory of the symmetric groups. Springer LNM {\bf 682}, 1978.



\bibitem[KN]{Nakano}
A. Kleshchev, D. Nakano:
On comparing the cohomology of general linear and symmetric groups.
Pacific J. Math. {\bf 201} (2001), no. 2, 339--355.


\bibitem[KT]{Kujawa}
J.R.~Kujawa, B.C.~Tharp:
The marked Brauer category. 
J. Lond. Math. Soc. (2) 95 (2017), no. 2, 393--413.
	
\bibitem[KX]{Koenig}
S.~K\"onig, C.C.~Xi:
A characteristic free approach to Brauer algebras.
Trans. Amer. Math. Soc. {\bf353} (2001), no. 4, 1489--1505.


\bibitem[LR]{LeducRam}
R.~Leduc, A.~Ram:
A ribbon Hopf algebra approach to the irreducible representations of centralizer algebras: the Brauer, Birman-Wenzl, and type A Iwahori-Hecke algebras.
Adv. Math. {\bf125} (1997), no. 1, 1--94.


	\bibitem[LZ1]{BrCat} G.I. Lehrer, R.B. Zhang:
	The Brauer category and invariant theory.
	J. Eur. Math. Soc. {\bf 17} (2015), 2311--2351.
	
	\bibitem[LZ2]{FFT} G.I. Lehrer, R.B. Zhang:
	The first fundamental theorem of invariant theory for the orthosymplectic supergroup.
	Comm. Math. Phys. 349 (2017), no. 2, 661--702.
	
	\bibitem[Ma1]{Mathas}
	A.~Mathas:
Iwahori-Hecke algebras and Schur algebras of the symmetric group.
University Lecture Series, {\bf15}. American Mathematical Society, Providence, RI, 1999.

\bibitem[Ma2]{MathasCrelle}
A.~Mathas:
Seminormal forms and Gram determinants for cellular algebras. 
With an appendix by Marcos Soriano. 
J. Reine Angew. Math. {\bf619} (2008), 141--173. 

\bibitem[Mo]{Moon}
D.~Moon:
Tensor product representations of the Lie superalgebra~$\mathfrak{p}(n)$ and their centralizers. 
Comm. Algebra {\bf31} (2003), no. 5, 2095--2140.

\bibitem[Na]{Naz}
M.~Nazarov:
Young's orthogonal form for Brauer's centralizer algebras. 
J. Algebra {\bf 182} (1996), no. 3, 664--693. 


\bibitem[Ro]{Rou}
R.~Rouquier:
$q$-Schur algebras and complex reflection groups.
Mosc. Math. J. {\bf8} (2008), no. 1, 119--158, 184. 


\bibitem[Se1]{Verapn}
V.~Serganova:
On representations of the Lie superalgebra p(n).
J. Algebra {\bf258} (2002), no. 2, 615--630.

\bibitem[Se2]{Vera}
V.~Serganova:
Finite dimensional representations of algebraic supergroups.
Proceedings of the International Congress of Mathematicians, Seoul, 2014. 


\bibitem[YL]{Yang}
G. Yang, Y.~ Li:
Standardly based algebras and 0-Hecke algebras.
J. Algebra Appl. {\bf14} (2015), no. 10, 1550141, 13 pp.

\end{thebibliography}
\end{document}